%% file: interp_dR.tex
\documentclass[11pt]{amsart}
\usepackage{amsmath,amssymb}

\usepackage{inputenc}
\usepackage{bm}
\usepackage{xcolor}
\usepackage{newtxmath}
\usepackage{mathrsfs}
\usepackage{verbatim}
\usepackage{float}
\usepackage{setspace}
\usepackage{amsthm}
\usepackage{geometry}
\usepackage{color}
\usepackage{cite}
\usepackage{tikz}
\usetikzlibrary{backgrounds}
\usetikzlibrary{arrows}
\usetikzlibrary{shapes,shapes.geometric,shapes.misc}
\usepackage{hyperref}
\hypersetup{
	colorlinks=true,
	linkcolor=cyan!50!blue,
	filecolor=blue,      
	urlcolor=red,
	citecolor=green,
}
\usepackage{cleveref}
\usepackage[all]{xy}
\allowdisplaybreaks[4]
\usepackage{geometry}
 \newtheorem{theorem}{Theorem}[section]
 \newtheorem{lemma}{Lemma}[section]
 \newtheorem{proposition}{Proposition}[section]
\theoremstyle{remark}
 \newtheorem{remark}{Remark}[section]

\usepackage{appendix}
\usepackage{chemarrow}
\usepackage{extarrows}
\usepackage{mathtools}
\definecolor{grassgreen}{RGB}{128,255,0}
\newcommand{\R}{\mathbb{R}}

\newcommand{\avg}{\mathcal{M}}
\newcommand{\Hproj}{\mathcal{Q}}
\newcommand{\proj}{\mathcal{P}}

\newcommand{\cP}{\mathcal P}
\newcommand{\cH}{\mathcal H}
\newcommand{\cQ}{\mathcal Q}
\newcommand{\cR}{\mathcal R}
\newcommand{\cM}{\mathcal M}

\DeclareMathOperator{\Span}{span}

\numberwithin{equation}{section}
\numberwithin{table}{section}
\numberwithin{figure}{section}

\DeclareMathOperator{\grad}{grad}
\DeclareMathOperator{\curl}{curl}
\DeclareMathOperator{\divergence}{div}
\renewcommand{\div}{\divergence}

\DeclareMathOperator{\rot}{rot}

\newcommand{\LL}{\langle\!\langle}
\newcommand{\RR}{\rangle\!\rangle}

\makeatletter
\@addtoreset{equation}{section}
\makeatother
\renewcommand{\theequation}{\arabic{section}.\arabic{equation}}
\begin{document}

\title[Commuting Projection Operator for Discrete de Rham Complexes]{Local Bounded Commuting Projection Operator for Discrete de Rham Complexes}
\author{Jun Hu}
\address{LMAM and School of Mathematical Sciences, Peking University, Beijing 100871,
P. R. China.}
\email{hujun@math.pku.edu.cn}
\author{Yizhou Liang}
\address{Institute of Mathematics, University of Augsburg, Universit\"{a}tsstraße 12A, 86159 Augsburg, Germany}
\email{yizhou.liang@uni-a.de}
\author{Ting Lin}
\address{School of Mathematical Sciences, Peking University, Beijing 100871,
P. R. China.}
\email{lintingsms@pku.edu.cn}

\begin{abstract}
The local bounded commuting projection operators of nonstandard finite element de Rham complexes in two and three dimensions are constructed systematically. The assumptions of the main result are mild and can be verified. For three dimensions, the result can be applied to the standard finite element de Rham complex, Hermite complex, Argyris  complex and Neilan's Stokes complex. For two dimensions, the result can be applied to the Hermite--Stenberg complex and the Falk--Neilan Stokes complex.
\end{abstract}
\thanks{The first author was supported by the National Natural Science Foundation of China
grants NSFC 12288101. The second author was supported by Humbodlt Research Fellowship for Postdocs.}
\maketitle

\section{Introduction}
Bounded projections which commute with the governing differential operators play an important role in the analysis of nonstandard finite
element methods, cf. \cite{1974Brezzi,1991BrezziFortin}. 
Particularly, such projections have become a primary instrument in the finite element exterior calculus (FEEC)\cite{2006ArnoldFalkWinther,2010ArnoldFalkWinther,2018Arnold}, which studies the structure preserving discretization of the de Rham complex in two dimensions
$$	\R \xrightarrow{\subset} H^1(\Omega) \xrightarrow[]{\operatorname{curl}} H(\div,\Omega;\mathbb R^2) \xrightarrow[]{\operatorname{div}} L^2(\Omega)\xrightarrow{}0, $$
and in three dimensions
$$	\R \xrightarrow{\subset} H^1(\Omega)\xrightarrow[]{\operatorname{grad}} H(\curl,\Omega;\mathbb R^3) \xrightarrow[]{\operatorname{curl}} H(\div,\Omega;\mathbb R^3) \xrightarrow[]{\operatorname{div}} L^2(\Omega)\xrightarrow{}0. $$

However, a key difficulty is that the canonical projections defined from the degrees of freedom of the finite element spaces are not well-defined on the basic function spaces above (e.g. $H^1(\Omega)$). 
How to construct bounded commuting projections of the de Rham complex to a finite element subcomplex 
$$	\R \xrightarrow{\subset}\Lambda_h^0 \xrightarrow[]{\operatorname{grad}} \Lambda_h^1 \xrightarrow[]{\operatorname{curl}}\Lambda_h^2 \xrightarrow[]{\operatorname{div}}\Lambda_h^3 \xrightarrow{}0$$
is then a difficult problem.

A remarkable construction of bounded commuting projections was given by Sch\"{o}berl~\cite{2005Schoberl} and later improved by Christiansen~\cite{2008ChristiansenWinther}, 
where a smoothing operator is combined with the canonical projection to obtain a commuting projection operator which is globally bounded. 
However, the projection of \cite{2008ChristiansenWinther} is not local, that is, the projection of a function $u$ on a simplex $\sigma$ cannot be determined solely by $u$ on the patch of elements surrounding $\sigma$. 
Recently, Falk and Winther \cite{2014FakWinther,2015FalkWinther} constructed local bounded commuting projections to the standard finite element de Rham complex, 
with the Lagrange finite element space considered as a subspace of the Sobolev space $H^1(\Omega)$, and the Raviart-Thomas\cite{1977RaviartThomas}, Brezzi-Douglas-Marini\cite{1985BrezziDouglasMarini}, and N\'{e}d\'{e}lec\cite{1980Nedelec,1986Nedelec} finite element spaces
considered as subspaces of $H(\operatorname{div}, \Omega; \mathbb R^3)$ or $H(\operatorname{curl}, \Omega; \mathbb R^3)$, respectively. 
More recently, Arnold and Guzm\'{a}n \cite{2021ArnoldGuzman} constructed another commuting projection operators for the standard finite element de Rham complex. 
Note that the projection operators in \cite{2021ArnoldGuzman} are $L^2$-bounded. See \cite{2014FakWinther,2021ArnoldGuzman} for more details.

In this paper, we construct local bounded, commuting projections from 2D and 3D de Rham complexes to subcomplexes consisting of finite element subspaces of arbitrary
polynomial degrees with extra smoothness. Some of these elements can be considered as generalizations of the scalar Argyris, Hermite, and Lagrange elements and vector-valued Stenberg element\cite{2010Stenberg}. 
For example, the two-dimensional Hermite-Stenberg finite element complex takes the Hermite finite element space as a subspace of $H^1(\Omega)$, the Stenberg $H(\div)$ conforming space as a subspace of $H(\div,\Omega;\mathbb R^2)$, which is illustrated as follows. 

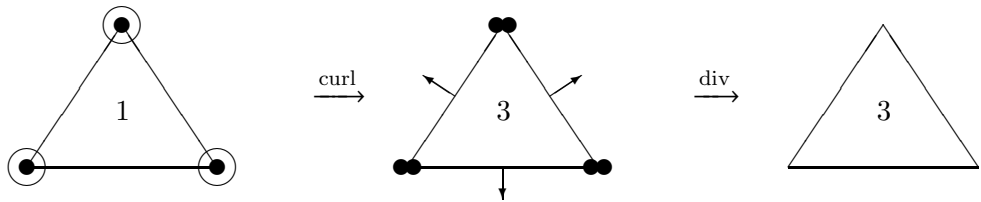
\begin{figure}[htb] \begin{center} \setlength{\unitlength}{1.20pt}

    \begin{picture}(270,45)(0,5)
\put(0,0){\begin{picture}(70,70)(0,0)\put(0,0){\line(1,0){60}}
    \put(0,0){\line(2,3){30}}\put(60,0){\line(-2,3){30}}
    \put(0,0){\circle*{5}}\put(0,0){\circle{12}}
    \put(60,0){\circle*{5}}\put(60,0){\circle{12}}
    \put(30,45){\circle*{5}}\put(30,45){\circle{12}}
	\put(28,15){$1$}
   \end{picture}}
  
\put(90,20){$\xrightarrow[]{\curl}$}
\put(120,0){\begin{picture}(70,70)(0,0)\put(0,0){\line(1,0){60}}
    \put(0,0){\line(2,3){30}}\put(60,0){\line(-2,3){30}}
    \put(-2,0){\circle*{5}}\put(2,0){\circle*{5}}
    \put(58,0){\circle*{5}}\put(62,0){\circle*{5}}
    \put(28,45){\circle*{5}}\put(32,45){\circle*{5}}
	\put(15,22.5){\vector(-3,2){10}}\put(45,22.5){\vector(3,2){10}}
    \put(30,0){\vector(0,-1){10}}
    \put(28,15){$3$}
   \end{picture}}
   \put(210,20){$\xrightarrow[]{\div}$}
  \put(240,0){\begin{picture}(70,70)(0,0)\put(0,0){\line(1,0){60}}
    \put(0,0){\line(2,3){30}}\put(60,0){\line(-2,3){30}}
    \put(28,15){$3$}
   \end{picture}}
\end{picture}
\end{center}
\end{figure}

We refer to \cite{2018ChristiansenHuHu} for the construction of these finite element de Rham complexes. For the simplicity of the presentation, we state an algebraic construction of these commuting projection operators for the 2D and 3D cases, under some conditions and assumptions. And we verify that the finite element de Rham complexes in \cite{2018ChristiansenHuHu} and the finite element Stokes complexes from \cite{2015Neilan,2013FalkNeilan} satisfy these conditions and assumptions. 

As in \cite{2015FalkWinther}, a family of weight functions, denoted as $z_{\sigma}^k$ (see sections below for more precise definitions), which resembles the \v{C}ech de Rham double complex, will be proposed to construct the projections. These weight functions are the generalizations to the more general cases of those for the first-order finite element de Rham subcomplex \cite{2015FalkWinther}.

The rest of the paper is organized as follows.  \Cref{sec:main-result} states the main result of this paper, i.e., a framework of the construction of local bounded commuting projections. \Cref{sec:proof} provides the proof of the main result. \Cref{sec:2D} and \Cref{sec:3D} consider the applications of the framework in two and three dimensions, respectively. 

\section{Assumptions and Main Results}
\label{sec:main-result}

Given any domain $\omega$ in $\mathbb R^2$ or $\mathbb R^3$, denote by $(u,v)_{\omega}$ the standard inner product $\int_{\omega} u\cdot v$ for scalar or vector valued functions $u,v \in L^2(\omega)$. Suppose $X \subset L^2(\omega)$, $u \in L^2(\omega)$, the notation $u \perp X$ means $(u, v)_{\omega} = 0, \forall v \in X$. Additionally, $u \perp \mathbb R$ means $(u,1)_{\omega} = 0$.

\subsection{Two-dimensional Results}

For two dimensions, suppose that a contractible polygonal domain $\Omega$ and a simplicial triangulation $\mathcal T$ are given. Denote by $\mathsf V$ the set of vertices, $\mathsf E$ the set of edges, $\mathsf F$ the set of faces, and $\mathsf S = \mathsf V\cup \mathsf E \cup \mathsf F$ the set of all simplices. The letters $\bm x$, $\bm e$, $\bm f$ are used to represent any vertex, edge, and face of the mesh $\mathcal T$, respectively.
Given a subsimplex $\sigma$ (vertex, edge and face) of $\mathcal T$, define the local patch $\omega_{\sigma}:= \bigcup\{\bm f: \bm \sigma \subseteq \overline{\bm f}\}$, and the extended local patches sequentially:
$$\omega_{\sigma}^{[\ell]} := \cup_{\bm f}\{ \bm f : \overline{\bm f} \cap \overline{\omega_{\sigma}^{[\ell-1]}} \neq \emptyset \}, \text{ with } \omega_{\sigma}^{[0]} = \omega_{\sigma}. $$ Here $\overline{\omega}$ is the closure of the domain $\omega$, see \Cref{fig:omega} for an illustration. Moreover, define the extended patch $\omega_{\sigma}^{h}: =\cup \{\omega_{\bm{x}}:\bm{x}\subseteq \overline{\sigma}\}$. Then it holds that
\begin{equation*}
	\omega_{\sigma}\subseteq \omega_{\sigma}^{h} \subseteq\omega_{\sigma}^{[1]}.
\end{equation*}
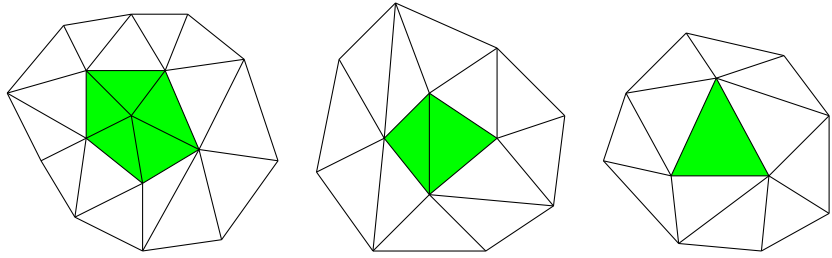
\begin{figure}[htbp]
	\centering
	\input{omega_x.tikz}
    \input{omega_e.tikz}
    \input{omega_f.tikz}	
	\caption{The (extended) local patch, $\omega_{\bm x}^{[1]}$, $\omega_{\bm e}^{[1]}$ and $\omega_{\bm f}^{[1]}$. Colored regions represent $\omega_{\bm x}$, $\omega_{\bm e}$ and $\omega_{\bm f}$ respectively.}
	\label{fig:omega}
\end{figure}

To unify the notation, denote the $\Lambda^0 = H^1$ finite element space as $\Lambda_h^0$, the $\Lambda^1 = H(\div)$ finite element space as $\Lambda_h^1$, and the $\Lambda^2 = L^2$ finite element space as $\Lambda_h^2$.

For the finite element space $\Lambda_h^0$, the (global) degrees of freedom consist of $u(\bm x), \bm x \in \mathsf V$,
and $ f_{\sigma, i}^0, i = 1,2,\ldots, N_{\sigma}^0, \sigma \in \mathsf S$. 
Here $N_{\sigma}^0$ is the number of the degrees of freedom (except $u(\bm x)$ if $\sigma = \bm x$) attached to the sub-simplex $\sigma$ of the space $\Lambda_h^0$. Note that in almost all the cases in this paper, $N_{\bm x}^0$ is nonzero for all vertices $\bm x \in \mathsf V$.
The corresponding basis functions are then denoted as $\varphi_{\bm x}$ and $\phi_{\sigma, i}^0$, respectively. 
Define $L_{\sigma}^0 = \sum_{i = 1}^{N_{\sigma}^0} f_{\sigma, i}^0(\cdot) \phi_{\sigma,i}^0$ as the collection of all the degrees of freedom defined on the sub-simplex $\sigma$ (expect $u(\bm x)$ if $\sigma = \bm x$).
Then by the definitions of the basis functions and degrees of freedom, it holds that for $u \in \Lambda_h^0$,
\begin{equation}\label{eq:dofid-u}
	u = \sum_{\sigma} L_{\sigma}^0u + \sum_{\bm x} u(\bm x) \varphi_{\bm x}.
\end{equation}
Here $\sum_{\sigma}$ and $\sum_{\bm x}$ are the abbreviations of $\sum_{\sigma \in \mathsf S}$ and $\sum_{\bm x \in \mathsf V}$ respectively, while $\sum_{\bm e}$ and $\sum_{\bm f}$ appearing below will be understood in the same way.

For the finite element space $\Lambda_h^1$, the (global) degrees of freedom consist of
$(\bm v \cdot \bm n,1)_{\bm e}, \bm e \in \mathsf E,$ (where $\bm n$ is the unit normal vector of the edge $\bm e$)
and $ f_{\sigma, i}^1(\bm v), i = 1,2,\ldots, N_{\sigma}^1$ for some nonnegative integer $N_{\sigma}^1$, $\sigma \in \mathsf S$. The corresponding basis functions are then denoted by $\varphi_{\bm e}$ and $ \phi_{\sigma, i}^1$, respectively. Define
$L_{\sigma}^1 = \sum_{i = 1}^{N_{\sigma}^1} f_{\sigma, i}^1(\cdot)\phi_{\sigma,i}^1.$
Then by the definitions of the basis functions and degrees of freedom, it holds that for $\bm v \in \Lambda_h^1$,
\[ \bm v = \sum_{\sigma} L_{\sigma}^1 \bm v + \sum_{\bm e} (\bm v \cdot \bm n,1)_{\bm e}  \varphi_{\bm e}.\]

For the finite element space $\Lambda_h^2$, the (global) degrees of freedom consist of
$(p,1)_{\bm f}, \bm f \in \mathsf F,$
and $ f_{\sigma, i}^2(p), i = 1,2,\ldots, N_{\sigma}^2$ for some nonnegative integer $N_{\sigma}^2$, $\sigma \in \mathsf S.$ The corresponding basis functions are then denoted by $\varphi_{\bm f}$ and $\phi_{\sigma, i}^2$, respectively. Define
$L_{\sigma}^2 = \sum_{i = 1}^{N_{\sigma}^2} f_{\sigma, i}^2(\cdot)  \phi_{\sigma,i}^2.$
Then by the definitions of the basis functions and degrees of freedom, it holds that for $p \in \Lambda_h^2$,
\[ p = \sum_{\sigma} L_{\sigma}^2 p + \sum_{\bm f} (p,1)_{\bm f} \varphi_{\bm f}.\]

Here are the assumptions for the two-dimensional main result of this paper.

\begin{remark}
	An implicit assumption is that {\bf (A0)} the finite element spaces contain such degrees of freedoms like $u(\bm x)$, $(\bm v\cdot \bm n,1)_{\bm e}$, and $(p,1)_{\bm f}$, respectively.
	\end{remark}

\begin{itemize}
\item[{\bf (A1)}]
The finite element sequence
\begin{equation}\label{eq:2D:hermitecompl}
	% \tag{{\sf C}}
	\R \xrightarrow{\subset} \Lambda^0_{h}\xrightarrow[]{\operatorname{curl}} \Lambda^1_{h} \xrightarrow[]{\operatorname{div}} \Lambda^2_{h}\xrightarrow{}0
\end{equation} is an exact complex. The exactness also holds when restricted on any (extended) patch $\omega =  \omega_{\sigma}^{h}$ or $\omega_{\sigma}^{[\ell]}$  for non-negative $\ell$, namely,
the sequence
\begin{equation}
	\R \xrightarrow{\subset} \Lambda^0_{h}(\omega)\xrightarrow[]{\operatorname{curl}} \Lambda^1_{h}(\omega)\xrightarrow[]{\operatorname{div}} \Lambda^2_{h}(\omega)\xrightarrow{}0,
\end{equation}
is exact, where $\Lambda_h^k(\omega)$ is the restriction on $\omega$ of $\Lambda_h^k$.

	\item[{\bf (A2)}] It holds that $\curl \varphi_{\bm x} \in \Span \{ \varphi_{\bm e}, \bm e \in \mathsf E\}$ for any vertex $\bm x \in \mathsf V$, and $\div \varphi_{\bm e} \in \Span  \{ \varphi_{\bm f}, \bm f \in \mathsf F\}$ for any edge $\bm e \in \mathsf E$. This implies that, \begin{equation}\label{eq:assumption}f_{\sigma,i}^1(\curl \varphi_{\bm x}) = 0,\,\,f_{\sigma,i}^2(\div \bm \varphi_{\bm e}) = 0 \text{ for all the simplices }\sigma \in \mathsf S.
	      \end{equation}
	\item[{\bf (A3)}] The collection $L_{\sigma}^0$ of the degrees of freedom  vanishes for the constant, namely, $L_{\sigma}^0(c) = 0$ for $c \in \R$ and $\sigma \in \mathsf S$.
	\item[{\bf (A4)}] For $l = 0,1,2$, the mapping $L_{\sigma}^l$ is locally defined and locally supported, namely, for $\eta \in \Lambda_h^l$, the value of $L_{\sigma}^l\eta$ on $\sigma$ only depends on the value of $\eta$ on $\omega_{\sigma}$, and the supports of $L_{\sigma}^l \eta$ and $\varphi_{\sigma}$ are a subset of  $\omega_{\sigma}.$

\end{itemize}

\begin{theorem}
	\label{thm:main-2d}
	Under the above assumptions {\bf(A1)}--{\bf(A4)}, there exist operators $\pi^k$, $k = 0,1,2$, which are projection operators from $\Lambda^k$ to  $\Lambda_h^k$, such that the following diagram commutes.
	\begin{equation}
		\label{eq:cd2}
		\xymatrix{
		\R  \ar[r] & H^1(\Omega) \ar[d]_-{\pi^0}\ar[r]^-{\curl} & H(\div, \Omega; \mathbb R^2) \ar[d]_-{\pi^1} \ar[r]^-{\div} & L^2(\Omega) \ar[d]_-{\pi^2} \ar[r] & 0 \\
		\R \ar[r] & \Lambda_{h}^0 \ar[r]^-{\curl} &\Lambda_{h}^1 \ar[r]^-{\div} & \Lambda_{h}^2 \ar[r] & 0
		}
	\end{equation}
	Namely, $\curl \pi^0 = \pi^1 \curl$, $\pi^2 \div = \div \pi^1$, and the value of $\pi^k(\cdot)$ on subsimplex $\sigma$ is locally determined.

	If moreover, $\mathcal T$ is shape-regular, and the shape function space and degrees of freedom in each element are affine-interpolant equivalent to each other, then the projection operators are local bounded, i.e.,
	$$\|\pi^0 u\|_{H^1(\bm f)} \le C\|u\|_{H^1(\omega_{\bm f}^{[1]})}, \|\pi^1 \bm v\|_{H(\div,\bm f)} \le C\|\bm v\|_{H(\div,\omega^{[2]}_{\bm f})}, \|\pi^2 p\|_{L^2(\bm f)} \le C\|p\|_{L^2(\omega^{[2]}_{\bm f})},$$
	for any $\bm f \in \mathcal T$.
	Here the constant $C$ only depends on the shape regularity and the reference finite element. As a consequence, all the operators are globally bounded, i.e., $\pi^0$ is $H^1$ bounded, $\pi^1$ is $H(\div)$ bounded, and $\pi^2$ is $L^2$ bounded.

\end{theorem}
At the end of this subsection, the lowest order finite element de Rham complex without extra smoothness in 2D will be introduced~\cite{2006ArnoldFalkWinther,2010ArnoldFalkWinther}:
\begin{equation*}
	\R \xrightarrow{\subset} \hat{\Lambda}^0_{h}\xrightarrow[]{\operatorname{curl}} \hat{\Lambda}^1_{h} \xrightarrow[]{\operatorname{div}} \hat{\Lambda}^2_{h}\xrightarrow{}0.
\end{equation*}
Here $\hat{\Lambda}^0_{h} = \mathbf{P}_1 \subset H^1(\Omega)$ denotes the corresponding space of continuous piecewise linear functions and $\hat{\Lambda}^2_{h} = \mathbf{P}_0^- \subset L^2(\Omega)$ denotes the space of piecewise constants. The space $\hat{\Lambda}^1_{h} = \mathbf{RT} \subset H(\operatorname{div},\Omega; \mathbb R^2)$ is the lowest order Raviart--Thomas space (cf. \cite{1977RaviartThomas}). This case can be also encompassed in the framework proposed in this subsection.

\subsection{Three-dimensional Results}

For three dimensions, suppose that a contractible polyhedral domain $\Omega$ and a triangulation $\mathcal T$ are given. Denote by $\mathsf V$ the set of vertices, $\mathsf E$ the set of edges, $\mathsf F$ the set of faces, $\mathsf K$ the set of cells, and $\mathsf S = \mathsf V\cup \mathsf E \cup \mathsf F \cup \mathsf K$ the set of all simplices. The letters $\bm x$, $\bm e$, $\bm f$, $\bm K$ are used to represent any vertex, edge, face, and tetrahedron of the mesh $\mathcal T$, respectively.

Similar to the two-dimensional case, the following expressions hold,
$$u = \sum_{\sigma} L_{\sigma}^0 u + \sum_{\bm x} u(\bm x) \varphi_{\bm x}, \quad \bm \xi = \sum_{\sigma} L_{\sigma}^1 \bm \xi + \sum_{\bm e} (\bm \xi \cdot \bm t, 1)_{\bm e} \varphi_{\bm e},$$
$$ \bm v = \sum_{\sigma} L_{\sigma}^2 \bm v + \sum_{\bm f} (\bm v\cdot \bm n,1)_{\bm f} \varphi_{\bm f},\quad p = \sum_{\sigma} L_{\sigma}^3 p + \sum_{\bm K}(p, 1)_K \varphi_{K},$$
for $u \in \Lambda_h^0, \bm \xi \in \Lambda_h^1, \bm v\in \Lambda_h^2,$ and $p \in \Lambda_h^3$,
where the collections of degrees of freedom $L_{\sigma}^l$, $l = 0,1,2,3$, are defined in a similar way as those for the two-dimensional case, $\bm t$ is the unit tangential vector for edge $\bm e \in \mathsf E$, and $\bm n$ is the unit normal vector for face $\bm f \in \mathsf F$.
Here are the assumptions for the three-dimensional main result of this paper.

\begin{itemize}
	\item[{\bf (B1)}]
The finite element sequence
\begin{equation}
	\R \xrightarrow{\subset} \Lambda^0_{h}\xrightarrow[]{\operatorname{grad}} \Lambda^1_{h} \xrightarrow[]{\operatorname{curl}} \Lambda^2_{h} \xrightarrow[]{\operatorname{div}} \Lambda^3_{h}\xrightarrow{}0
\end{equation} is an exact complex. The exactness also holds when restricted on the patch $\omega_{\sigma}^{h}$ and $\omega_{\sigma}^{[\ell]}$, namely, the sequence 

\begin{equation}
	\R \xrightarrow{\subset} \Lambda^0_{h}(\omega)\xrightarrow[]{\operatorname{grad}} \Lambda^1_{h}(\omega) \xrightarrow[]{\operatorname{curl}} \Lambda^2_{h}(\omega) \xrightarrow[]{\operatorname{div}} \Lambda^3_{h}(\omega) \xrightarrow{}0
\end{equation}
is exact for the patch $\omega =\omega_{\sigma}^{h}$ or $ \omega_{\sigma}^{[\ell]}$ for any $\ell \ge 0$ and $\sigma \in \mathsf S$.

	\item[{\bf (B2)}] It holds that $\grad \varphi_{\bm x} \in \Span \{ \varphi_{\bm e}, \bm e \in \mathsf E\}$, and $\curl \varphi_{\bm e} \in \Span  \{ \varphi_{\bm f}, \bm f \in \mathsf F\}$, and $\div \varphi_{\bm f} \in \Span\{\varphi_{\bm K}, \bm K \in \mathsf K\}.$ This indicates that
	      \begin{equation}\label{eq:assumption3d}f_{\sigma,i}^1(\grad \varphi_{\bm x}) = 0,\,\,f_{\sigma,i}^2(\curl \bm \varphi_{\bm e}) = 0, \,\, f_{\sigma,i}^3(\div \varphi_{\bm f}) = 0\,\text{ for all simplices }\sigma \in \mathsf S.
	      \end{equation}
	\item[{\bf (B3)}] It holds for the collection $L_{\sigma}^0$ of the degrees of freedom that $L_{\sigma}^0(c) = 0$ for $c \in \R$ and $\sigma \in \mathsf S$.
	\item[{\bf (B4)}] Given $u \in \Lambda_h^k$, the value of $L_{\sigma}^ku$ only depends on the value of $u$ on $\omega_{\sigma}$, and the support of $L_{\sigma}^k u$ and $\varphi_{\sigma}$ is a subset of $\omega_{\sigma}.$
\end{itemize}

\begin{theorem}
	\label{thm:main-3d}
	Under the above assumptions {\bf(B1)}--{\bf(B4)}, there exist operators $\pi^k$, $k = 0,1,2,3$, which are projection operators from $\Lambda^k$ to $\Lambda_h^k$, such that the following diagram commutes.
	\begin{equation}
		\label{eq:cd3}
		\xymatrix{
		\R  \ar[r] & H^1(\Omega) \ar[d]_-{\pi^0} \ar[r]^-{\grad} & H(\curl,\Omega;\mathbb R^3) \ar[d]_-{\pi^1} \ar[r]^-{\curl} & H(\div, \Omega; \mathbb R^3) \ar[d]_-{\pi^2} \ar[r]^-{\div} & L^2(\Omega) \ar[d]_-{\pi^3} \ar[r] & 0 \\
		\R \ar[r] & \Lambda_{h}^0 \ar[r]^-{\grad} &\Lambda_{h}^1 \ar[r]^-{\curl} & \Lambda_{h}^2 \ar[r]^-{\div} & \Lambda_{h}^3 \ar[r] & 0
		}
	\end{equation}
	Given $u \in \Lambda^k$, the value of $\pi^k u$ is locally determined. Here $\Lambda^0 = H^1(\Omega)$, $\Lambda^1 = H(\curl,\Omega;\mathbb R^3)$, $\Lambda^2 = H(\div,\Omega;\mathbb R^3)$, $\Lambda^3 = L^2(\Omega)$.

	If moreover, $\mathcal T$ is shape-regular, and the shape function space and degrees of freedom in each element are affine-interpolant equivalent to each other, then the projection operators are local bounded, i.e.,
	$$\|\pi^0 u\|_{H^1(\bm K)} \le C\|u\|_{H^1(\omega^{[1]}_{\bm K})},\,\, \|\pi^1 \bm \xi\|_{H(\curl,\bm K)} \le C\|\bm \xi\|_{H(\curl,\omega^{[2]}_{\bm K})},$$
	$$\|\pi^2 \bm v\|_{H(\div, \bm K)} \le C\|\bm v\|_{H(\div, \omega_{\bm K}^{[3]})},\,\, \|\pi^3 p\|_{L^2(\bm K)} \le C\|p\|_{L^2(\omega_{\bm K}^{[3]})},$$
	for any $u \in \Lambda^0, \bm \xi \in \Lambda^1, \bm v \in \Lambda^2$ and $p \in \Lambda^3$.
	Here the constant $C$ only depends on the shape regularity and the reference finite element. As a consequence, all the operators are globally bounded, i.e., $\pi^0$ is $H^1$ bounded, $\pi^1$ is $H(\curl)$ bounded, $\pi^2$ is $H(\div)$ bounded and $\pi^3$ is $L^2$ bounded.

\end{theorem}

Similarly, consider the lowest order finite element de Rham complex without extra smoothness in 3D\cite{2006ArnoldFalkWinther,2010ArnoldFalkWinther}:
\begin{equation*}
	\R \xrightarrow{\subset} \hat{\Lambda}^0_{h}\xrightarrow[]{\operatorname{grad}} \hat{\Lambda}^1_{h} \xrightarrow[]{\operatorname{curl}} \hat{\Lambda}^2_{h}\xrightarrow[]{\operatorname{div}} \hat{\Lambda}^3_{h}\xrightarrow{}0.
\end{equation*}
Here $\hat{\Lambda}^0_{h}\subset H^1(\Omega)$ denotes the corresponding space of continuous piecewise linear functions and $\hat{\Lambda}^3_{h}\subset L^2(\Omega)$ denotes the space of piecewise constants. The space $\hat{\Lambda}^1_{h}\subset H(\operatorname{curl},\Omega; \mathbb R^3)$ is the lowest order N\'{e}d\'{e}l{e}c element and $\hat{\Lambda}^2_{h}\subset H(\operatorname{div},\Omega; \mathbb R^3)$ is the lowest order Raviart--Thomas space.

\subsection{Double Complexes and Weight Functions}
This subsection briefly introduces the weight functions $z_{\sigma}^{k}$ in $\mathbb {R}^{D}(D=2,3)$~\cite{2014FakWinther,2015FalkWinther}, for $k = 0,1,\ldots,D$, and $\sigma \in \mathsf S$. Hereafter, denote by $\vmathbb{1}_{\omega}$ the indicator function of the domain $\omega$. 

For three dimensions, given $\bm x \in \mathsf V$, define $z_{\bm x}^0 =  \frac{\vmathbb{1}_{\omega_{x}}}{|\omega_{\bm x}|},$ consider $\mathcal M^0 u := \sum_{\bm x} (u, z_{\bm x}^0)_{\omega_{\bm x}}  \varphi_{\bm x}$ and $M_{\bm x} u:= (u, z_{\bm x}^0)_{\omega_{\bm x}}$.
Note that $M_{\bm x} u $ is the integral mean of $u$ on the patch $\omega_{\bm x}$.

By Assumption (B2), it holds that $\grad \mathcal M^0 u \in \Span_{\bm e}\{\varphi_{\bm e}\}$, that is,
\begin{equation}
	\begin{split}
		\grad \mathcal M^0 u = \sum_{\bm e} (\grad \mathcal M^0 u, \bm t)_{e} \varphi_{\bm e} 
		= \sum_{e} ( M_{\bm y}u -  M_{\bm x}u) \varphi_{\bm e} 
		= \sum_{e} (u,  \frac{\vmathbb{1}_{\omega_{\bm y}}}{|\omega_{\bm y}|} - \frac{\vmathbb{1}_{\omega_{\bm x}}}{|\omega_{\bm x}|})_{\bm \omega_e^{h}} \varphi_{\bm e},
	\end{split}
\end{equation}
where $\bm e = [\bm x,\bm y]$. 
Since the integral mean of $(\frac{\vmathbb{1}_{\omega_{\bm y}}}{|\omega_{\bm y}|} - \frac{\vmathbb{1}_{\omega_{\bm x}}}{|\omega_{\bm x}|})$ vanishes on $\omega_{\bm e}^{h}$, there exists $\bm z_{\bm e}^1 \in \hat{\Lambda}_{h}^2(\omega_{\bm e}^{h})$ such that 
$
	-\div \bm z_e^1 = (\frac{\vmathbb{1}_{\omega_{\bm y}}}{|\omega_{\bm y}|} - \frac{\vmathbb{1}_{\omega_{\bm x}}}{|\omega_{\bm x}|})$
and 
	$\bm z_{\bm e}^1\cdot \bm n = 0
	\text{ on the boundary of }
	\omega_{\bm e}^{h}.$ 
Here $\bm n$ is the unit normal vector on the boundary. 
As a result, it follows that $$\grad \mathcal M^0 u = \sum_{\bm e} (u, -\div z_{\bm e}^1)_{\omega_{\bm e}^{h}} \varphi_{\bm e} = \sum_{\bm e} (\grad u, z_{\bm e}^1)_{\omega_{\bm e}^{h}} \varphi_{\bm e}. $$

This motivates to define $\mathcal M^1 \bm \xi = \sum_{\bm e}(\bm \xi, z_{\bm e}^1)_{\omega_{\bm e}^{h}} \varphi_{\bm e}$ for $\bm \xi \in L^2(\Omega; \mathbb R^2)$.  Then it holds that 
\begin{equation}\label{eq:M1vn}(\mathcal M^1 \bm \xi, \bm t)_{\bm e} = (\bm \xi, z_{\bm e}^1)_{\omega_{\bm e}^{h}},\end{equation}
for any edge $\bm e$ with the unit tangential vector $\bm t$.
Similarly, it follows from Assumption (B2) that
\begin{equation}
	\begin{split}
		\curl \mathcal M^1 \bm \xi 
		= &
		\sum_{\bm f} (\curl \mathcal M^1 \bm \xi, \bm n)_{\bm f} \varphi_{\bm f} 
		\\ = & 
		\sum_{\bm f} \sum_{\bm e \subset \bm f } (\mathcal M^1 \bm \xi, \bm t)_{\bm e} \varphi_{\bm f} 
		\\ = &
		\sum_{\bm f} (\bm \xi, \bm z_{[a,b]}^1 \vmathbb{1}_{\omega_{[a,b]}^{h}} + \bm z_{[b,c]}^1 \vmathbb{1}_{\omega_{[b,c]}^{h}} + \bm z_{[c,a]}^1 \vmathbb{1}_{\omega_{[c,a]}^{h}})_{\omega_{\bm f}^{h}}\varphi_{\bm f}.
	\end{split}
\end{equation}
Here $\bm f = [a,b,c]$, and $[a,b], [b,c], [c,a]$ are the three edges of $\bm f$, and the last line comes from \eqref{eq:M1vn}.
Since $\bm z_{[a,b]}^1 \vmathbb{1}_{\omega_{[a,b]}^{h}} + \bm z_{[b,c]}^1 \vmathbb{1}_{\omega_{[b,c]}^{h}} + \bm z_{[c,a]}^1 \vmathbb{1}_{\omega_{[c,a]}^{h}}\in \hat{\Lambda}_{h}^2(\omega_{\bm f}^{h})$ is  divergence free and the normal component vanish on the boundary of $\omega_{\bm f}^{h}$, there exists  $z_{\bm f}^2 \in \hat{\Lambda}_h^1(\omega_{\bm f}^{h})$ such that $$\curl z_{\bm f}^2 = \bm z_{[a,b]}^1 \vmathbb{1}_{\omega_{[a,b]}^{h}} + \bm z_{[b,c]}^1 \vmathbb{1}_{\omega_{[b,c]}^{h}} + \bm z_{[c,a]}^1 \vmathbb{1}_{\omega_{[c,a]}^{h}},$$
and $z_{\bm f}^2 \times \bm n$ vanishes on the boundary of the patch $\omega_{\bm f}^{h}$.
As a consequence,
$$\curl \mathcal M^1 \bm \xi = \sum_{f} (\curl \bm \xi, z_{\bm f}^2)_{\omega_{\bm f}^{h}} \varphi_{\bm f}.$$

This motivates to define $\cM^2 \bm v = \sum_{\bm f}(\bm v, z_{\bm f}^2)_{\omega_{\bm f}^h} \varphi_{\bm f}$ for $\bm v \in L^2(\Omega; \mathbb R^3)$. A similar argument yields that 
\begin{equation}
	\begin{split}
		\div \mathcal M^2 \bm v 
		= &
		\sum_{\bm K} (\div \mathcal M^2 \bm v, 1)_{\bm K} \varphi_{\bm K} 
		\\ = & 
		\sum_{\bm K} \sum_{\bm f \subset \bm K} (\mathcal M^2 \bm v, \bm n)_{\bm f} \varphi_{\bm K} 
		\\ = &
		\sum_{\bm K} (\bm v, \bm z_{[a,b,c]}^2 \vmathbb{1}_{\omega_{[a,b,c]}^{h}} + \bm z_{[b,c,d]}^2 \vmathbb{1}_{\omega_{[b,c,d]}^{h}} + \bm z_{[c,d,a]}^2 \vmathbb{1}_{\omega_{[c,d,a]}^{h}} + \bm z_{[d,a,b]}^2 \vmathbb{1}_{\omega_{[d,a,b]}^{h}})_{\omega_{\bm K}^{h}}\varphi_{\bm K},
	\end{split}
\end{equation}
where the element $\bm K = [a,b,c,d]$. Since $\bm z_{[a,b,c]}^2 \vmathbb{1}_{\omega_{[a,b,c]}^{h}} + \bm z_{[b,c,d]}^2 \vmathbb{1}_{\omega_{[b,c,d]}^{h}} + \bm z_{[c,d,a]}^2 \vmathbb{1}_{\omega_{[c,d,a]}^{h}} + \bm z_{[d,a,b]}^2 \vmathbb{1}_{\omega_{[d,a,b]}^{h}}$ is curl free, there exists $z_{\bm K}^3 \in \hat{\Lambda}_{h}^0(\omega_{\bm K}^{h})$ with $z_{\bm K}^3=0$ on the boundary of $\omega_{\bm K}^{h}$ such that 
$$-\grad z_{\bm K}^3 = \bm z_{[a,b,c]}^2 \vmathbb{1}_{\omega_{[a,b,c]}^{h}} + \bm z_{[b,c,d]}^2 \vmathbb{1}_{\omega_{[b,c,d]}^{h}} + \bm z_{[c,d,a]}^2 \vmathbb{1}_{\omega_{[c,d,a]}^{h}} + \bm z_{[d,a,b]}^2 \vmathbb{1}_{\omega_{[d,a,b]}^{h}},$$
and therefore $\div \mathcal M^2 \bm v  = \sum_{\bm K} (\div \bm v, z_{\bm K}^3)_{\omega_{\bm K}^{h}} \varphi_{\bm K}$.

The property of the weight functions $z_{\sigma}^k (\sigma \in \mathbf S, k = 0,1,2 \text{ in 2D }, k = 0,1,2,3 \text{ in 3D})$ is summarized here.

		\begin{proposition}[\cite{2014FakWinther,2015FalkWinther}]
		\label{prop:dblcmplx}

		For three dimensions, let $\mathcal M^0 u = \sum_{\bm x} (u, z^0_{\bm x})_{\omega_{\bm x}} \varphi_{\bm x}$, $\mathcal M^1 \bm \xi = \sum_{\bm e} (\bm \xi, \bm z_{\bm e}^1)_{\omega_{\bm e}^{h}} \varphi_{\bm e},$ $\mathcal M^2 \bm v = \sum_{f} (\bm v, \bm z_{\bm f}^2)_{\omega_{f}^{h}} \varphi_{\bm f}$ and $\mathcal M^3 p = \sum_{\bm K}(p, z_{\bm K}^3)_{\omega_{\bm K}^{h}} \varphi_{\bm K}$ be defined as above. Then the following diagram commutes.
		\begin{equation}
			\label{eq:cd3ske}
			\xymatrix{
			\R  \ar[r] & H^1(\Omega) \ar[d]_-{\mathcal M^0} \ar[r]^-{\grad} & H(\curl,\Omega;\mathbb R^3) \ar[d]_-{\mathcal M^1} \ar[r]^-{\curl} & H(\div,\Omega;\mathbb R^3) \ar[d]_-{\mathcal M^2} \ar[r]^-{\div} & L^2(\Omega) \ar[d]_-{\mathcal M^3} \ar[r] & 0 \\
			\R \ar[r] & \Span\{\varphi_{\bm x}\} \ar[r]^-{\grad} & \Span\{\varphi_{\bm e}\}  \ar[r]^-{\curl} & \Span\{\varphi_{\bm f}\} \ar[r]^-{\div} & \Span\{\varphi_{\bm K}\} \ar[r] & 0
			}
		\end{equation}
		For two dimensions, there exist weighted functions $z_{\bm x}^0 = \frac{1}{|\omega_{\bm x}|} \vmathbb{1}_{\omega_{\bm x}} \in L^2(\omega_{\bm x})$, $z_{\bm e}^1 \in L^2(\omega_{\bm e}^h)$, and $z_{\bm f}^2 \in L^2(\omega_{\bm f}^h)$, such that the following result holds. Let $\mathcal M^0 u = \sum_{\bm x} (u, z^0_{\bm x})_{\omega_{\bm x}} \varphi_{\bm x}$, $\mathcal M^1 \bm v = \sum_{\bm e} (\bm v, \bm z_{\bm e}^1)_{\omega_{\bm e}^{h}} \varphi_{\bm e},$ $\mathcal M^2 p = \sum_{f} (p, z_{\bm f}^2)_{\omega_{f}^{h}} \varphi_{\bm f}$ be defined similarly. Then the following diagram commutes.
\begin{equation}
	\label{eq:cd2ske}
	\xymatrix{
		\R  \ar[r] & H^1(\Omega) \ar[d]_-{\cM^0}\ar[r]^-{\curl} & H(\div, \Omega;\mathbb R^2) \ar[d]_-{\cM^1} \ar[r]^-{\div} & L^2(\Omega) \ar[d]_-{\cM^2} \ar[r] & 0 \\
		\R \ar[r] & \Span\{\varphi_{\bm x}\} \ar[r]^-{\curl} &\Span\{\varphi_{\bm e}\}\ar[r]^-{\div} & \Span\{\varphi_{\bm f}\}\ar[r] & 0
	}
\end{equation}

	\end{proposition}

	\begin{proof}
		The case of three dimensions has been thoroughly discussed, and the same argument can be applied to the two-dimensional case as well. For more information, the readers can refer to \cite{2014FakWinther,2015FalkWinther} for detailed explanations.
	\end{proof}

	Notice that in the simplest Lagrange--Raviart--Thomas finite element complex (or so-called the Whitney form), the interpolation operator can be recovered from the above proposition, see \Cref{sec:lrt} below. The main result of this paper can be regarded as a generalization of the above proposition to the more general cases where other degrees of freedom are involved.

\section{Proof of Main Results}
\label{sec:proof}
In this section, we present the proof of \Cref{thm:main-3d}. The proof of \Cref{thm:main-2d} follows a similar argument and for the sake of completeness, we provide a detailed proof in \Cref{sec:proof-2d}. Our approach extends the method used in \cite{2021ArnoldGuzman,2015FalkWinther}, but with additional technical details.

\subsection{Auxiliary Operators: An Abstract View}
\label{sec:auxproj}
The auxiliary operators used in this section are either exactly or closely related to the following harmonic projection, which is frequently used in the construction of interpolation operators, global or local. This definition is also useful in the following two sections, where it is employed to define the degrees of freedom. Consider the following diagram (where each space is a Hilbert space),

\begin{equation}
	\label{eq:cdharmproj}
	\xymatrix{
	&  & \widetilde{Y} \ar[r]_-{d'}\ar@<-1mm>[d]_-{\mathcal H} & \widetilde Z & \\
	0 \ar[r] & X \ar[r]_-{d} &Y \ar[r]_-{d'}\ar@<-1mm>[u]_-{\iota}  &Z\ar[r] \ar[u]_-{\iota} & 0,
	}
\end{equation}
where $d : X \to Y$ and $d': \widetilde Y \to \widetilde Z$ are two bounded operators such that the lower row is short exact (meaning that $d$ is injective, $d'$ is surjective). The inclusion $\iota$ means that $\widetilde{Y} \subset Y$ and $\widetilde Z \subset Z$, and the upper row implies $d'(Y) \subset Z$. The harmonic projection operator $\mathcal H: \widetilde Y \to Y$ is defined as the projection operator with respect to the following inner product
$$\langle\!\langle u, v \rangle\!\rangle_{Y} = (\cP_{d(X)} u, \cP_{d(X)} v)_Y + (d' u, d' v)_Y, \quad\forall u, v \in Y,$$
where $\mathcal P_{d(X)}$ is the $L^2$ projection operator from $Y$ to the image $d(X)$, and $(\cdot,\cdot)_Y$ is the inner product on $Y$.
More precisely, given $u \in \widetilde{Y}$, the element $\cH u \in Y$ is determined by the following variational problems,
\[
\begin{cases}
	(\cH u, d \varphi)_{Y} = (u, \iota (d\varphi))_{\widetilde Y}, \quad \forall \varphi \in X, \\
	( d' \cH u,  d' v)_{Z} = ( d'u, \iota (d' v))_{\widetilde Z}, \quad \forall v \in Y.
\end{cases}
\]

Here the inclusion operator will be omitted if it is clear from the context. It follows from the exactness of $X\to Y \to Z$ that the bilinear form $\langle\!\langle u, v \rangle\!\rangle_{Y}$ is indeed an inner product. As a result, the harmonic projection operator $\cH$ is well-defined, and $\cH u = u$ for $u \in Y$.
Moreover, the following decomposition holds:
\begin{equation}
	\cH u = d \cH^{-} u + \cH^{+} d'u,
\end{equation}
where $\cH^- u : \widetilde Y \to X$ is defined by
\begin{equation}
	(d\cH^{-} u, d \varphi)_{Y} = (u, \iota (d\varphi))_{\widetilde Y}, \quad \forall \varphi \in X,
\end{equation}
and $\cH^+ \eta : \widetilde Z \to Y$ is defined as the unique function $\mathcal H^+ \eta \in Y$ such that $\cH^+\eta \perp d(X)$ and
\begin{equation}
	\mathcal (d'\cH^{+} \eta, d'v)_{Z} = (\eta, \iota( d'v))_{\widetilde Z}, \quad \forall \varphi \in Y.
\end{equation}
It follows from the short exactness that both $\cH^-$ and $\cH^+$ are well-posed, see \eqref{eq:cdharmprojsplitting} for an illustration. 

\begin{equation}
	\label{eq:cdharmprojsplitting}
	\xymatrix{
	&  & \widetilde{Y}\ar@<-1mm>@{.>}[dl]_-{\mathcal H^-} \ar[r]^-{d'}  &\widetilde Z \ar@<-1mm>@{.>}[dl]^-{\mathcal H^+}  & \\
	0 \ar[r] & X \ar[r]_-{d} &Y \ar[r]_-{d'}\ar[u]_-{\iota}  &Z \ar[r] & 0,
	}
\end{equation}

Next, consider the following exact (cochain) sequence, (here $X^k \subset \widetilde X^k$ for all $k \ge 0$, and the inclusion relationship is omitted in the diagram) which can be regarded as stacking multiple diagrams like \eqref{eq:cdharmprojsplitting} together,

\begin{equation}
	\label{eq:cd:splitproj-long}
	\xymatrix{
	& \widetilde{X}^0 \ar[r]^-{d^0}\ar@<-1mm>@{.>}[dl]|-{\mathcal R^0} \ar@<-1mm>[d]|-{\mathcal Q^0} & \widetilde{X}^1\ar@<-1mm>@{.>}[dl]|-{\mathcal R^1}\ar[r]^-{d^1}\ar@<-1mm>[d]|-{\mathcal Q^1} &\widetilde X^2 \ar@<-1mm>@{.>}[dl]|-{\mathcal R^2} \ar[r]^-{d^2} \ar@<-1mm>[d]|-{\mathcal Q^2}&  \widetilde X^3\ar@<-1mm>@{.>}[dl]|-{\mathcal R^3}\ar[r]^-{d^3} \ar@<-1mm>[d]|-{\mathcal Q^3}& \cdots \ar@<-1mm>@{.>}[dl]\\
	X^{-1} \ar[r]_-{d^{-1}} & X^0 \ar[r]_-{d^0} & X^1 \ar[r]_-{d^1}  & X^2 \ar[r]_-{d^2}& X^3 \ar[r]_-{d^3}& \cdots,
	}
\end{equation}
where the projection operator $\cR^{k}: \widetilde X^{k} \to X^{k-1}$ is defined by
\begin{equation}\label{eq:defabs-Rk}\cR^{k} u^k \perp d^{k-2}(X^{k-2}), \,\,\text{and}\,\, (d^{k-1}\cR^k u^k, d^{k-1} v^{k-1})_{X_{k}} = (u^k, d^{k-1}v^{k-1})_{\widetilde X^{k}}, \forall v^{k-1} \in X^{k-1}\end{equation}
and the projection operator $\cQ^k : \widetilde X_k \to X_k$ is defined as
\begin{equation}\label{eq:defabs-Qk}\cQ^k u^k = d^{k-1} \cR^{k} u^k + \cR^{k+1} d^{k} u^k, \,\, \forall u^k \in \widetilde X^k.\end{equation}
It follows that $\cQ^k u^k = u^k$ for $u^k \in X_k$.

\subsection{The Construction of Projection Operators in Three Dimensions}

\subsubsection{A family of auxiliary operators}
The auxiliary operators can be specified 
	with $\widetilde X^0 = H^1(\omega)$, $\widetilde X^1 = H(\curl,\omega;\mathbb R^3)$, $\widetilde X^2 = H(\div,\omega;\mathbb R^3)$, $\widetilde X^3 = L^2(\omega)$, $X^{-1} = \mathbb R$, $X^{k} = \Lambda_h^k(\omega)$, $d^{-1} = \subset$, $d^0 = \grad$, $d^1 = \curl$, $d^2 = \div$. 
	This results in the following diagram,
\begin{equation}
	\label{eq:longdiagram}
	\xymatrix{
	& H^1(\omega)\ar@<-1mm>@{.>}[dl]|-{\cR^0_{\omega}} \ar[r]^-{\grad} \ar@<-1mm>[d]^-{\mathcal Q^0_{\omega}} & H(\curl, \omega; \mathbb R^3) \ar@<-1mm>@{.>}[dl]|-{\cR^1_{\omega}} \ar[r]^-{\curl} \ar@<-1mm>[d]^-{\mathcal Q^1_{\omega}} &H(\div, \omega; \mathbb R^3) \ar@<-1mm>@{.>}[dl]|-{\cR^2_{\omega}} \ar[r]^-{\div} \ar@<-1mm>[d]^-{\mathcal Q^2_{\omega}} & L^2(\omega) \ar@<-1mm>@{.>}[dl]|-{\cR^3_{\omega}} & \\
	\mathbb R \ar[r] & \Lambda_h^0(\omega) \ar[r]_-{\grad} & \Lambda_h^1(\omega) \ar[r]_-{\curl}  &\Lambda_h^2(\omega) \ar[r]_-{\div} & \Lambda_h^3(\omega) \ar[r] & 0 .
	}
\end{equation}

For the sake of completeness, all the auxiliary operators will be concretely constructed as follows. First, construct the operators $\mathcal R_{\omega}^k$, $k = 0,1,2,3$:
\begin{itemize}
\item[-]
given $u \in H^1(\omega)$, define $\cR^0_{\omega} u = M_{\omega} u \in \mathbb R $  with $M_{\omega} u = \fint_{\omega} u$, the integral mean of $u$ over $\omega$;
\item[-] for $\bm \xi \in H(\curl,\omega; \mathbb R^3)$,
define $\cR^1_{\omega} \bm \xi \in \Lambda_h^0(\omega)$ such that $\cR^1_{\omega} \bm \xi \perp \mathbb R$ and $(\grad \cR^1_{\omega} \bm \xi, \grad v)_{\omega} = (\bm \xi, \grad v)_{\omega}$, $\forall v \in \Lambda_h^0(\omega)$;
\item[-]
for $\bm v \in H(\div, \omega; \mathbb R^3)$, define $\cR^2_{\omega} \bm v \in \Lambda_h^1(\omega)$ such that $\mathcal R_{\omega}^2 \bm v \perp \grad(\Lambda_h^0(\omega))$ and $(\curl \cR^2_{\omega} \bm v, \curl \bm \eta)_{\omega} = (\bm v, \curl \bm \eta)_{\omega}$, $\forall v \in \Lambda_h^1(\omega)$;
\item[-]
for $p \in L^2(\omega)$, define $\cR^3_{\omega} p \in \Lambda_h^3(\omega)$ such that $\mathcal R_{\omega}^3 p \perp \curl(\Lambda_h^1(\omega))$ and $(\div \cR_{\omega}^3 p, \div \bm w)_{\omega} = (p, \div \bm w)_{\omega}$, $\forall \bm w \in \Lambda_h^2(\omega).$
\end{itemize}

Once $\cR_{\omega}^{k}, k = 0,1,2,3,$ are determined, the harmonic projection operators are then determined by:
\begin{itemize}
\item[-] for  $u \in H^1(\omega)$, set $\cQ^0_{\omega} u = \cR^0_{\omega} u + \cR^{1}_{\omega} \nabla u$; 
\item[-] for $\bm \xi \in H(\curl,\omega;\mathbb R^3)$, set $\cQ^1_{\omega} \bm \xi = \nabla \cR^1_{\omega} \bm \xi + \cR^2_{\omega} \curl \bm \xi$; 
\item[-] for  $\bm v \in H(\div,\omega;\mathbb R^3)$, set $\cQ^2_{\omega} \bm v = \curl \cR^2_{\omega} \bm v + \cR^3_{\omega} \div \bm v.$
\end{itemize}
It then follows from the argument in \Cref{sec:auxproj} that under Assumption (B1), $\cQ_{\omega}^k$, $k = 0,1,2$ are projection operators, i.e., $\cQ_{\omega}^0 u =u,$ $\cQ_{\omega}^1 \bm \xi = \bm \xi$, $\cQ_{\omega}^2 \bm v= \bm v$ for $u \in \Lambda_0^h$, $\bm \xi \in \Lambda_1^h$ and $\bm v \in \Lambda_2^h$, respectively. 

For convenience, given a simplex $\sigma$ and $k \ge 0$, let $M_{\sigma} = M_{\omega_{\sigma}}$, $\cQ_{\sigma}^k := \cQ_{\omega_{\sigma}}^k$ and $\cR_{\sigma}^k  := \cR_{\omega_{\sigma}}^k$; given $\ell > 0$, let $\cQ_{\sigma,[\ell]}^k := \cQ_{\omega_{\sigma}^{[\ell]}}^k$, $\cR_{\sigma,[\ell]}^k := \cR_{\omega_{\sigma}^{[\ell]}}^k$.

\subsubsection{Construction of $\pi^0$}
The construction is based on the following diagram from \eqref{eq:longdiagram}:
\begin{equation}
	\label{eq:pi0harmonic-3d}
	\xymatrix{
	&  & H^1(\omega) \ar@<-1mm>@{.>}[dl]|-{\cR^0_{\omega}} \ar[r]^-{\grad} \ar@<-1mm>[d]^-{\mathcal Q^0_{\omega}} &H(\curl,\omega;\mathbb R^3) \ar@<-1mm>@{.>}[dl]^-{\mathcal R^1_{\omega}}  & \\
	0 \ar[r] & \mathbb R \ar[r]_-{\subset} & \Lambda_h^0(\omega) \ar[r]_-{\curl}  &\Lambda_h^1(\omega) 
	}
\end{equation}
Precisely, given $u \in H^1(\Omega)$, the projection $\pi^0 u \in \Lambda_h^0$ is defined as 
$$\pi^0 u = \sum_{\sigma} L_{\sigma}^0\cQ_{\sigma}^0u + \sum_{\bm x} (\cQ _{\bm x}^0 u)(\bm x) \varphi_{\bm x}$$
It follows from Assumption (B3) and the relationship $\cQ_{\omega}^0 u = \cR_{\omega}^0 u + \cR_{\omega}^1(\nabla u)$, that 

\begin{equation}
	\label{eq:pi0u-3d}
	\pi^0 u = \sum_{\sigma} L_{\sigma}^0 \cR^1_{\sigma} \nabla u + \sum_{\bm x} (\cR^1_{\bm x} \nabla u)(\bm x) \varphi_{\bm x} + \cM^0 u.
\end{equation}
Note that $\cR_{\omega}^0 u $ is a constant, and $\mathcal M^0 u = \sum_{\bm x} M_{\bm x} u \varphi_{\bm x} = \sum_{\bm x} ( u, z_{\bm x}^0)_{\omega_{\bm x}} \varphi_{\bm x}$ is defined in \Cref{prop:dblcmplx}.

We now check that $\pi^0$ is a projection operator. Given $u \in \Lambda_h^0$, it holds that $(\mathcal Q^0_{\sigma} u )|_{\omega_{\sigma}} = u_{\omega_{\sigma}}$. Therefore, by Assumption (B4), for $u \in \Lambda_h^0$ it holds that 
$$\pi^0 u = \sum_{\sigma}L^0_{\sigma} u + \sum_{x} u(\bm x)\varphi_{\bm x} = u.$$

It is worth mentioning that 
% the value of $\pi^0 u $ on the simplex $\sigma$ is only determined by the value of $u$ on $\omega_{\sigma}${\color{red}[Is it right?]}. Consequently, 
the value of $\pi^0 u$ on $\omega_{\sigma}$ is determined by the value of $u$ on $\omega^{[1]}_{\sigma}$.

\subsubsection{Construction of $\pi^1$}
Taking $\nabla$ in \eqref{eq:pi0u-3d} yields that 
\begin{equation}
	\label{eq:nablapi0u-3d} \nabla \pi^0 u = \sum_{\sigma} \nabla L_{\sigma}^0 \cR^1_{\sigma} \nabla u + \sum_{\bm x} (\cR^1_{\bm x} \nabla u)(\bm x) \nabla \varphi_{\bm x} + \nabla\cM^0  u.
\end{equation}

It follows from \Cref{prop:dblcmplx} that $\nabla \mathcal M^0 u = \mathcal M^1 \nabla u$ for all $u \in H^1(\Omega)$. This motivates to define an operator $\widehat{\pi}^1 : H(\curl, \Omega; \mathbb R^3) \to \Lambda_h^1$ by 
\begin{equation}
	\label{eq:hatpi1-3d}
	\widehat{\pi}^1 \bm \xi = \sum_{\sigma} \nabla L_{\sigma}^0 \cR^1_{\sigma} \bm \xi + \sum_{\bm x} (\cR^1_{\bm x} \bm \xi)(\bm x) \nabla \varphi_{\bm x}  + \avg^1 \bm \xi,
\end{equation}
for $\bm \xi \in H(\curl, \Omega; \mathbb R^3)$. Note that the value of $\widehat{\pi}^1 \bm \xi$ on $\omega_{\sigma}$ is determined by the value of $\bm \xi$ on $\omega_{\sigma}^{[1]}$.

Then, it follows from \eqref{eq:nablapi0u-3d}  and \eqref{eq:hatpi1-3d} that $ \widehat{\pi}^1 \curl u = \curl \pi^0 u$ for any $u\in H^1(\Omega)$. In particular, it holds that $ \widehat{\pi}^1 \curl u|_{\omega_{\sigma}} = \curl \pi^0 u|_{\omega_{\sigma}}$ depends on the value of $u$ on $\omega_{\sigma}^{[1]}$. As a result, when considering the value on $\omega_{\sigma}$, the operator $\widehat{\pi}^1$ can be regarded as an operator from $H(\curl, \omega_{\sigma}^{[1]};\mathbb R^3)$ to $\Lambda_h^1(\omega_{\sigma}^{[1]})$.
However, the operator $\widehat \pi^1$ is not a projection operator, which needs the following modification. Consider the following diagram,
\begin{equation*}
	\xymatrix{
	&  & H(\curl, \omega; \mathbb R^3) \ar@<-1mm>@{.>}[dl]|-{\cR^1_{\omega}} \ar[r]^-{\curl} \ar@<-1mm>[d]^-{\mathcal Q^1_{\omega}} & H(\div,\omega;\mathbb R^3)\ar@<-1mm>@{.>}[dl]^-{\mathcal R^2_{\omega}}   \\
	\mathbb R \ar[r] & \Lambda_h^0(\omega) \ar[r]_-{\grad} & \Lambda_h^1(\omega) \ar[r]_-{\curl}  &\Lambda_h^2(\omega). 
	}
\end{equation*}

Define the following modified interpolation of $\widehat{\pi}^1 \bm \xi$, for any $\bm \xi \in H(\curl,\Omega;\mathbb R^3)$, 
\begin{equation}\label{eq:pi1-3d}
	\begin{split}
		\pi^1 \bm \xi = &\widehat{\pi}^1 \bm \xi + \sum_{\sigma} L _{\sigma}^1(I - \widehat{\pi}^1) \cQ_{\sigma,[1]}^1  \bm \xi + \sum_{\bm e} ( (I - \widehat{\pi}^1) \cQ_{\bm e,[1]}^1 \bm \xi,\bm t)_{\bm e}\varphi_{\bm e}, 
	\end{split}
\end{equation}
where $\cQ^1_{\sigma, [\ell]}$ is the harmonic interpolation operator $\cQ^1$ with respect to the patch $\omega_{\sigma}^{[\ell]}$. 
Note that the value $\widehat{\pi}^1\bm \xi$ on $\omega_{\sigma}$ only depends on the value of $\bm \xi$ on $\omega_{\sigma}^{[1]}.$ Therefore, it follows from Assumption (B4) that $\pi^1$ in \eqref{eq:pi1-3d} is well-defined. We first show that $\pi^1$ is a projection operator. For any $\bm \xi \in \Lambda_h^1$, $(\cQ_{\sigma,[1]} \bm \xi )|_{\omega_{\sigma}^{[1]}}= \bm \xi|_{\omega_{\sigma}^{[1]}}$ according to the definition of the harmonic projection, and hence
$$ \pi^1 \bm \xi = \widehat \pi^1 \bm \xi + \sum_{\sigma} L _{\sigma}^1(I - \widehat{\pi}^1) \bm \xi + \sum_{\bm e} ( (I - \widehat{\pi}^1)\bm \xi,\bm t)_{\bm e}\varphi_{\bm e} = \widehat \pi^1 \bm \xi+ (I - \widehat \pi^1) \bm \xi = \bm \xi.$$

It follows from the definition 
$\mathcal Q_{\omega}^1 \bm \xi = \nabla \mathcal R_{\omega}^1\bm \xi + \mathcal R_{\omega}^2 \curl \bm \xi,$
that on patch $\omega_{\sigma}$
\begin{equation*}
	\begin{split}
		(I - \widehat \pi^1) \cQ_{\sigma, [1]}^1 \bm \xi  = &  (I - \widehat \pi^1 ) \nabla \cR^1_{\sigma, [1]} \bm \xi + (I - \widehat \pi^1) \cR_{\sigma, [1]}^2\curl \bm \xi \\
		= & \nabla (I - \pi^0)\cR^1_{\sigma, [1]} \bm \xi + (I - \widehat \pi^1) \cR_{\sigma, [1]}^2\curl \bm \xi \\
		= & (I - \widehat \pi^1) \cR_{\sigma, [1]}^2\curl \bm \xi.
	\end{split}
\end{equation*}

Therefore, the projection $\pi^1\bm \xi$ can be simplified as 
		$$\pi^1\bm \xi =  \widehat{\pi}^1 \bm \xi + \sum_{\sigma} L _{\sigma}^1(I - \widehat{\pi}^1) \cR_{\sigma,[1]}^2  \curl \bm \xi + \sum_{\bm e} ( (I - \widehat{\pi}^1) \cR_{\bm e,[1]}^2 \curl \bm \xi, \bm t)_{\bm e}\varphi_{\bm e}.$$
This and the fact that $\widehat{\pi}^1 \nabla u = \nabla \pi^0 u$ lead to the crucial commuting property: $\pi^1 \nabla  u = \nabla \pi^0 u$ for $u \in H^1(\Omega)$. The value of $\pi^1 \bm\xi$ on $\omega_{\sigma}$ is determined by the value of $\bm\xi$ on $\omega^{[2]}_{\sigma}$.

\subsubsection{Construction of $\pi^2$}
Now consider the construction of $\pi^2$. Taking $\curl$ in \eqref{eq:hatpi1-3d} yields that 
\begin{equation}\label{eq:curlhatpi1-3d}
\curl \hat{\pi}^1 \bm \xi = \curl \avg^1 \bm \xi = \avg^2 \curl \bm \xi,
\end{equation}
where the latter equation comes from \Cref{prop:dblcmplx}. As a result, 
\begin{equation}\label{eq:curlpi1-3d} \curl \pi^1\bm \xi = \avg^2 \curl \bm \xi + \sum_{\sigma} \curl L _{\sigma}^1(I - \widehat{\pi}^1) \cR_{\sigma,[1]}^2  \curl \bm \xi + \sum_{\bm e} ( (I - \widehat{\pi}^1) \cR_{\bm e,[1]}^2 \curl \bm \xi, \bm t)_{\bm e}\curl \varphi_{\bm e}.\end{equation}
Motivated by this, for $\bm v \in H(\div,\Omega; \mathbb R^3)$, define
\begin{equation}
	\label{eq:hatpi2v-3d}
\widehat{\pi}^2\bm v= \avg^2 \bm v + \sum_{\sigma} \curl L_{\sigma}^1 (I - \widehat{\pi}^1)\cR^{2}_{\sigma,[1]} \bm v + \sum_{\bm e} ( (I - \widehat{\pi}^1)\cR^{2}_{\bm e,[1]} \bm v ,\bm t)_{\bm e} \curl \varphi_{\bm e}.
\end{equation}
Then it follows from {\Cref{prop:dblcmplx}, \eqref{eq:curlhatpi1-3d} and \eqref{eq:pi1-3d} that } $\widehat{\pi}^2 \curl \bm \xi = \curl \pi^1 \bm \xi$, and that the value of $\widehat \pi^2 \bm v$ on $\omega_{\sigma}$ only depends on the value of $\bm v$ on $\omega_{\sigma}^{[2]}$. Therefore, the operator $\widehat{\pi}^2$ can be regarded as from $H(\div,\omega_{\sigma}^{[2]};\mathbb R^3)$ to $\Lambda_h^2(\omega_{\sigma}^{[2]})$, when only its value on $\omega_{\sigma}$ is considered. The modification of $\widehat{\pi}^2$ to a projection operator is similar as that in the construction of $\pi^1$, which is based on the following diagram,

\begin{equation*}
	\xymatrix{
	&  & H(\div, \omega; \mathbb R^3) \ar@<-1mm>@{.>}[dl]|-{\cR^2_{\omega}} \ar[r]^-{\div} \ar@<-1mm>[d]^-{\mathcal Q^2_{\omega}} & L^2(\omega)\ar@<-1mm>@{.>}[dl]|-{\mathcal R^3_{\omega}}  & \\
	\Lambda_h^0(\omega) \ar[r] & \Lambda_h^1(\omega) \ar[r]^-{\curl} & \Lambda_h^2(\omega) \ar[r]^-{\div}  &\Lambda_h^3(\omega)  & .
	}
\end{equation*}

It follows from the definition 
$\mathcal Q_{\omega}^2 \bm v = \curl \mathcal R_{\omega}^2 \bm v + \mathcal R_{\omega}^3 \div \bm v,$
that on patch $\omega_{\sigma}$
\begin{equation*}
	\begin{split}
		(I - \widehat \pi^2) \cQ_{\sigma, [2]} \bm v  = &  (I - \widehat \pi^2 ) \curl \cR^2_{\sigma, [2]} \bm v + (I - \widehat \pi^2) \cR_{\sigma, [2]}^3\div \bm v \\
		= & \curl (I - \pi^1)\cR^2_{\sigma, [2]} \bm v + (I - \widehat \pi^2) \cR_{\sigma, [2]}^3\div \bm v \\
		= & (I - \widehat \pi^2) \cR_{\sigma, [2]}^3\div \bm v.
	\end{split}
\end{equation*}

Therefore, for $\bm v \in H(\div, \Omega; \mathbb R^3)$, define
\begin{equation}\label{eq:pi2-3d}
	\begin{split}
		\pi^2 \bm v = &~ \widehat{\pi}^2 \bm v + \sum_{\sigma} L _{\sigma}^2(I - \widehat{\pi}^2) \cQ^{2}_{\sigma, [2]} \bm v + \sum_{\bm f} ( (I - \widehat{\pi}^2) \cQ^{2}_{\sigma, [2]}\bm v,\bm n)_{\bm f}\varphi_{\bm f} \\ = & ~\widehat{\pi}^2 \bm v + \sum_{\sigma} L _{\sigma}^2(I - \widehat{\pi}^2) \cR^{3}_{\sigma, [2]} \div \bm v + \sum_{\bm f} ( (I - \widehat{\pi}^2) \cR^{3}_{\sigma, [2]}\div \bm v,\bm n)_{\bm f}\varphi_{\bm f}.
	\end{split}
\end{equation}
This shows that $\pi^2 \curl \bm \xi = \curl \pi^1 \bm \xi$ for all $\bm \xi \in H(\curl,\Omega;\mathbb R^3)$.
It follows from $\cQ_{\sigma,[2]}^2$ is {a projection operator on $\omega_{\sigma}^{[2]}$ that $\pi^2$ is also a projection operator, namely, $\pi^2 \bm v = \bm v$ for $\bm v \in \Lambda_h^2$. Moreover, the value of $\pi^2 \bm v$ on $\omega_{\sigma}$ is determined by the value of $\bm v$ on $\omega^{[3]}_{\sigma}$.

\subsubsection{Construction of $\pi^3$}
Finally, given $p \in L^2(\Omega)$, define
\begin{equation}\label{eq:pi3p-3d} \pi^3 p  = \avg^3 p + \sum_{\sigma} \div  L_{\sigma} (I - \widehat{\pi}^{2}) \cR^{3}_{\sigma, [2]} p+  \sum_{\bm f} ( (I - \widehat{\pi}^2) \cR^{3}_{\sigma, [2]} p,\bm n)_{\bm f}\operatorname{div}\varphi_{\bm f}.
\end{equation}
Then it follows from \Cref{prop:dblcmplx}, \eqref{eq:pi2-3d} and the definition of $\pi^3$ that $\pi^3 \div \bm v = \div \pi^2 \bm v$, and the commuting property immediately comes from the surjectivity of $\div$ from $H(\div, \Omega; \mathbb R^3)$ to $L^2(\Omega)$. In addition, $\pi^3$ is also a projection operator since $\div : \Lambda_h^2\to \Lambda_h^3$ is surjective. In fact, for $p \in \Lambda_h^3$, there exists $\bm v \in \Lambda_h^2$ such that $\div \bm v = p$ and therefore $\pi^3 p = \pi^3 \div \bm v = \div \pi^2 \bm v = \div \bm v = p.$ Note that, the value of $\pi^3 p$ on $\omega_{\sigma}$ is determined by the value of $\bm p$ on $\omega^{[3]}_{\sigma}$.

\subsection{Local Estimates in Three Dimensions}
In this subsection, the local estimates on Theorem~\ref{thm:main-3d} are proved.
Hereafter, the notation $a \lesssim b$ means $a \le Cb$ for some positive constant $C$,  where the hidden constant might depend on the shape regularity and the polynomial degree of the finite element complex.

The following lemma gives the estimates of the auxiliary operators constructed in \Cref{sec:auxproj}.
\begin{lemma}\label{lem:bounds-3d}
	For three dimensions, the following local estimates hold
	\begin{equation}
		\label{eq:bounds-3d-Qk}
		\|\cQ^0_{\omega} u\|_{H^1(\omega)} \lesssim \|u\|_{H^1(\omega)}, \|\cQ^1_{\omega} \bm \xi \|_{H(\curl, \omega)} \lesssim \| \bm \xi  \|_{H(\curl, \omega)}, \|\cQ^2_{\omega} \bm v \|_{H(\div, \omega)} \lesssim \| \bm v  \|_{H(\div, \omega)},
	\end{equation}
	and
	\begin{equation}
		\label{eq:bounds-3d-Rk}
		\|\cR^1_{\omega} \bm \xi\|_{H^1(\omega)} \lesssim \| \bm \xi \|_{H(\curl,\omega)}, \|\cR^2_{\omega} \bm v\|_{H(\curl, \omega)} \lesssim \|\bm v \|_{H(\div, \omega)}, \|\cR^3_{\omega} p\|_{H(\div,\omega)} \lesssim \|p\|_{L^2(\omega)}
	\end{equation}
	for $u \in H^1(\omega)$, $\bm \xi \in H(\curl,\omega ; \mathbb R^3)$, $\bm v \in H(\div,\omega ; \mathbb R^3)$ and $p \in L^2(\omega)$. Here the constant depends on only the shape of $\omega$ and the finite element method under consideration. If $\omega = \omega_{\sigma}^{[\ell]}$ for some $\ell \ge 0$, then the constant is uniformly bounded and depends on the shape regularity, $\ell$, and the finite element under consideration.
\end{lemma}

\begin{proof}
	Take $\cQ_{\omega}^1$ and $\cR_{\omega}^1$ as the examples, and the other estimates can be proved similarly.
	Since $\Lambda_h^k(\omega)$ is of finite dimension and $\mathcal Q_{\omega}^1$ is a projection under the inner product $\LL \cdot, \cdot \RR_{\Lambda_h^1(\omega)}$, it follows from the norm equivalence theorem that 
	$$ \|\cQ_{\omega}^1 \bm \xi \|_{H(\curl,\omega)}^2 \le C \LL \cQ_{\omega}^1  \bm \xi,  \cQ_{\omega}^1  \bm \xi \RR_{\Lambda_h^1(\omega)} \le C \LL  \bm \xi, \bm \xi \RR_{\Lambda_h^1(\omega)}.
	$$
	Here the constant $C$ depends on $\Lambda_h^1(\omega)$.
	Since $\mathcal P_{\grad \Lambda_h^0 (\omega)}$ is an $L^2$ projection, it holds that 
	$$ 
	\LL  \bm \xi, \bm \xi \RR_{\Lambda_h^1(\omega)}  = (\mathcal P_{\grad \Lambda_h^0 (\omega)} \bm \xi, \mathcal P_{\grad \Lambda_h^0 (\omega)} \bm \xi )_{\omega} + (\curl \bm \xi, \curl \bm \xi)_{\omega} \le \|\bm \xi\|_{H(\curl,\omega)}^2.
	$$
	As a result, it follows that $\|\cQ_{\omega}^1 \bm \xi \|_{H(\curl,\omega)} \le C \|\bm \xi\|_{H(\curl,\omega)}$, which completes the proof of \eqref{eq:bounds-3d-Qk}.

	Since $\cR_{\omega}^1 \bm \xi  \perp \mathbb R$, it then follows from the norm equivalence theorem on $\Lambda_h^0(\omega)$ that 
	$$\|\cR_{\omega}^1 \bm \xi\|_{H^1(\omega)} \le C \| \grad \cR_{\omega}^1 \bm \xi \|_{L^2(\omega)}  \le C\|\cQ_{\omega}^1 \bm \xi\|_{L^2(\omega)},
$$ 
	which indicates the estimation of \eqref{eq:bounds-3d-Rk}. Here, the constant $C$ is also dependent on $\Lambda_h^0(\omega)$.
	
	Finally, the standard scaling and compactness argument (cf. \cite{2015FalkWinther}) completes the proof.

\end{proof}

Under the assumptions of Theorem~\ref{thm:main-2d}, all $L_{\sigma}^k$ defined on the finite element spaces $\Lambda_h^k, k = 0,1,2,3,$ are local bounded, namely 
\begin{equation}\label{eq:dofbound-3d}
	\begin{split}
	\|L_{\sigma}^0 u \|_{H^1(\omega_{\sigma})} \lesssim \| u \|_{H^1(\omega_{\sigma})},\,\,\|L_{\sigma}^1 \bm \xi \|_{H(\curl,\omega_{\sigma})} \lesssim \| \bm \xi \|_{H(\curl,\omega_{\sigma})}, \\ 
	\|L_{\sigma}^2 \bm v \|_{H(\div,\omega_{\sigma})} \lesssim \| \bm v \|_{H(\div, \omega_{\sigma})},\,\,\|L_{\sigma}^3 p \|_{L^2(\omega_{\sigma})} \lesssim \| p \|_{L^2(\omega_{\sigma})},
	\end{split}
\end{equation}
and 
\begin{equation}\label{eq:dofbound2-3d}
\begin{split}
\| u(\bm x)\varphi_{\bm x}\|_{H^1(\omega_{\sigma})} \lesssim \| u \|_{H^1(\omega_{\sigma})}, \,\, 
\| (\bm \xi \cdot \bm t, 1)_{\bm e} \varphi_{\bm e}\|_{H(\curl,\omega_{\sigma})} \lesssim \| \bm \xi \|_{H(\curl,\omega_{\sigma})}, \\
\| (\bm v \cdot \bm n, 1)_{\bm f} \varphi_{\bm f}\|_{H(\div,\omega_{\sigma})} \lesssim \| \bm v \|_{H(\div, \omega_{\sigma})},\,\,
\| (p,1)_{\bm K} \varphi_{\bm K}\|_{L^2(\omega_{\sigma})} \lesssim \| p \|_{L^2(\omega_{\sigma})},
\end{split}
\end{equation}
for finite element functions $u \in \Lambda_h^0, \bm \xi \in \Lambda_h^1, \bm v \in \Lambda_h^2$ and $p \in \Lambda_h^3$. Here the universal constant comes from the shape regularity and the dimension of the corresponding local finite element spaces. 

The following estimates are based on \Cref{lem:bounds-3d}, \eqref{eq:dofbound-3d} and \eqref{eq:dofbound2-3d}, and the following lemmas bounding the operators $\mathcal M^k$ for $k = 0,1,2,3$.

\begin{lemma}
\label{lem:boundsM-3d}
If $\mathcal T$ is shape-regular, then 
\begin{equation}
\begin{split}
	\| \mathcal M^0 u\|_{H^1(\bm K)} \lesssim \| u \|_{H^1(\omega_{\bm K}^{[1]})}, \,\,
	\| \mathcal M^1 \bm \xi \|_{H(\curl,\bm K)} \lesssim \| \bm \xi \|_{H(\curl,\omega_{\bm K}^{[1]})},\\
	\| \mathcal M^2 \bm v \|_{H(\div,\bm K)} \lesssim \|\bm v\|_{H(\div,\omega_{\bm K}^{[1]})},\,\,
	\| \mathcal M^3 p\|_{L^2(\bm K)} \lesssim \| p\|_{L^2(\omega_{\bm K}^{[1]})}.
\end{split}
\end{equation}
Here the constants depend only on $\omega$ and the finite element space.
\end{lemma}
\begin{proof}
See \cite{2015FalkWinther, 2014FakWinther}.
\end{proof}

Now first consider the estimation of $\pi^0$.
\begin{lemma} 
For $u \in H^1(\omega_{\bm K}^{[1]})$, it holds that $\|\pi^0 u\|_{H^1(\bm K)} \lesssim \|u\|_{H^1(\omega_{\bm K}^{[1]})}.$
\end{lemma}
\begin{proof}
	It follows from the definition in \eqref{eq:pi0u-3d}, \eqref{eq:dofbound-3d}, \eqref{eq:bounds-3d-Qk}, and the facts that $\omega_{\sigma} \subset \omega_{\bm K}^{[1]}$ for any subsimplex $\sigma$ of $\bm K$, and the number of subsimplices and vertices are uniformly bounded (depending on the shape regularity) that
\begin{equation}
	\begin{split}
		\|\pi^0u\|_{H^1(\bm K)} & \le \sum_{\sigma \in \bm K} \|L_{\sigma}^0 \cQ_{\sigma}^0u \|_{H^1(\bm K)} + \sum_{\bm x \in \bm K }\|(\cQ _{\bm x}^0 u)(\bm x) \varphi_{\bm x} \|_{H^1(\bm K)} \\
		& \lesssim  \sum_{\sigma \in \bm K} \|\cQ_{\sigma}^0u\|_{H^1(\omega_{\sigma})} + \sum_{\bm x \in \bm K} \|\cQ_{x}^0u\|_{H^1(\omega_{\bm x})} \\
		& \lesssim \|u\|_{H^1(\omega_{\bm K}^{[1]})}.
	\end{split}
\end{equation}
This completes the proof.

\end{proof}

Next, consider the estimation of ${\pi}^1$.

\begin{lemma} For $\bm \xi \in H(\curl, \omega_{\bm K}^{[2]}; \mathbb R^3)$, it holds that
	$\|\pi^1\bm \xi\|_{H(\curl, \bm K)} \lesssim \|\bm \xi\|_{H(\curl,\omega_{\bm K}^{[2]})}.$
\end{lemma}

\begin{proof}
	First, consider the boundness of $\hat{\pi}^1$ defined in \eqref{eq:hatpi1-3d}. Clearly it follows from \Cref{lem:boundsM-3d} that the last term in the construction \eqref{eq:hatpi1-3d} can be estimated as 
$$ \|\avg^1 \bm \xi\|_{H(\curl,\bm K)} \lesssim \|\bm \xi\|_{H(\curl,\omega_{\bm K}^{[1]})}.$$

By \eqref{eq:dofbound-3d} and \eqref{eq:bounds-3d-Rk}, the first term in \eqref{eq:hatpi1-3d} can be estimated as 
$$ \| \nabla L_{\sigma}^0 \mathcal R_{\sigma}^1 \bm \xi\|_{H(\curl, \bm K)} \le \| L_{\sigma}^0 \mathcal R_{\sigma}^1 \bm \xi\|_{H^1(\bm K)} \lesssim \|\mathcal R_{\sigma}^1 \bm \xi\|_{H^1(\bm \omega_{\sigma})} \lesssim \|\bm \xi\|_{H(\curl,\omega_{\sigma})}$$ for any $\sigma \in \bm K$,
and similarly the second term can be analyzed as
$$\| (\mathcal R_{\bm x}^1 \bm \xi)(\bm x)\nabla \varphi_{\bm x}\|_{H(\curl, \bm K)} \lesssim \|\bm \xi\|_{H(\curl,\omega_{\bm x})}.$$
Therefore, it holds that
\begin{equation}
	\|\widehat \pi^1 \bm \xi \|_{H(\curl,\bm K)} \lesssim \| \bm \xi\|_{H(\curl,\omega_{\bm K}^{[1]})},
\end{equation}
and hence
\begin{equation}
	\|(I - \widehat \pi^1) \bm \xi \|_{H(\curl,\bm K)} \lesssim \| \bm \xi\|_{H(\curl,\omega_{\bm K}^{[1]})}.
\end{equation}

It is now ready to bound $\pi^1 \bm \xi$ of \eqref{eq:pi1-3d}. Since for all subsimplices $\sigma$ of $\bm K$ it holds that $\omega_{\bm K}^{[1]} \subset \omega_{\sigma}^{[1]}$, it follows that the second term of \eqref{eq:pi1-3d} is estimated as
\begin{equation}
	\begin{split}
		\|L _{\sigma}^1(I - \widehat{\pi}^1) \cQ_{\sigma,[1]}^1  \bm \xi\|_{H(\curl, \bm K)} & \lesssim \|(I - \widehat{\pi}^1) \cQ_{\sigma,[1]}^1  \bm \xi\|_{H(\curl, \omega_{\sigma})} \\
		& \lesssim \|\cQ_{\sigma,[1]}^1 \bm \xi\|_{H(\curl, \omega_{\sigma}^{[1]})} \\
		& \lesssim \|\bm \xi\|_{H(\curl,\omega_{\sigma}^{[1]})},
	\end{split}
\end{equation}
and similarly, the third term is bounded as
\begin{equation}
	\| ( (I - \widehat{\pi}^1) \cR_{\bm e,[1]}^2 \curl \bm \xi,\bm t)_{\bm e}\varphi_{\bm e} \|_{H(\curl, \bm K)} \lesssim \|\bm \xi \|_{H(\curl,\omega_{e}^{[1]})}
\end{equation}
As a result,
\begin{equation}
	\|\pi^1\bm \xi\|_{H(\curl, \bm K)} \lesssim \|\bm \xi\|_{H(\curl,\omega_{\bm K}^{[2]})}.
\end{equation}
This completes the proof.
\end{proof}
Next, consider the operator $\pi^2$ determined in \eqref{eq:pi2-3d}.

\begin{lemma} For $\bm v \in H(\div, \omega_{\bm K}^{[3]}; \mathbb R^3)$, it holds that 
	$\|\pi^2\bm v\|_{H(\div, \bm K)} \lesssim \|\bm v\|_{H(\div,\omega_{\bm K}^{[3]})}.$
\end{lemma}

\begin{proof}
First, given $\bm v \in H(\div, \Omega; \mathbb R^3)$, it holds that the first term of \eqref{eq:hatpi2v-3d} is bounded as
\begin{equation}
	\|\avg^2 \bm v\|_{H(\div, \bm K)} \lesssim \|\bm v\|_{H(\div, \omega_{\bm K}^{[1]})},
\end{equation}
and that the second term of \eqref{eq:hatpi2v-3d} is analyzed as
\begin{equation}
	\begin{split}
		\| \curl L_{\sigma}^1 (I - \widehat{\pi}^1)\cR^{2}_{\sigma,[1]} \bm v\|_{H(\div, \bm K)} & \lesssim \|L_{\sigma}^1 (I - \widehat{\pi}^1) \cR^{2}_{\sigma,[1]}  \bm v\|_{H(\curl, \bm K)} \\
		& \lesssim \|(I - \widehat{\pi}^1) \cR^{2}_{\sigma,[1]}  \bm v \|_{H(\curl, \omega_{\bm \sigma})}\\
		& \lesssim \|\cR^{2}_{\sigma,[1]} \bm v\|_{H(\curl, \omega_{\sigma}^{[1]})} \\
		& \lesssim \|\bm v\|_{H(\div, \omega_{\sigma}^{[1]})}.
	\end{split}
\end{equation}
Similarly, the last term of $\hat{\pi}^2 \bm v$ can be estimated as
$$\|((I - \widehat{\pi}^1)\cR^{2}_{\bm e,[1]}\bm v \cdot \bm t,1)_{\bm e} \curl \varphi_{\bm e}\|_{H(\div, \bm K)} \lesssim \|\bm v\|_{H(\div, \omega_{\bm e}^{[1]})}.$$
Therefore, a summation of the above three estimates yields
\begin{equation}
	\|\widehat{\pi}^2 \bm v \|_{H(\div, \bm K)} \lesssim \|\bm v\|_{H(\div, \omega_{\bm K}^{[2]})}.
\end{equation}
Now, consider the estimates of $\pi^2\bm v$ defined in \eqref{eq:pi2-3d}. Since the second term of \eqref{eq:pi2-3d} is bounded by 
\begin{equation}
	\begin{split}
		\| L_{\sigma}^2(I - \widehat{\pi}^2) \cQ^2_{\sigma,[2]}  \bm v\|_{H(\div, \bm K)} & \lesssim \|(I - \widehat{\pi}^2)  \cQ^2_{\sigma,[2]}   \bm v\|_{H(\div, \omega_{\sigma})} \\
		& \lesssim \| \cQ^2_{\sigma,[2]}   \bm v\|_{H(\div, \omega_{\sigma}^{[2]})} \\
		& \lesssim \|\bm v\|_{H(\div, \omega_{\sigma}^{[2]})}, 
	\end{split}
\end{equation}
and third term of \eqref{eq:pi2-3d} is estimated as
\begin{equation}
	\|( (I - \widehat{\pi}^2)\cQ^2_{\sigma,[2]}  \bm v,\bm n)_{\bm f}\bm \varphi_{\bm f}\|_{H(\div,\bm K)} \lesssim \|\bm v\|_{H(\div, \omega_{\bm f}^{[2]})}.
\end{equation}
It follows that
\begin{equation}
	\| \pi^2 \bm v \|_{H(\div, \bm K)} \lesssim \| \bm v \|_{H(\div,\omega_{\bm K}^{[3]})}.
\end{equation}
\end{proof}

Finally, consider the estimation of $\pi^3$ defined in \eqref{eq:pi3p-3d}.
\begin{lemma} For $p \in L^2(\omega_{\bm K}^{[3]})$, it holds that
	$\|\pi^3p\|_{L^2(\bm K)} \lesssim \|p\|_{L^2(\omega_{\bm K}^{[3]})}.$
\end{lemma}
\begin{proof}
The estimate of $\pi^3p$ can be derived from that of $\pi^2\bm v$, and the surjectivity of the div operator. Since for $p \in L^2(\omega_{\bm K}^{[3]})$ there exists $\bm v \in H(\div, \omega_{\bm K}^{[3]}; \mathbb R^3)$ such that $\div \bm v = p$. Hence, it follows that
\begin{equation}
	\begin{split}
		\|\pi^3 p\|_{L^2(\bm K)} = \| \pi^3 \div \bm v \|_{L^2(\bm K)} = \| \div \pi^2 \bm v \|_{L^2(\bm K)}
		\lesssim \| \bm v \|_{H(\div,\omega_{\bm K}^{[3]})} \lesssim \| \bm v \|_{L^2(\omega_{\bm K}^{[3]})}.
	\end{split}
\end{equation}
\end{proof}
\section{Applications in Two Dimensions}
\label{sec:2D}
Without otherwise specified, the exactness of the finite element complexes discussed in this section can be referred to \cite{2018Arnold}, and \cite{2018ChristiansenHuHu}. Therefore, Assumptions (A1) and (A4) hold for the applications in this section. Some modifications are performed when introducing the degrees of freedom of the elements. Let $\lambda_0,\lambda_1,\lambda_2$ be the barycenter coordinates of $\bm f$, and $\lambda_0,\lambda_1$ be the barycenter coordinates of $\bm e$.

In what follows, the non-bold notation $P_k(\bm f)$ stands for the space of polynomials of degree $\le k$ over $\bm f$, and the non-bold  notation $RT_k(\bm f) = [P_k(\bm f)]^2 + [x_1,x_2]^T P_k(\bm f)$ stands for the Raviart--Thomas element shape function space over $\bm f$. For convenience, both ends of the complex are omitted. For the standard finite element complexes, $\mathbb R \to $ and $\to 0$ are omitted, and for the bubble complexes, $0 \to $ and $\to 0$ are omitted. For clarity, the degrees of freedom $u(\bm x)$, $(\bm v\cdot \bm n,1)_{\bm e}$, $(p,1)_{\bm f}$ will be put first when introducing the degrees of freedom.

\subsection{The Lagrange--Raviart--Thomas Complex}
\label{sec:lrt}
This subsection considers the interpolation operators of the Lagrange--Raviart--Thomas complex \cite{2006ArnoldFalkWinther} for $k \ge 1$:
\begin{equation}
	\tag{{\sf LRT}}
	\mathbf P_{k} \xrightarrow{\curl} \mathbf{RT}_{k-1} \xrightarrow{\div} \mathbf P_{k-1}^{-}.
\end{equation}
Here
$$\mathbf P_k := \{ u \in C^0(\Omega) :  u|_{\bm f} \in P_k(\bm f), \, \forall \bm f \in \mathcal T\},$$
$$\mathbf{RT}_{k-1} := \{ \bm v \in H(\div, \Omega; \mathbb R^2) : \bm v|_{\bm f}\in RT_{k-1}(\bm f), \, \forall \bm f \in \mathcal T \},$$and
$$\mathbf P_{k-1}^- := \{ p \in L^2(\Omega) : p|_{\bm f} \in P_{k-1}(\bm f),\, \forall \bm f \in \mathcal T\}.$$

For the lowest order case with $k = 1$, there are no other degrees of freedom but $u(\bm x), (\bm v\cdot \bm n,1)_{e},  (p,1)_{\bm f}$, for $u \in \mathbf P_k$, $\bm v \in \mathbf{RT}_{k-1}$, $p \in \mathbf{P}_{k-1}^{-}$, respectively. Hence, Assumptions (A1)-(A4) are reduced to the exactness of the finite element complex, namely, Assumption (A1). 

For the higher-order case with $k > 1$, the corresponding degrees of freedom are given below.

The degrees of freedom of the $H^1$ finite element space $\mathbf P_k$ are defined by:
for $u \in P_k(\bm f)$, define the degrees of freedom as, for any $\bm f$,
\begin{enumerate}
	\item[-] the function value $u(\bm x)$ at each vertex $\bm x$ of $\bm f$, and no other degrees of freedom $f_{\bm x}^0$ are imposed;
	\item[-] on each edge $\bm e$ of $\bm f$, set
		$ f_{\bm e,i}^0(u) = (\frac{\partial u}{\partial \bm t},\frac{\partial b_{\bm e}^i}{\partial \bm t} )_{\bm e},$
		where $b_{\bm e}^i, i = 1,\cdots, k-1$, form a basis of the edge bubble function space $B_{\bm e, k-2} := \{ p \in P_k(e): p({\bm x}) = p({\bm y}) = 0\} = \lambda_0\lambda_1 P_{k-2}(\bm e),$ where $\bm x$ and $\bm y$ are the two endpoints of $\bm e $;
	\item[-] in $\bm f$, set
		$f_{\bm f, i}^0(u) = (\curl u,\curl b_{\bm f}^i),$
		where $b_{\bm f}^i, i = 1,2,\cdots, \frac{1}{2}(k-2)(k-1)$, form a basis of the face bubble function space $B_{\bm f,k-3} = \{ p \in P_k(f): p|_{\bm e} = 0, \forall \bm e \subset \partial \bm f \}= (\lambda_0\lambda_1\lambda_2) P_{k-3}(\bm f).$
\end{enumerate}

The degrees of freedom of the $H(\div)$ finite element space $\mathbf{RT}_{k-1}$ are defined by: for $\bm v \in RT_{k-1}(\bm f)$, define the degrees of freedom as follows, for any $\bm f$,
\begin{itemize}
	\item[-] the moment of the normal component $(\bm v\cdot \bm n,1)_{\bm e}$ for each edge $\bm e$ of $\bm f$;
	\item[-] on each edge $\bm e$ of $\bm f$, set
		$f_{\bm e, i}^1(\bm v) = (\bm v \cdot \bm n, \frac{\partial b_{\bm e}^i}{\partial \bm t})_{\bm e},$
		where $\frac{\partial b_{\bm e}^i}{\partial \bm t} , i = 1,2,\cdots, k-1$, form a basis of $P_{k-1}(\bm e)/\R$;
	\item[-] in $\bm f$, set
		$f_{\bm f, i}^1(\bm v) = (\bm v, \curl b_{\bm f}^i)_{\bm e},$
		where $\curl b_{\bm f}^i, i = 1,\cdots, \frac{1}{2}(k-2)(k-1) $, form a basis of $\curl B_{\bm f,k-3}$
		and
		$f_{\bm f, i}^1(\bm v) = (\div \bm v, q_i)_{\bm f}$
		where $q_i, i =  \frac{1}{2}(k-1)(k-2) + 1,\cdots, \frac{1}{2}(k-1)(k-2) + \frac{1}{2}k(k+1)-1$, form a basis of $P_{k-1}(\bm f) / \R$.
\end{itemize}

The degrees of freedom of the $L^2$ element space $\mathbf{P}_{k-1}^-$ are defined by: for $p \in P_{k-1}(\bm f)$, define the degrees of freedom as follows, for any $\bm f$,
\begin{itemize}
	\item[-] the moment on of $p$ in each face $\bm f$, namely, $(p,1)_{\bm f}$;
	\item[-] let $f_{\bm f,i}^2(p) = (p,q_i)_{\bm f}$, where $q_i, i =1,2,\cdots,\frac{1}{2}k(k+1)-1$, form a basis of $P_{k-1}(\bm f) / \R$.
\end{itemize}

We now verify Assumptions (A1)-(A4). Clearly, all $f_{\sigma}^0$ satisfy that $L_{\sigma}^0 (\R) = 0$. 
For any $\xi_{\bm e}^i \in P_{k-1}(\bm e)/\R$, there exists $b \in B_{\bm e,k-2}$ such that $\frac{\partial b}{\partial \bm t} = \xi_{\bm e}^i$. Therefore, $f_{\bm e, i}^1(\curl \varphi_{\bm x}) = (\curl \varphi_{\bm x} \cdot \bm n, \frac{\partial b}{\partial \bm t}) = 0,$ since $\varphi_{\bm x}$ vanishes for the degrees of freedom $f_{\bm e}^0$. Also, for $i \leq \frac{1}{2}(k-2)(k-1)$, it holds that $f_{\bm f,i}^1(\curl \varphi_{\bm x}) = (\curl \varphi_{\bm x}, \curl b_{\bm f, i}^1) = 0$. While for other $i$, $f_{\bm f,i}^1(\curl \varphi_{\bm x}) = 0$ since $\div\curl \varphi_{\bm x} = 0$. The verification of $f_{\bm f,i}^2(\div \varphi_{\bm e}) = 0$ is straightforward.
% It follows from Lemma~\ref{lem:verification2d} that $f_{\sigma,i}^1(\curl \varphi_{\bm x}) = 0$ for any $\sigma$ and $\bm x$.
% Next, it follows from the definition that $f_{\bm f,i}^2(\div \varphi_{\bm e}) = 0$, for any edge $\bm e$, face $\bm f$ and $i$.
% Therefore, the complex satisfies the assumptions.

\subsection{The Lagrange--Brezzi--Douglas--Marini Complex}
\label{sec:lbdm}
Consider the following finite element complex for $k \ge 2$,
\begin{equation}
	\label{eq:lbdm}
	\tag{{\sf LBDM}}
	\mathbf P_{k} \xrightarrow{\curl} \mathbf{BDM}_{k-1} \xrightarrow{\div}\mathbf P_{k-2}^{-},
\end{equation}
where $\mathbf{BDM}_{k-1} := \{ \bm v \in H(\div,\Omega;\mathbb R^2) : \bm v |_{\bm f} \in [P_{k-1}(\bm f)]^2,\,\forall \bm f \in \mathcal T \}.$ 

For the $H(\div)$ finite element space $\mathbf{BDM}_{k-1}$, the degrees of freedom are defined by: for $\bm v \in [P_k(\bm f)]^2$, define the degrees of freedom as, for any $\bm f$,
\begin{itemize}
	\item[-] the moment of the normal component on each edge $\bm e$ of $\bm f$, namely, $(\bm v\cdot \bm n,1)_{\bm e}$;
	\item[-] on each edge $\bm e$ of $\bm f$, set
		$f_{\bm e, i}^1(\bm v) = (\bm v \cdot \bm n, \frac{\partial b_{\bm e}^i}{\partial \bm t})_{\bm e}$,
		where $\frac{\partial b_{\bm e}^i}{\partial \bm t} , i = 1,2,\cdots, k-1$, form a basis of $P_{k-1}(\bm e)/\R$.
	\item[-] inside $\bm f$, set
		$f_{\bm f, i}^1(\bm v) = ( \bm v, \curl b_{\bm f}^i)_{\bm f},$
		where $\curl b_{\bm f}^i, i = 1,\cdots, \frac{1}{2}(k-2)(k-1) $, form a basis of $\curl B_{\bm f,k-3}$,
		and
		$f_{\bm f,i}^1(\bm v) = (\div \bm v, q_i)_{\bm f}$,
		where $q_i, i =  \frac{1}{2}(k-2)(k-1) +  1 ,\cdots,\frac{1}{2}(k-2)(k-1) +  \frac{1}{2}(k-1)k-1$, form a basis of $P_{k-2}(\bm f) / \R$.
\end{itemize}

It is easy to see that this complex satisfies Assumptions (A1)-(A4).

\subsection{The Hermite--Stenberg Complex}
\label{sec:hs}

In this subsection, we explore the Hermite-Stenberg complex, a finite element complex in two dimensions that enhances local regularity at vertices. Within our framework, the additional degrees of freedom at vertices can be treated as the same as those on edges and faces. The Hermite-Stenberg complex was initially introduced in \cite{2018ChristiansenHuHu} for $k\geq 3$,

\begin{equation}
	\label{eq:hs}
	\tag{{\sf HS}}
	\mathbf {Hm}_{k} \xrightarrow{\curl} \mathbf{St}_{k-1} \xrightarrow{\div} \mathbf P_{k-2}^{-},
\end{equation}
	where 
	$$\mathbf{Hm}_{k} := \{ u \in C^0(\Omega) : u|_{\bm f} \in P_k(\bm f), \forall \bm f \in \mathcal T, u \text{ is } C^1 \text{ at each vertex of } \mathcal T\},$$
	and  
	$$\mathbf{St}_{k-1} := \{ \bm v \in H(\div,\Omega;\mathbb R^2) : \bm v|_{\bm f} \in [P_{k-1}(\bm f)]^2, \forall \bm f \in \mathcal T, \bm v \text{ is } C^0  \text{ at each vertex of } \mathcal T\}.$$

Next, we demonstrate how this finite element complex can be incorporated into our proposed framework.

For the $H^1$ finite element space $\mathbf{Hm}_k$, the degrees of freedom are defined by: for $u \in P_k(\bm f)$, the degrees of freedom are defined as, for any $\bm f$, 
\begin{itemize}
	\item[-] the function value at each vertex $\bm{x}$ of $\bm f$;
	\item[-] at each vertex $\bm x$ of $\bm f$, define $f_{\bm x,1}^0(u) = \frac{\partial u}{\partial x}(\bm x)$ and  $f_{\bm x,2}^0(u) = \frac{\partial u}{\partial y}(\bm x)$;
	\item[-] on each edge $\bm e$ of $\bm f$, define $f_{\bm e,i}^0(u) = (\frac{\partial u}{\partial \bm t},\frac{\partial b_{\bm e}^i}{\partial \bm t} )_{\bm e}$
		where $b_{\bm e}^i, i = 1,\cdots, k-3$, form a basis of edge bubble $B^{1}_{\bm e,k-4} := \{ p \in P_k(\bm e): u(\bm x) = u'(\bm x) = u(\bm y) = u'(\bm y) = 0\} = (\lambda_0\lambda_1)^2 P_{k-4}(\bm e),$ where $\bm x$ and $\bm y$ are the two endpoints of $\bm e$;
	\item[-] inside $\bm f$, define
		$f_{\bm f, i}^0(u) = (\operatorname{curl}u,\operatorname{curl}b_{\bm f}^i)_{\bm{f}}$
		where $b_{\bm f}^i, i = 1,\cdots, \frac{1}{2}(k-1)(k-2) $, form a basis of $B_{\bm f, k-3}  := (\lambda_0\lambda_1\lambda_2) P_{k-3}(\bm f).$
\end{itemize}

For the $H(\div)$ Stenberg element space $\mathbf{St}_{k-1}$, the degrees of freedom are defined by: for $\bm v = [v_x, v_y]^T \in [P_{k-1}(\bm f)]^2$, define the degrees of freedom for any $\bm f$,
\begin{itemize}
	\item[-] the moment of the normal component on each edge $\bm v $ of $\bm f$, $(\bm v \cdot \bm n ,1)_{\bm e}$;
	\item[-] at each vertex $\bm x$ of $\bm f$, define $f_{\bm x, 0}^1(\bm v) =  v_{x}$ and $f_{\bm x,1}^1(\bm v) = v_y$, where $v_x$ and $v_y$ are two components of $\bm v$;
	\item[-] on each edge $\bm e$ of $\bm f$, set
		$f_{\bm e, i}^1(\bm v) = (\bm v \cdot \bm n, \frac{\partial b_{\bm e}^i}{\partial \bm t} )_{\bm e},$
		where $\frac{\partial b_{\bm e}^i}{\partial \bm t} , i = 1,2,\cdots, k-3$, form a basis of the space $B_{\bm e, k-3}/ \mathbb R$;
	\item[-] inside $\bm f$, set
		$f_{\bm f, i}^1(\bm v) = ( \bm v, \operatorname{curl}b_{\bm f}^i)_{\bm f},$
		where $\curl b_{\bm f}^i, i = 1,\cdots, \frac{1}{2}(k-2)(k-1) $, form a basis of $\curl B_{\bm f,k-3}$
		and
		$f_{\bm f,i}^1(\bm v) = (\div \bm v, q_i)_{\bm f}$
		where $q_i,i = \frac{1}{2}(k-2)(k-1) + 1 ,\cdots, \frac{1}{2}(k-2)(k-1) + \frac{1}{2}(k-1)k-1$, form a basis of $P_{k-2}(\bm f) / \R$.
\end{itemize}
For this case, Assumptions (A1)-(A4) can be similarly verified.

\subsection{A bubble complex viewpoint}
\label{sec:bubble-complex}

The rest of this section discusses how to verify Assumption (A2) conveniently.
A key point is the so-called bubble complex related to edges and faces, which will be introduced below. Define the face $H(\div)$ bubble
$$B^{BDM}_{\bm f,k-1} = \{\bm v \in [P_{k-1}(\bm f)]^2 : \bm v \cdot \bm n = 0 \text{ on } \partial \bm f\}.$$

Then there exists the following exact face bubble complex for both  the LBDM complex in \Cref{sec:lbdm} and the HS complex in \Cref{sec:hs}:
\begin{equation}
	B_{\bm f,k-3} \xrightarrow{\curl} B^{BDM}_{\bm f,k-1} \xrightarrow{\div} P_{k-2}(\bm f)/\mathbb R,
\end{equation}
which is a standard result in the finite element exterior calculus \cite{2018Arnold}. There also exists the following edge bubble complex:
$$B_{\bm e,k-2} \xrightarrow{\partial/\partial \bm t} P_{k-1}(\bm e)/\mathbb R$$
for the LBDM complex and
$$B^{1}_{\bm e,k-4} \xrightarrow{\partial/\partial \bm t} B_{\bm e, k-3}/\mathbb R$$ for the HS complex.

The bubble complexes on edges and faces can be regarded as a geometric decomposition of the global finite element spaces \cite{2018ChristiansenHuHu}. In this paper, the bubble complexes give a systematic way to define the degrees of freedom, via the harmonic inner product introduced in \Cref{sec:auxproj}. For example, for the $H^1$ finite element, define 
$\LL u, v\RR_{0, \bm f} = (\curl u, \curl v)_{\bm f}$; for the $H(\div)$ element, define $$\LL \bm \xi, \bm \eta \RR_{1,\bm f} = (\mathcal P_{ \curl B_{\bm f, k -3}} \bm \xi, \mathcal P_{ \curl B_{\bm f, k -3}} \bm \eta)_{\bm f} + (\div \bm \xi, \div \bm \eta)_{\bm f}, \forall \bm \xi, \bm \eta \in [P_{k-1}(\bm f)]^2.$$ 
Then the degrees of freedom $f_{\bm f, i}^1(\bm v)$ can be defined as $\LL \bm v, b_i\RR_{1,\bm f}$, where $b_i$, $i = 1,2,\cdots, \frac{1}{2}(k-2)(k+1)$, form a basis of $B_{\bm f, k - 1}^{BDM}$. 
Now see how the harmonic inner products and the exactness of the bubble complexes on edges and faces can be used to examine Assumptions (A1)-(A4). We verify Assumption (A2) for the HS complex, but the degrees of freedom $f_{\bm f, i}^1(\bm v)$ are given by the harmonic inner products  $\LL \bm v, b_i\RR_{1,\bm f}$ defined above, where $b_i$, $i = 1,2,\cdots, \frac{1}{2}(k-2)(k+1)$, form a basis of $B_{\bm f, k - 1}^{BDM}$.
\begin{itemize}
\item[-] For $f_{\bm e, i}^1(\curl \varphi_{\bm x}) = (\frac{\partial}{\partial \bm t} \varphi_{\bm x}, \frac{\partial b_{\bm e}^i}{\partial \bm t})_{\bm e}$ for some $\frac{\partial b_{\bm e}^i}{\partial \bm t} \in B_{\bm e, k-3} / \mathbb R$, it follows from the exactness of the edge bubble complex for the HS complex that $b_{\bm e}^i$ can be assumed in $B_{\bm e,k-4}^1$. By the definition of the degrees of freedom of the $H^1$ element in \Cref{sec:hs}, $f_{\bm e, i}^1(\curl \varphi_{\bm x}) = (\frac{\partial}{\partial \bm t} \varphi_{\bm x}, \frac{\partial}{\partial \bm t} b)_{\bm e} = 0.$ 

\item[-] For $f_{\bm f, i}^1(\curl \varphi_{\bm x}) = \LL \curl \varphi_{\bm x}, b_2 \RR_{1,\bm f}$ for some $b_2 \in B^{BDM}_{\bm f, k-1}$, then it follows from the exactness of the face bubble complex of the HS complex that there exist $b'' \in B_{\bm f, k-3}$ such that $\curl b'' = b_2$. By the definition of the degrees of freedom of the $H^1$ element in \Cref{sec:hs}, $f_{\bm f, i}^1(\curl \varphi_{\bm x}) = \LL \varphi_{\bm x}, b''\RR_{0, \bm f} =0 .$ 

\item[-] For $f_{\bm f, i}^2(\div \varphi_{\bm e}) = (\div \varphi_{\bm e}, q)_{\bm f}$ for some $q \in P_{k-2}(\bm f) / \mathbb R$, it follows from the exactness of the face bubble complex of the HS complex that there exists $\bm v \in B_{\bm f, k-1}^{BDM}$ such that 
$$\bm v \perp \curl B_{\bm f, k-3}$$
and that $$\div \bm v = q.$$
By the definition of the degrees of freedom of the $H(\div)$ element, it holds that 
$$f_{\bm f, i}^2(\div \varphi_{\bm e}) = (\div \varphi_{\bm e}, \div \bm v)_{\bm f} = \LL \varphi_{\bm e}, \bm v \RR_{1, \bm f} = 0.$$
\end{itemize}
The harmonic inner products and the bubble complex viewpoint play an important role in the remaining sections, especially on three dimensions, see Section~\ref{sec:3D}.

\subsection{The Argyris--Falk--Neilan Complex (Stokes Complex)}

To close the discussion in two dimensions, let us finally consider the following finite element Stokes complexes with enhanced global regularity. The bubble complex view point will play a crucial rule in simplifying the verification procedure. The degrees of freedom on edges of the $H^2$ finite element space and the $H^1$ finite element space will be split into two parts, which require different treatment. We will use the exactness of the bubble complex on edges, combining with some direct algebraic calculation with respect to the $\curl$ operator, to check Assumption (A2). The finite element Stokes complex was introduced by Falk and Neilan in \cite{2013FalkNeilan}, for $k \ge 5$,

\begin{equation}
	\tag{{\sf AFN}}
	\mathbf {Ar}_{k} \xrightarrow{\curl} \mathbf{FN}^{v}_{k-1} \xrightarrow{\div} \mathbf{FN}^p_{k-2},
\end{equation}
where	$$\mathbf{Ar}_{k} := \{ u \in C^1(\Omega) : u|_{\bm f} \in P_k(\bm f), \forall \bm f \in \mathcal T, u \text{ is } C^2 \text{ at each vertex of } \mathcal T\},$$
	$$\mathbf{FN}_{k-1}^v := \{ \bm v \in [C^0(\Omega)]^2 : \bm v|_{\bm f} \in [P_{k-1}(\bm f)]^2, \forall \bm f \in \mathcal T, \bm v \text{ is } C^1 \text{ at each vertex of } \mathcal T\},$$
	and 
	$$\mathbf{FN}_{k-2}^p := \{ p \in L^2(\Omega) : p|_{\bm f} \in P_{k-2}(\bm f), \forall \bm f \in \mathcal T, p \text{ is } C^0 \text{ at each vertex of } \mathcal T\}.$$

For the $H^1$ finite element space $\mathbf{Ar}_{k}$, the degrees of freedom are defined by: for $u \in P_k(\bm f)$, define the degrees of freedom as, for any $\bm f$,
\begin{itemize}
	\item[-] the function value at each vertex $\bm{x}$ of $\bm f$;
	\item[-] at each vertex $\bm x$ of $\bm f$, define $f_{\bm x,1}^0(u) = \frac{\partial u}{\partial x}(\bm x)$, $f_{\bm x,2}^0(u) = \frac{\partial u}{\partial y}(\bm x)$, $f_{\bm x,3}^0(u) = \frac{\partial^2 u}{\partial x^2}(\bm x)$, $f_{\bm x,4}^0(u) = \frac{\partial^2 u}{\partial y^2}(\bm x)$ and $f_{\bm x,5}^0(u) = \frac{\partial^2 u}{\partial x\partial y}(\bm x)$;
	\item[-] on each edge $\bm e$ of $\bm f$, define
		$f_{\bm e,i}^0(u) = (\frac{\partial u}{\partial \bm t},\frac{\partial b_{\bm e}^i}{\partial \bm t} )_{\bm e}$,
		where $b_{\bm e}^i, i = 1,\cdots, k-5$, form a basis of the edge bubble space
		$B^{2}_{\bm e,k-6}:= (\lambda_0\lambda_1)^3P_{k-6}(\bm e)$. Moreover, define
		$f_{\bm e,i}^0(u) = (\frac{\partial u}{\partial \bm n},b_{\bm e}^{\ast,i} )_{\bm e},$
		where $b_{\bm e}^{\ast,i}, i = (k-5)+ 1, \cdots, (k-5)+ (k-4)$, form a basis of the edge bubble space $B_{\bm e,k-5}^1 := (\lambda_0 \lambda_1)^2 P_{k-5}(\bm e)$;
	\item[-] in face $\bm f$, define $f_{\bm f, i}^0(u) = (\operatorname{curl}u,\operatorname{curl}b_{\bm f}^i)_{\bm{f}}$,
		where $b_{\bm f}^i, i =1,2,\cdots, \frac{1}{2}(k-5)(k-4)$, form a basis of $B^{1}_{\bm f, k-6} := (\lambda_0\lambda_1\lambda_2)^2 P_{k-6}(\bm f).$
\end{itemize}

For the $H(\div)$ finite element space $\mathbf{FN}^v_{k-1}$, the degrees of freedom are defined by: for $\bm v \in [P_k(\bm f)]^2$, define the degrees of freedom as, for any $\bm f$,
\begin{itemize}
	\item[-] the moment of the normal component $(\bm v\cdot \bm n, 1)_{\bm e}$ for each edge $\bm e$ of $\bm f$;
	\item[-] at each vertex $\bm x$ of $\bm f$, define $f_{\bm x,1}^1(u) = \frac{\partial v_{x}}{\partial x}(\bm x)$, $f_{\bm x,2}^1(u) = \frac{\partial v_{x}}{\partial y}(\bm x)$, $f_{\bm x,3}^1(u) = \frac{\partial v_{y}}{\partial x}(\bm x)$, $f_{\bm x,4}^1(u) = \frac{\partial v_{y}}{\partial y}(\bm x)$, $f_{\bm x,5}^1(u) = v_{x}(\bm x)$, $f_{\bm x,6}^1(u) = v_{y}(\bm x)$;
	\item[-] on each edge $\bm e$ of $\bm f$, define
		$f_{\bm e, i}^1(\bm v) = (\bm v \cdot \bm n, \frac{\partial b_{\bm e}^i}{\partial \bm t} )_{\bm e},$
		where $\frac{\partial b_{\bm e}^i}{\partial \bm t} , i = 1,2,\cdots, k-5$, form a basis of $B_{\bm e,k-5}^1 /\R$.
		Further, define
		$f_{\bm e, i}^1(\bm v) = (\bm v\cdot \bm t, b_{\bm e}^{*,i}),$ where $b_{\bm e}^{*,i}, i = (k-5) + 1, \cdots, (k-5) + (k-4)$, form a basis of the edge bubble space $B_{\bm e,k-5}^1$;

	\item[-] in face $\bm f$, set
		$f_{\bm f, i}^1(\bm v) = \LL \bm v, b_i \RR_{1,\bm f}$, where $b_i,i =1,2,\cdots, (k-3)(k-2)$, form a basis of $[(\lambda_0\lambda_1\lambda_2)P_{k-4}(\bm f)]^2$, and the inner product is defined as 
		$$\LL \bm v, \bm z \RR_{1, \bm f} := (\mathcal P_{\curl B_{\bm f, k-6}^1} \bm v, \mathcal P_{\curl B_{\bm f, k-6}^1} \bm z)_{\bm f} + (\div \bm v, \div \bm z)_{\bm f}.$$

\end{itemize}

For the $L^2$ finite element space $\mathbf{FN}_{k-2}^p$, the local degrees of freedom are defined by: for $p \in P_k(\bm f)$, define the degrees of freedom as, for any $\bm f$,

\begin{itemize}
	\item[-] the moment $(p, 1)_{\bm f}$ for each face $\bm f$;
	\item[-] at each vertex $\bm x$ of $\bm f$, define $f_{\bm x,1}^2(p) = p(\bm x)$;
	\item[-] inside face $\bm f$, define $f_{\bm f,i}^1(p) = (p, q_i)_{\bm f}$,
		where $q_i,i = 1,2,\cdots,\frac{1}{2}(k-1)k - 4$, form a basis of $P_{k-2}^{(0)}(\bm f)/\mathbb R,$ where $P_{k-2}^{(0)}(\bm f):=\{ q \in P_{k-2}(\bm f): q(\bm x) = q(\bm y) = q(\bm z) = 0\}$.
\end{itemize}

The corresponding edge bubble complex is
$$(\lambda_0\lambda_1)^3P_{k-6}(\bm e) \xrightarrow{\partial/\partial \bm t} (\lambda_0\lambda_1)^2P_{k-5}(\bm e)/\mathbb R,$$
and the face bubble complex is
\begin{equation}
	\label{eq:ar-facebubble} (\lambda_0\lambda_1\lambda_2)^2 P_{k-6}(\bm f) \xrightarrow{\curl}[(\lambda_0\lambda_1\lambda_2)P_{k-4}(\bm f)]^2 \xrightarrow[]{\div} P_{k-2}^{(0)}(\bm f) /\mathbb R.
\end{equation}

\begin{lemma}
	The above polynomial sequence \eqref{eq:ar-facebubble} is an exact complex.
\end{lemma}
\begin{proof}
	It is straightforward to see that $\curl ((\lambda_0\lambda_1\lambda_2)^2P_{k-6}(\bm f)) \subseteq [(\lambda_0\lambda_1\lambda_2) P_{k-4}(\bm f)]^2.$ Since for any $\bm v_1 \in [P_{k-4}(\bm f)]^2$, $$\div((\lambda_0\lambda_1\lambda_2)\bm v_1) = \nabla(\lambda_0\lambda_1\lambda_2) \cdot \bm v_1 + (\lambda_0\lambda_1\lambda_2)\div \bm v_1,$$ and both $\nabla(\lambda_0\lambda_1\lambda_2)$ and $(\lambda_0\lambda_1\lambda_2)$ vanish at all the vertices of $\bm f$, it then implies that $\div (\lambda_0\lambda_1\lambda_2)[P_{k-4}(\bm f)]^2 \subset P_{k-2}^{(0)}(\bm f)$. Hence, it follows from $$\int_{\bm f} \div \bm v = \int_{\partial \bm f} \bm v \cdot \bm n = 0, \, \forall \bm v \in [(\lambda_0\lambda_1\lambda_2)P_{k-4}(\bm f)]^2$$ that $\div (\lambda_0\lambda_1\lambda_2)[P_{k-4}(\bm f)]^2 \subset P_{k-2}^{(0)}(\bm f) / \mathbb R.$ As a result, the sequence \eqref{eq:ar-facebubble} is a complex.

	Next, suppose that $\div \bm v = 0$ for some $\bm v \in [(\lambda_0\lambda_1\lambda_2)P_{k-4}(\bm f)]^2$, then there exists $\varphi \in (\lambda_0\lambda_1\lambda_2)^2 P_{k-6}(\bm f)$ such that $\curl \varphi = \bm v.$ It follows from the exactness of the polynomial complex that there exists $\varphi \in P_{k}(\bm f)$ such that $\curl \varphi = \bm v$. Moreover, it can be assumed that $\varphi(\bm x) = 0$ at some vertex $\bm x$ of $\bm f$. Since $\grad \varphi = 0$ on $\partial \bm f$, implying that $\varphi = 0$ on $\partial \bm f$. As a result, it holds that $\varphi \in (\lambda_0\lambda_1\lambda_2)^2 P_{k-6}(\bm f)$.
	
	It suffices to count the dimensions of the bubble spaces therein. From
	$$\dim P_{k-6} + \dim P_{k-2}^{(0)} - 1 = \frac{1}{2}(k-5)(k-4) + \frac{1}{2}(k-1)k - 3 - 1 = (k-2)(k-3) = 2 \dim P_{k-4},$$
	it follows that the complex \eqref{eq:ar-facebubble} is exact.
\end{proof}

\section{Applications in Three Dimensions}
\label{sec:3D}

Before considering applications, introduce further notation. For a face $\bm f$ of the element $\bm K$, let $\bm n$ be its unit normal vector, $\bm t_i, i = 1,2,$ be its two linearly independent unit tangential vectors, such that $\bm t_1 \perp \bm t_2$. For each edge $\bm e$ of the element, let $\bm t$ be its unit tangential vector, $\bm n_i, i = 1,2$, be its two linearly independent unit normal vectors, such that $\bm n_1 \perp \bm n_2$. Define the operator $E_{\bm f} : \mathbb R^3 \to \mathbb R^2$ such that $E_{\bm f}\bm v = [\bm v \cdot \bm t_1, \bm v\cdot \bm t_2]^T$ for $\bm v \in \mathbb R^3$. Let $\nabla_{\bm f}$ be the gradient operator with respect to $\bm t_1$ and $\bm t_2$, such that $\nabla_{\bm f} u = [\bm t_1 \cdot \nabla u, \bm t_2 \cdot \nabla u]^T,$ and $\div_{\bm f}$ be the divergence operator with respect to $\bm t_1$ and $\bm t_2$, such that $\div_{\bm f}\bm v = \frac{\partial v_1}{\partial \bm t_1} + \frac{\partial v_2}{\partial \bm t_2}$ with $\bm v = [v_1, v_2]^T$ defined on $\bm f$. Accordingly, define $\curl_{\bm f} u = [ - \bm t_2 \cdot \nabla u, \bm t_1 \cdot \nabla u]^T$ and $\rot_{\bm f} \bm v = - \frac{\partial v_2}{\partial \bm t_1} + \frac{\partial v_1}{\partial \bm t_2}.$

Next, a family of harmonic bilinear forms is introduced. For the $H^1$ finite element, define the following three bilinear forms (they are inner products if the space under consideration does not contain constants.)

\begin{itemize}
\item[ ] $\LL u,v \RR_{0, \bm K}  = (\grad u, \grad v)_{\bm K},        \forall u,v \in H^1(\bm K) ;$
\item[] $\LL u,v \RR_{0, \bm f}  = (\grad_{\bm f} u, \grad_{\bm f} v)_{\bm f},                 \forall u,v \in H^1(\bm f) ;$   
\item[] $\LL u,v \RR_{0, \bm e}  = (\frac{\partial u}{\partial \bm t}, \frac{\partial v}{\partial \bm t})_{\bm e}, \forall u,v \in H^1(\bm e),  $          
\end{itemize}
where $\bm t$ is the unit tangential vector of $\bm e$. 
The inner products $\LL \bm \xi, \bm \eta \RR_{1, \bm f}$, $\LL \bm \xi, \bm \eta\RR_{1, \bm K}$ and $\LL \bm \xi, \bm \eta\RR_{2, \bm K}$ will be defined specifically in each finite element complex discussed in this section, with respect to their corresponding bubble function spaces.

In three dimensions, there are four types of the standard finite element de Rham complexes by the N\'edel\'ec element, RT element, and BDM element, see \cite{2018Arnold}. However, they are not discussed in this paper, and the bounded projection operators can be constructed similarly to that in the previous section. Instead, only those finite element complexes with additional regularity are discussed. The exactness of the finite element complexes discussed below is shown in \cite{2018ChristiansenHuHu,2015Neilan}, and therefore Assumption (B1) holds.

For convenience, both ends of the complex are omitted. For the standard finite element complexes, $\mathbb R \to $ and $\to 0$ are omitted, and for the bubble complexes, $0 \to $ and $\to 0$ are omitted.

\subsection{The 3D Hermite--Stenberg Complex}

Consider the Hermite--Stenberg complex in three dimensions from \cite{2018ChristiansenHuHu}. Note that both the structure itself and the argument are quite similar to those in two dimensions, as there are no complicated local regularity requirements. However, a significant difference in three dimensions is the need to introduce more harmonic inner products, including two for the $H(\curl)$ finite element space on the faces and in the elements, respectively, and one for the $H(\div)$ finite element space in the elements. To help the readers get used into the argument in this section, we choose this example and present the verification in details. Although the Hermite--Stenberg complex may seem straightforward (and perhaps somewhat routine), it showcases the general principles in the construction of the degrees of freedom in three dimensions. For $k \ge 4$, consider the following finite element complexes,
\begin{equation}
	\tag{{\sf HS3}}
	\mathbf {Hm}_{k} \xrightarrow{\grad} \mathbf {St}^{\curl}_{k-1} \xrightarrow{\curl} \mathbf{BDM}_{k-2} \xrightarrow{\div} \mathbf P_{k-3}^{-},
\end{equation}
	where 	$$\mathbf{Hm}_{k} := \{ u \in C^0(\Omega) ~:~ u|_{\bm K} \in P_k(\bm K), \forall \bm K \in \mathcal T, u \text{ is } C^1 \text{ at each vertex of } \mathcal T\},$$
	$$\mathbf{St}_{k-1}^{\curl} := \{ \bm \xi \in H(\curl, \Omega; \mathbb R^3) ~:~ \bm \xi |_{\bm K} \in [P_{k-1}(\bm K)]^3, \forall \bm K \in \mathcal T, \bm \xi \text{ is } C^0 \text{ at each vertex of } \mathcal T\},$$
	$$\mathbf{BDM}_{k-2} :=  \{ \bm v \in H(\div, \Omega; \mathbb R^3) ~:~ \bm v |_{\bm K} \in [P_{k-2}(\bm K)]^3, \forall \bm K \in \mathcal T\},$$
	and 
	$$\mathbf P_{k-3}^- := \{ p \in L^2(\Omega) ~:~  p|_{\bm K} \in P_{k-3}(\bm K),\, \forall \bm K \in \mathcal T\}.$$
	For simplicity, we abuse the notation of the spaces which has been defined in two dimensions.

The (three-dimensional) Hermite element space $\mathbf{Hm}_k$ is taken as the $H^1$ finite element space. The corresponding degrees of freedom are defined by: for $u \in P_k(\bm K)$, define the degrees of freedom as follows, for any $\bm K$,
\begin{itemize}
	\item[-] the function value $u(\bm x)$ at each vertex $\bm x$ of $\bm K$;
	\item[-] at each vertex $\bm x$ of $\bm K$, set $f_{\bm x,i}^0(u)$ be the first order derivatives of $u$ at $\bm x$;
	\item[-] on each edge $\bm e$ of $\bm K$, set $f_{\bm e, i}^0(u) = \LL u , b_{hm, \bm e}^i \RR_{0,\bm e}$, where $b_{hm, \bm e}^i$, form a basis of the edge bubble space $(\lambda_0\lambda_1)^2P_{k-4}(\bm e)$;
	\item[-] on each face $\bm f$ of $\bm K$, set $f_{\bm f,i}^0(u) = \LL u, b_{hm, \bm f}^i \RR_{0,\bm f}$, where $b_{hm, \bm f}^i$, form a basis of the face bubble space $B_{\bm f, k-3} =  (\lambda_0\lambda_1\lambda_2)P_{k-3}(\bm f)$;
	\item[-] inside $\bm K$, set $f_{\bm K,i}^0(u) = \LL u,  b_{hm, \bm K}^i \RR_{0, \bm K}$, where $b_{hm, \bm K}^i$, form a basis of the element bubble space $B_{\bm K, k-4} = (\lambda_0\lambda_1\lambda_2\lambda_3)P_{k-4}(\bm K).$
\end{itemize}

To define the $H(\curl)$ finite element space, the following inner products are introduced. 
Define 
$$\LL \bm \xi,\bm \eta \RR_{1, \bm f} = (\proj_{\grad_{\bm f} B_{\bm f, k - 3} }\bm \xi, \proj_{\grad_{\bm f} B_{\bm f, k - 3}}\bm \eta)_{\bm f} + (\rot_{\bm f}  \bm \xi, \rot_{\bm f} \bm w)_{\bm f},\quad\forall \bm \xi, \bm \eta \in [P_{k-1}(\bm f)]^2, $$
and 
$$\LL \bm \xi, \bm \eta \RR_{1, \bm K} = (\proj_{ \grad B_{\bm K, k-4}} \bm \xi, \proj_{\grad B_{\bm K, k-4}} \bm \eta)_{\bm K} + (\curl \bm \xi, \curl \bm \eta)_{\bm K}, \quad\forall \bm \xi, \bm \eta \in [P_{k-1}(\bm K)]^3.$$

The $H(\curl)$ finite element space $\mathbf{St}^{\curl}_{k-1}$ is defined as follows: for $\bm \xi = [\xi_x, \xi_y, \xi_z]^T \in [P_{k-1}(\bm K)]^3$, define the following degrees of freedom, for any $\bm K$,
\begin{itemize}
	\item[-] the moment of the tangential component $(\bm \xi \cdot \bm t, 1)_{\bm e}$ for each edge $\bm e$ of $\bm K$;
	\item[-] at each vertex $\bm x$ of $\bm K$, define $f_{\bm x,1}^1(\bm \xi) = \xi_{x}, f_{\bm x,2}^1(\bm \xi) = \xi_y$ and $f_{\bm x,3}^1(\bm \xi) = \xi_z$;
	\item[-] for each edge $\bm e$ of $\bm K$, set
		$f_{\bm e,i}^1(\bm \xi) = (\bm \xi \cdot \bm t, b_{st,\bm e}^{\curl,i})_{\bm e},$ where $b_{st,\bm e}^{\curl,i}$, form a basis of $ (\lambda_0\lambda_1)P_{k-3}(\bm e)/\mathbb R$;
	\item[-] for each face $\bm f$ of $\bm K$, set
		$f_{\bm f, i}^1(\bm \xi) = \LL E_{\bm f}\bm \xi, b_{st, \bm f}^{\curl, i} \RR_{1,\bm f}$, where $b_{st, \bm f}^{\curl, i}$, form a basis of $[B_{\bm f,k-1}^{BDM}]^{\perp}$, the rotated BDM bubble function space on face $\bm f$; \footnote{While in the original statement, it uses a moment with the test function in some Raviart--Thomas space. However, this is just equivalent degrees of freedom for the BDM bubble, see \cite{2013BoffiBrezziFortin}.}

	\item[-] inside $\bm K$, set
		$f_{\bm K, i}^1(\bm \xi) =  \LL \bm \xi, b_{st,\bm K}^{\curl,i}\RR_{1,\bm K},$
		where $  b_{st,\bm K}^{\curl,i}$, form a basis of $B^N_{\bm K, k-1}$, where $B^N_{\bm K, k-1}$ is the bubble space of the N\'edel\'ec bubble space $$\{ \bm w \in [P_{k-1}(\bm K)]^3 ~:~ (E_{\bm f} u) = 0 \text{ on } \partial \bm K\}.$$

\end{itemize}

To define the $H(\div)$ finite element space, the following inner product is introduced. 
For $\bm v, \bm z \in [P_{k-2}(\bm K)]^3$, define 
$$ \LL \bm v, \bm z\RR_{2, \bm K} = (\mathcal P_{\curl B_{\bm K, k-1}^N} \bm v, \mathcal P_{\curl B_{\bm K, k-1}^N} \bm z)_{\bm K} + (\div \bm v, \div \bm z)_{\bm K}.$$ 

The $H(\div)$ element is just the BDM element, which is defined by: for $\bm v \in [P_{k-2}(\bm K)]^3$, define the degrees of freedom as follows, for any $\bm K$,
\begin{itemize}
	\item[-] the moment of the normal component, $(\bm v \cdot \bm n, 1)_{\bm f}$ for each face $\bm f$ of $\bm K$;
	\item[-] for each face $\bm f$ of $\bm K$, set $f_{\bm f, i}^2(\bm \xi) = (\bm v \cdot \bm n, b_{bdm,\bm f}^i)_{\bm f},$
		where $ b_{bdm,\bm f}^i$, form a basis of $P_{k-2}(\bm f)/\mathbb R$;
	\item[-] inside $\bm K$, set
		$f_{\bm K, i}^2(\bm v) = \LL \bm v, b_{bdm,\bm K}^i \RR_{2,\bm K}$, where $ b_{bdm,\bm K}^i$, form a basis of \begin{equation}\label{eq:bdm-bubble} B_{\bm K,k-2}^{BDM} := \{ \bm \eta \in [P_{k-2}(\bm K)]^3: (\bm \eta \cdot \bm n) = 0 \text{ on } \partial \bm K\}.\end{equation}
\end{itemize}

The degrees of freedom of the $L^2$ finite element space $\mathbf{P}_{k-3}^-$ are defined by : for $p \in P_{k-3}(\bm K)$, define the degrees of freedom as follows, for any $\bm K$,
\begin{itemize}
	\item[-] the moment of $p$ in each element $\bm K$, namely, $(p,1)_{\bm K}$;
	\item[-] let $f_{\bm K,i}^3(p) = (p,q_i)_{\bm f}$ where $q_i$, form a basis of $P_{k-3}(\bm K) / \R$.
\end{itemize}

The key to verifying Assumption (B2) is the exactness of the associated bubble complexes. Consider the following edge bubble complex for any edge $\bm e$,
\begin{equation}
	\label{eq:hs3-edgebubble}
	(\lambda_0\lambda_1)^2P_{k-4}(\bm e) \xrightarrow{\partial/\partial \bm t} (\lambda_0\lambda_1)P_{k-3}(\bm e)/ \mathbb R,
\end{equation}
the following face bubble complex for any face $\bm f$,
\begin{equation}
	\label{eq:hs3-facebubble}
	(\lambda_0\lambda_1\lambda_2)P_{k-3}(\bm f)\xrightarrow{\grad_{\bm f}} [B^{BDM}_{\bm f, k -1}]^{\perp} \xrightarrow{\rot_{\bm f}} P_{k-2}(\bm f)/\mathbb R,
\end{equation}
or the rotated version\footnote{In what follows, only this version will be provided.}
\begin{equation}
	\label{eq:hs3-facebubble-perp}
	(\lambda_0\lambda_1\lambda_2)P_{k-3}(\bm f)\xrightarrow{\curl_{\bm f}} B^{BDM}_{\bm f, k -1} \xrightarrow{\div_{\bm f}} P_{k-2}(\bm f)/\mathbb R,
\end{equation}
and the following element bubble complex for any $\bm K$,
\begin{equation}
	\label{eq:hs3-elembubble}
	(\lambda_0\lambda_1\lambda_2\lambda_3)P_{k-4}(\bm K)\xrightarrow{\grad} B_{\bm K,k-1}^{N} \xrightarrow{\curl} B_{\bm K,k-2}^{BDM} \xrightarrow{\div} P_{k-3}(\bm K)/\mathbb R.
\end{equation}
Note that the face bubble complex and the edge bubble complex have appeared in the two-dimensional HS complex.

Here we present how to check Assumptions (B1)-(B4).
It is easy to see that $f_{\sigma,i}^{0}(c) = 0 $ for any real number $c$. To check Assumption (B2), it suffices to use the harmonic inner products, and the exactness of the bubble complexes, as mentioned before. 

First, check that $f_{\sigma,i}^1(\nabla \varphi_{\bm x}) = 0$ for all $f_{\sigma,i}^1(\cdot)$. Denote by $\bm \xi = \nabla \varphi_{\bm x}$. 
\begin{itemize}
	
	\item[-] It follows from the definition of the degrees of freedom that  $f_{\bm x, i}^1(\nabla \varphi_{\bm x}) = 0$.
\item[-] For $f_{\bm e,i}^1(\bm \xi) = (\bm \xi\cdot \bm t, b_1)_{\bm e} $ with $b_1\in (\lambda_0\lambda_1)P_{k-3}(\bm e)/\mathbb R$, by the exactness of \eqref{eq:hs3-edgebubble}, there exists $b \in (\lambda_0\lambda_1)^2 P_{k-4}(\bm e)$ such that $\frac{\partial b}{\partial \bm t} = b_1$. As a result, it follows from the definition of $\varphi_{\bm x}$ that 
 $$f_{\bm e,i}^1(\bm \xi) = (\bm \xi \cdot \bm t, b_1)_{\bm e} = (\frac{\partial \varphi_{\bm x}}{\partial \bm t}, \frac{\partial }{\partial \bm t} b)_{\bm e} = \LL \varphi_{\bm x}, b\RR_{0,\bm e} = 0.$$
\item[-] For $f_{\bm f,i}^1(\bm \xi) = \LL E_{\bm f}\bm \xi, b_2 \RR_{1,\bm f}$ with $b_2 \in [B_{\bm f, k-1}^{BDM}]^{\perp}$, it follows from the exactness of \eqref{eq:hs3-facebubble} that there exists $b' \in (\lambda_0\lambda_1\lambda_2)P_{k-3}(\bm f)$ such that $\nabla_{\bm f} b' = \mathcal P_{\grad_{\bm f} B_{\bm f,k-3}} b_2$. It follows from the definition of the basis function $\varphi_{\bm x}$ that $E_{\bm f} \nabla \varphi_{\bm x} = \nabla_{\bm f} \varphi_{\bm x}$  that 
\[\begin{split}
	f_{\bm f, i}^1(\bm \xi) & = \LL  E_{\bm f} \bm\xi, b_2\RR_{1,\bm f} \\ & = (\mathcal P_{\grad_{\bm f} B_{\bm f,k-3}} E_{\bm f}\bm \xi,  \mathcal P_{\grad_{\bm f} B_{\bm f,k-3}} b_2)_{\bm f} + (\rot_{\bm f} \nabla_{f}\varphi_{\bm x}, \rot_{\bm f} b_2 )_{\bm f} \\
	& = (\nabla_{\bm f} \varphi_{\bm x},\nabla_{\bm f} b')_{\bm f}  = \LL \varphi_{\bm x}, b' \RR_{0,\bm f} = 0.
\end{split}\]
\item[-] For $f_{\bm K,i}^1(\bm \xi) = \LL \xi, b_3 \RR_{1, \bm K}$ with $b_3 \in B_{\bm K,k-1}^N$, it follows from the exactness of \eqref{eq:hs3-elembubble} that there exists $b^{''} \in (\lambda_0\lambda_1\lambda_2\lambda_3)P_{k-4}(\bm K)$ such that $\nabla b^{''} = \mathcal P_{\grad B_{\bm K,k-4}} b_3$. Again, by the definition of $\varphi_{\bm x}$,
\[\begin{split}f_{\bm K, i}^1(\bm \xi) & = \LL \bm\xi, b_2\RR_{1,\bm K} \\ & = (\mathcal P_{{\grad B_{\bm K,k-4}}} \bm \xi,  \mathcal P_{\grad B_{\bm K,k-4}} \nabla b'')_{\bm K} + (\curl \bm \xi, \curl b_3 )_{\bm K} \\
	& = (\nabla  \varphi_{\bm x},\nabla  b'')_{\bm K} 
 = \LL \varphi_{\bm x}, b'' \RR_{0,\bm K} = 0.
\end{split}\] 
\end{itemize}

Next, check that $f_{\sigma,i}^2(\curl \varphi_{\bm e}) = 0$ for all $f_{\sigma,i}^2(\cdot)$. Denote by $\bm v = \curl \varphi_{\bm e}$.
\begin{itemize}
\item[-] For $f_{\bm f,i }^2(\bm v) = (\bm v \cdot \bm n, b_1)_{\bm f}$ with $b_1 \in P_{k-2}(\bm f)/\mathbb R$, it follows from the exactness of \eqref{eq:hs3-facebubble} that there exists $b \in [B^{BDM}_{\bm f, k -1}]^{\perp}$ such that $b \perp  \grad_{\bm f} B_{\bm f, k-3}$ and $\rot_{\bm f} b = b_1.$ It then follows from the definition of $\LL\cdot, \cdot \RR_{1, \bm f}$ and the definition of the basis function $\varphi_{\bm e}$ that 
\[ f_{\bm f,i }^2(\bm v) = (\bm v \cdot \bm n, b_1)_{\bm f} = (\rot_{\bm f} E_{\bm f} \varphi_{\bm e}, \rot_{\bm f} b)_{\bm f} = \LL E_{\bm f} \varphi_{\bm e}, b\RR_{1,\bm f} = 0\]
\item[-] For $f_{\bm K,i}^2(\bm v) = \LL \bm v, b_2 \RR_{2,\bm K}$ with $b_2 \in B_{\bm K,k-2}^{BDM}$, it follows from the exactness of \eqref{eq:hs3-elembubble} that there exists $b' \in B_{\bm K, k-1}^N$ such that 
$b' \perp \grad  B_{\bm K,k-4}$ and 
$ \curl b' = P_{\curl B_{\bm K,k-1}^N} b_2.$ Again, by the definition of $\varphi_{\bm e}$, it holds that, 
\[\begin{split}f_{\bm K,i}^2(\bm v) & = \LL \bm v, b_2 \RR_{2,\bm K} \\ & = (\mathcal P_{ \curl B_{\bm K,k-1}^N} \curl \varphi_{\bm e}, \mathcal P_{\curl B_{\bm K,k-1}^N} b_2)_{\bm K} + (\div \curl \varphi_{\bm e}, \div b_2)_{\bm K}\\
&= (\curl \varphi_{\bm e}, \curl b')_{\bm K}= \LL \varphi_{\bm e}, b' \RR _{\bm 1, \bm K} = 0.
\end{split}\]
\end{itemize}

A similar argument shows that $f_{\bm K, i}^3(\div \varphi_{\bm f}) = 0$ for all $f_{\bm K,i}^3(\cdot)$. 

\subsection{The 3D Argyris Complex}

The argument of the 3D Argyris complex is similar to the two-dimensional case, with the degrees of freedom of each finite element space separated into two parts. However, unlike the 2D case, the 3D Argyris complex does not admit a higher global regularity. This property reduces the procedure when checking Assumption (B2).

Recall the Argyris complex from \cite{2018ChristiansenHuHu}, for $k\ge 5$,
\begin{equation}
	\tag{{\sf A3}}
	\mathbf {Ar}_{k} \xrightarrow{\grad} \mathbf {A}^{\curl}_{k-1} \xrightarrow{\curl}  \mathbf{St}_{k-2}^{\div} \xrightarrow{\div}  \mathbf P_{k-3}^{-},
\end{equation}
where
	\[\begin{split} \mathbf{Ar}_{k} := \{ u \in C^0(\Omega) ~:~ & u|_{\bm K} \in P_k(\bm K), \forall \bm K \in \mathcal T, \\ & u \text{ is } C^2 \text{ at each vertex of } \mathcal T, \text{ is } C^1 \text{ on each edge of } \mathcal T\}, \end{split}\]
	\[\begin{split} \mathbf {A}^{\curl}_{k-1} := \{ \bm \xi \in H(\curl, \Omega; \mathbb R^3) ~:~ &\bm \xi|_{\bm K} \in [P_{k-1}(\bm K)]^3, \forall \bm K \in \mathcal T, \\ &\bm \xi \text{ is } C^1 \text{ at each vertex of } \mathcal T, \text{ is } C^0 \text{ on each edge of } \mathcal T\},\end{split}\]
	and
	$$ \mathbf{St}_{k-2}^{\div} := \{ \bm v \in H(\div,\Omega;\mathbb R^3) : \bm v |_{\bm K} \in  [P_{k-2}(\bm K)]^3, \forall \bm K\in \mathcal T, \bm v \text{ is } C^0 \text{ at each vertex of } \mathcal T\}. $$

The $H^1$ finite element space is the Argyris element space $\mathbf{Ar}_k$, and the degrees of freedom are defined by: for $u \in P_k(\bm K)$, define the following degrees of freedom as, for any $\bm K$,
\begin{itemize}
	\item[-] the function value $u(\bm x)$ at each vertex $\bm x$ of $\bm K$;
	\item[-] at each vertex $\bm x$ of $\bm K$, define $f_{\bm x,i}^0(u)$ as the first and second order derivatives of $u$ at $\bm x$;
	\item[-] on each edge $\bm e$ of $\bm K$, set $f_{\bm e, i}^0(u) = \LL u , b_{ar,\bm e}^{i} \RR_{0,\bm e}$, where $b_{ar,\bm e}^{i}$, form a basis of the edge bubble space $(\lambda_0\lambda_1)^3P_{k-6}(\bm e)$,
		and set
		$f_{\bm e,i}^0(u) = (\frac{\partial u}{\partial \bm n_{j}}, b_{ar,\bm e}^{\ast, i})_{e}$ with $\bm n_{j}$ two linearly independent normal vectors with respect to the edge $\bm e$, where $b_{ar,\bm e}^{\ast, i}$, form a basis of the edge bubble space $(\lambda_0\lambda_1)^2P_{k-5}(\bm e)$;

	\item[-] on each face $\bm f$ of $\bm K$, set $\LL u, b_{ar, \bm f}^i \RR_{0, \bm f}$, where $b_{ar, \bm f}^i$, form a basis of the face bubble space $B_{\bm f, k-6}^1 =  (\lambda_0\lambda_1\lambda_2)^2P_{k-6}(\bm f)$;
	\item[-] inside $\bm K$, set $\LL u, b_{ar, \bm K}^i \RR_{0, \bm K}$, where $b_{ar, \bm K}^i$, form a basis of the element bubble space $B_{\bm K, k-4} = (\lambda_0\lambda_1\lambda_2\lambda_3)P_{k-4}(\bm K).$
\end{itemize}

To define the $H(\curl)$ finite element space, the following inner products are introduced:  
$$\LL \bm \xi,\bm \eta \RR_{1, \bm f} = (\proj_{\grad_{\bm f} B^1_{\bm f, k - 6} }\bm \xi, \proj_{\grad_{\bm f} B^1_{\bm f, k - 6}}\bm \eta)_{\bm f} + (\rot_{\bm f}  \bm \xi, \rot_{\bm f} \bm \eta)_{\bm f}, \quad\forall \bm \xi, \bm \eta \in [P_{k-1}(\bm f)]^2,$$
and 
$$\LL \bm \xi, \bm \eta \RR_{1, \bm K} = (\proj_{ \grad B_{\bm K, k-4}} \bm \xi, \proj_{\grad B_{\bm K, k-4}} \bm \eta)_{\bm K} + (\curl \bm \xi, \curl \bm \eta)_{\bm K},\quad\forall \bm \xi, \bm \eta \in [P_{k-1}(\bm K)]^3.$$

The $H(\curl)$ finite element space $\mathbf {A}^{\curl}_{k-1}$ is defined as follows: for $\bm \xi \in [P_{k-1}(\bm K)]^3$, the degrees of freedom are defined as, for any $\bm K$,
\begin{itemize}
	\item[-] the moment of the tangential component $(\bm \xi \cdot \bm t, 1)_{\bm e}$ for each edge $\bm e$ of $\bm K$;
	\item[-] at each vertex $\bm x$ of $\bm K$, define $f_{\bm x, i}^1(\bm \xi)$ as the function value and the first order derivatives of $\bm \xi$.
	\item[-] on each edge $\bm e$ of $\bm K$, set $f_{\bm e, i}^1(\bm \xi) = ( \bm \xi \cdot \bm t, b_{a,\bm e}^{\curl,i})_{\bm e}$, where $b_{a, \bm e}^{\curl,i}$, form a basis of the edge bubble space $(\lambda_0\lambda_1)^2P_{k-5}(\bm e) /\mathbb R$,
		and set $f_{\bm e, i}^1(\bm \xi) = (\bm \xi\cdot \bm n_{j}, b_{a,\bm e}^{\ast, \curl, i})_{\bm e}$, where $b_{a,\bm e}^{\ast, \curl, i}$, form a basis of the edge bubble space $(\lambda_0\lambda_1)^2P_{k-5}(\bm e),$ and $j = 1, 2$;
	\item[-] on each face $\bm f$ of $\bm K$, set $f_{\bm f,i}^1(\bm \xi) = \LL \bm E_{\bm f}\bm \xi, b_{a,\bm f}^{\curl, i} \RR_{1,\bm f}$, where $b_{a,\bm f}^{\curl, i}$, form a basis of the face bubble space $(\lambda_0\lambda_1\lambda_2)[P_{k-4}(\bm f)]^2$;
	\item[-] inside $\bm K$, set $f_{\bm K,i}^1(\bm \xi) = \LL \bm \xi, b_{a,\bm K}^{\curl, i} \RR_{1,\bm K}$, where $b_{a, \bm K}^{\curl, i}$, form a basis of $B^N_{\bm K, k-1}$, with $B^N_{\bm K, k-1}$ the N\'edel\'ec bubble space 
	\begin{equation}\label{eq:nedelecbubble}B^N_{\bm K, k-1} := \{ \bm \eta \in [P_{k-1}(\bm K)]^3 ~:~ (E_{\bm f} \bm \eta) = 0 \text{ on } \partial \bm K\}.\end{equation}
\end{itemize}

To define the $H(\div)$ finite element space, the following inner product is introduced:
for $\bm v, \bm z \in [P_{k-2}(\bm K)]^3$, define 
$$ \LL \bm v, \bm z\RR_{2, \bm K} = (\mathcal P_{\curl B_{\bm K, k-1}^N} \bm v, \mathcal P_{\curl B_{\bm K, k-1}^N} \bm z)_{\bm K} + (\div \bm v, \div \bm z)_{\bm K}.$$ 

The $H(\div)$ finite element space $\mathbf{St}_{k-2}^{\div} $ is defined as follows: for $\bm v \in [P_{k-2}(\bm f)]^3$, define the following degrees of freedom, for any $\bm K$,
\begin{itemize}
	\item[-] the moment of the normal component, $(\bm v\cdot \bm n, 1)_{\bm f}$ for each face of $\bm K$;
	\item[-] at each vertex $\bm x$ of $\bm K$, let $f_{\bm x, i}^2(\bm v)$ be the function value of $\bm v$ at each vertex $\bm x$ of $\bm K$;
	\item[-] on each face $\bm f$ of $\bm K$, set $f_{\bm f, i}^2(\bm v) = (\bm v\cdot \bm n, b_{st,\bm f}^{\div, i})_{\bm f}$, where $b_{st,\bm f}^{\div,i}$, form a basis of the space $P_{k-2}^{(0)}(\bm f) /\mathbb R$, recall that $P_{k-2}^{(0)}(\bm f) = \{u\in P_{k-2}(\bm f):u\text{ vanishes at vertices of }\bm f\}$;
	\item[-] inside $\bm K$, set $f_{\bm K, i}^2(\bm v) = \LL \bm v, b_{st, \bm K}^{\div,i} \RR_{2,\bm K}$, where $b_{st, \bm K}^{\div, i}$, form a basis of the element bubble space $ B_{\bm K,k-2}^{BDM}$.
\end{itemize}

The corresponding edge bubble complex is
$$(\lambda_0\lambda_1)^3P_{k-6}(\bm e) \xrightarrow{\partial/\partial \bm t}(\lambda_0\lambda_1)^2P_{k-5}(\bm e) / \mathbb R,$$
the face bubble complex is
$$ (\lambda_0\lambda_1\lambda_2)^2 P_{k-6}(\bm f) \xrightarrow{\curl_{\bm f}} [(\lambda_0\lambda_1\lambda_2)P_{k-4}(\bm f)]^2 \xrightarrow{\div_{\bm f}} P_{k-2}^{(0)}(\bm f) /\mathbb R,$$
and the cell bubble complex is
$$
	(\lambda_0\lambda_1\lambda_2\lambda_3)P_{k-4}(\bm K) \xrightarrow{\grad} B^N_{\bm K, k-1} \xrightarrow{\curl}B_{\bm K,k-2}^{BDM}\xrightarrow{\div} P_{k-3}(\bm f)/\mathbb R.
$$
Note that the face bubble complex and the edge bubble complex have appeared in the two-dimensional Argyris--Falk--Neilan complex.

To check Assumption (B2), for those degrees of freedom defined by the harmonic inner products, the argument is similar to that of the previous sections. It suffices to check those degrees of freedom which are {\bf not} defined by the harmonic inner product, which will be finished by an algebraic calculation argument: To prove that $(\grad \varphi_{\bm x} \cdot \bm n_{j}, b_{a, \bm e}^{\ast, \curl i})_{e} = 0$, it suffices to observe that $(\grad \varphi_{\bm x} \cdot \bm n_{j}, b_{a, \bm e}^{\ast, \curl i})_{e} = (\frac{\partial \varphi_{x}}{\partial \bm n_{j}}, b_{ar, \bm e}^{\ast, i})_{e}$, while the two bubble functions $b_{a,\bm e}^{\ast,\curl,i}$ and $b_{ar,\bm e}^{\ast, i}$ come from the same edge bubble space $(\lambda_0\lambda_1)^2P_{k-5}(\bm e)$.

\subsection{The 3D Neilan Complex}
\label{sec:neilan3d}

At last, we consider the Neilan complex in three dimensions~\cite{2015Neilan}, which possesses higher global regularity. This is the most technical part in this paper, since the local regularity of the $H^1(\curl)$ finite element space (i.e., $\mathbf{N}_{k-1}^{\curl}$, introduced below) are more complicated than those discussed above, in the sense that some degrees of freedom are more complex, cf. \cite{2015Neilan}. Furthermore, verifying Assumption (B2) for this finite element complex is more challenging. Besides using the exactness of the bubble complex (via the corresponding harmonic inner products) and the algebraic calculations individually, the subsequent verification also relies on a combination of both techniques. This approach appears to be helpful when considering the finite element complexes with even higher global regularity.
Recall the 3D Neilan complex for $k \ge 9$,
\begin{equation}
	\tag{$\mathsf{N3}$}
	\mathbf{Z}_k \xrightarrow{\grad} \mathbf{N}^{\curl}_{k-1} \xrightarrow{\curl} \mathbf{N}^{\div}_{k-2} \xrightarrow{\div} \mathbf{N}^0_{k-3},
\end{equation}
which is a finite element subcomplex of the following continuous Stokes complex in three dimensions:
\begin{equation}
	H^2(\Omega) \xrightarrow{\grad} H^1(\curl, \Omega; \mathbb R^3) \xrightarrow{\curl} H^1(\Omega;\mathbb R^3) \xrightarrow{\div} L^2(\Omega).
\end{equation}
Here, 
\[\begin{split}\mathbf Z_k := \{ u \in C^1(\Omega) ~:~ & u|_{\bm K} \in P_k(\bm K), \forall \bm K \in \mathcal T,\\ & u \text{ is } C^4 \text{ at each vertex of } \mathcal T, \text{ is } C^2 \text{ on each edge of } \mathcal T\},\end{split}\]

\[ \begin{split} \mathbf{N}^{\curl}_{k-1}   := \{ \bm \xi \in [C^0(\Omega)]^3 ~:~ & \bm \xi |_{\bm K} \in [P_{k-1}(\bm K)]^3, \forall \bm K \in \mathcal T, \curl \bm \xi \in [C^0(\Omega)]^3, \\ & \bm \xi \text{ is } C^3 \text{ at each vertex of } \mathcal T, \text{ is } C^1 \text{ on each edge of } \mathcal T, \\ 
	& \operatorname{curl}\bm \xi  \text{ is } C^1 \text{ on each edge of } \mathcal T
	\},
\end{split}\]

\[\begin{split}\mathbf{N}^{\div}_{k-2}   := \{ \bm v \in [C^0(\Omega)]^3 ~:~& \bm v |_{\bm K} \in [P_{k-2}(\bm K)]^3, \forall \bm K \in \mathcal T, \\ & \bm v \text{ is } C^2 \text{ at each vertex of } \mathcal T, \text{ is } C^1 \text{ on each edge of } \mathcal T\},\end{split}\]
and 
\[\begin{split}
	\mathbf{N}^0_{k-3} :=\{ p \in L^2(\Omega) ~:~ & p|_{\bm K} \in P_{k-3}(\bm K), \forall \bm K \in \mathcal T,\\ &p \text{ is } C^1 \text{ at each vertex of } \mathcal T, \text{ is } C^0 \text{ on each edge of } \mathcal T\}.\end{split}\]

The $H^1$ finite element space $\mathbf Z_k$ is defined as follows: for $u \in P_k(\bm K)$, define the following degrees of freedom, for any $\bm K$,
\begin{itemize}
	\item[-] the function value $u(\bm x)$ at each vertex $\bm x$ of $\bm K$;
	\item[-] at each vertex $\bm x$ of $\bm K$, set $f_{\bm x,i}^0(u)$ as the partial derivatives $D^{\alpha}u(\bm x)$ with $1 \le |\alpha| \le 4$;
	\item[-] on each edge $\bm e$ of $\bm K$,
		\begin{itemize}
			\item[$\diamond$] set $f_{\bm e, i }^0(u)$ as $\LL u, b_{z, \bm e}^{i} \RR_{0, \bm e}$, where $b_{z,\bm e}^{i}$, form a basis of $(\lambda_0\lambda_1)^5P_{k-10}(\bm e)$;
			\item[$\diamond$]  set $f_{\bm e, i }^0(u)$ as $\LL \frac{\partial u}{\partial \bm n_{j}}, b_{z, \bm e}^{\ast, i}\RR_{0,\bm e}$, where $b_{z, \bm e}^{\ast, i}$, form a basis of $(\lambda_0\lambda_1)^4P_{k-9}(\bm e)$, $j = 1,2$;
			\item[$\diamond$]  set $f_{\bm e, i}^0(u)$ as $(\frac{\partial^2 u}{\partial \bm n_{j} \partial \bm n_{j'}}, b_{z, \bm e}^{\ast\ast, i})_{\bm e}$, where $b_{z, \bm e}^{\ast\ast,i}$, form a basis of $(\lambda_0\lambda_1)^3P_{k-8}(\bm e)$, $j, j' = 1,2$;
		\end{itemize}
	\item[-] on each face $\bm f$ of $\bm K$, set $f_{\bm f, i}^0(u) = \LL u, b_{z,\bm f}^{i} \RR_{0,\bm f}$, where $b_{z, \bm f}^i$, form a basis of $B_{\bm f, k -9}^2 =  (\lambda_0\lambda_1\lambda_2)^3P_{k-9}(\bm f)$; set $f_{\bm f, i}^0(u) = (\frac{\partial u}{\partial \bm n}, b_{z, \bm f}^{\ast, i})_{\bm f}$, where $b_{z, \bm f}^{\ast, i}$, form a basis of $(\lambda_0\lambda_1\lambda_2)^2P_{k-7}(\bm f)$;
	\item[-] inside $\bm K$, set $f_{\bm f,i}^0(u) = \LL u, b_{z,\bm K}^i \RR_{0,\bm K}$, where $b_{z, \bm K}^i$, form a basis of $B_{\bm K, k - 8}^1 =  (\lambda_0 \lambda_1 \lambda_2 \lambda_3)^2P_{k-8}(\bm K)$.
\end{itemize}

To define the $H(\curl)$ finite element space, the following inner products are introduced:
$$\LL \bm \xi,\bm \eta \RR_{1, \bm f} = (\proj_{\grad_{\bm f} B_{\bm f, k -9}^2  }\bm \xi, \proj_{\grad_{\bm f} B_{\bm f, k -9}^2 }\bm \eta)_{\bm f} + (\rot_{\bm f}  \bm \xi, \rot_{\bm f} \bm \eta)_{\bm f}, \quad\forall \bm \xi, \bm \eta \in [P_{k-1}(\bm f)]^2, $$
and 
$$\LL \bm \xi, \bm \eta \RR_{1, \bm K} = (\proj_{ \grad B_{\bm K, k - 8}^1} \bm \xi, \proj_{\grad B_{\bm K, k - 8}^1} \bm \eta)_{\bm K} + (\curl \bm \xi, \curl \bm \eta)_{\bm K}, \quad\forall \bm \xi, \bm \eta \in [P_{k-1}(\bm K)]^3.$$

The $H(\curl)$ finite element $\mathbf{N}^{\curl}_{k-1}$ is defined as follows. For $\bm \xi \in [P_{k-1}(\bm K)]^3$, define the following degrees of freedom for any $\bm K$,
\begin{itemize}
	\item[-] the moment of the tangential component $(\bm \xi \cdot \bm t, 1)_{\bm e}$ for each edge $\bm e$ of $\bm K$;
	\item[-] at each vertex $\bm x$ of $\bm K$, set $f_{\bm x, i}^{1}(\bm \xi)$ as the zeroth, first, second, and third order derivatives of $\bm \xi$ at vertex $\bm x$;
	\item[-] on each edge $\bm e$ of $\bm K$,
		\begin{itemize}
			\item[(1.a)] set $f_{\bm e, i}^1(\bm \xi) = (\bm \xi \cdot \bm t, b_{n,\bm e}^{\curl, i})_{\bm e}$, where $b_{n, \bm e}^{\curl, i}$, form a basis of $(\lambda_0\lambda_1)^4P_{k-9}(\bm e) / \R$;

			\item[(1.b)] set $f_{\bm e, i}^1(\bm \xi) = \LL\bm \xi\cdot \bm n_{j},  b_{n,\bm e}^{\ast, \curl, i}\RR_{0,\bm e}$, where $b_{n, \bm e}^{\ast, \curl, i}$, form a basis of $(\lambda_0\lambda_1)^4P_{k-9}(\bm e)$, $j = 1,2$;
			\item[(1.c)] set $f_{\bm e, i}^1(\bm \xi) = (\frac{\partial \bm \xi \cdot \bm t}{\partial \bm n_{j}},  b_{n, \bm e}^{\ast 2, \curl, i} )_{\bm e}$, where $b_{n, \bm e}^{\ast 2, \curl, i}$, form a basis of $ (\lambda_0\lambda_1)^3P_{k-8}(\bm e)$, $j = 1,2$;
			\item[(1.d)] set $f_{\bm e, i}^1(\bm \xi) = (\frac{\partial \bm \xi \cdot \bm n_{j'}}{\partial \bm n_{j}}, b_{n, \bm e}^{\ast 2, \curl, i} )_{\bm e}$, where $b_{n, \bm e}^{\ast 2, \curl, i}$, form a basis of $ (\lambda_0\lambda_1)^3P_{k-8}(\bm e)$, $j, j'=  1,2$;
			\item[(1.e)] set $f_{\bm e, i}^1(\bm \xi)  = (\frac{\partial \curl \bm \xi}{\partial \bm n_1}, b_{n, \bm e}^{\ast 3, \curl, i})_{\bm e}$, where $b_{n, \bm e}^{\ast 3, \curl ,i}$, form a basis of $(\lambda_0\lambda_1)^2 [P_{k-7}(\bm e)]^3$;
			\item[(1.f)] set $f_{\bm e, i}^1(\bm \xi)  = ((I - \bm n_2\otimes \bm n_2)\frac{\partial \curl \bm \xi}{\partial \bm n_2}, b_{n, \bm e}^{\ast 3, \curl, i})_{\bm e}$, where $b_{n, \bm e}^{\ast 3, \curl ,i}$, form a basis of $(\lambda_0\lambda_1)^2 [P_{k-7}(\bm e)]^3$;
		\end{itemize}
	\item[-] on each face $\bm f$ of $\bm K$, 
		\begin{itemize}
			\item[(2.a)] set $f_{\bm f, i}^1(\bm \xi) = (\bm \xi\cdot \bm n, b_{n,\bm f}^{\curl, i})_{\bm f}$, where $b_{n, \bm f}^{\curl, i}$, form a basis of $(\lambda_0\lambda_1\lambda_2)^2P_{k-7}(\bm f)$;
			\item[(2.b)] set $f_{\bm f, i}^1(\bm \xi) = \LL E_{\bm f} \bm \xi, b_{n, \bm f}^{\ast, \curl, i}\RR_{1,\bm f}$, where $b_{n, \bm f}^{\ast, \curl, i}$, form a basis of $(\lambda_0\lambda_1\lambda_2)^2  [B^{BDM}_{\bm f, k-8}]^{\perp}$;

			\item[(2.c)] set $f_{\bm f, i}^1(\bm \xi) = (E_{\bm f}\curl \bm \xi, b_{n, \bm f}^{\ast\ast, \curl, i})_{\bm f}$, where $b_{n, \bm f}^{\ast\ast, \curl, i}$, form a basis of $(\lambda_0\lambda_1\lambda_2)^2 P_{k-8}(\bm f)$;
			      % $B_{n, \bm f}^{\ast\ast, \curl, i}$.
			      % \item Set $f_{\bm f, i}^1(\bm v) = \LL E_f v, b_{n,\bm f}^{\ast, \curl,i}\RR_{1, \bm f},$ where $b_{n,\bm f}^{\ast, \curl,i} $, form a basis of $(\lambda_0\lambda_1\lambda_2)^2P_{k-8}(\bm f)$.\lt{Not sure here.}
		\end{itemize}
	\item[-] inside $\bm K$, set $f_{\bm K,i}^1(\bm \xi) = \LL \bm \xi, b_{n, \bm K}^{\curl, i} \RR_{1, \bm K}$, where $b_{n, \bm K}^{\curl, i}$, form a basis of $B_{n, \bm K}^{\curl} = \{ \bm \eta \in [P_{k-1}(\bm K)]^3 : \bm \eta = \curl \bm \eta = 0 \text{ on } \partial \bm K\}$.
\end{itemize}

The unisolvency of the above degrees of freedom with respect to $[P_{k-1}(\bm K)]^3$ is shown in the appendix.

To define the $H(\div)$ finite element space, the following inner product is introduced. 
For $\bm v, \bm z \in [P_{k-2}(\bm K)]^3$, define 
$$ \LL \bm v, \bm z\RR_{2, \bm K} = (\mathcal P_{\curl B_{n,\bm K}^{\curl}} \bm v, \mathcal P_{\curl B_{n,\bm K}^{\curl}} \bm z)_{\bm K} + (\div \bm v, \div \bm z)_{\bm K}.$$

The $H(\div)$ finite element space $\mathbf{N}^{\div}_{k-2}$ is defined as follows: for $\bm v \in [P_{k-2}(\bm K)]^3$, define the following degrees of freedom, for any $\bm K$,
\begin{itemize}
	\item[-] the moment of the normal component $(\bm v\cdot \bm n,1)_{\bm f}$ for each face $\bm f$ of $\bm K$;
	\item[-] at each vertex $\bm x$ of $\bm K$, let $f_{\bm x, i}^2(\bm v)$ be the function value, first order and second order derivatives of $\bm v$ at $\bm x$;
	\item[-] on each edge $\bm e$ of $\bm K$, set $f_{\bm e, i}^2(\bm v) = ( \bm v ,b_{n,\bm e}^{\div, i})_{\bm e}$, where $b_{n, \bm e}^{\div, i}$, form a basis of $B_{n, \bm e}^{\div} = (\lambda_0\lambda_1)^3[P_{k-8}(\bm e)]^3$; 
		set $f_{\bm e, i}^2(\bm v) = (\frac{\partial \bm v}{\partial \bm n_j}, b_{n, \bm e}^{\ast, \div, i})_{\bm e}$, where $b_{n, \bm e}^{\ast, \div, i}$, form a basis of $B_{n, \bm e}^{\ast, \div} = (\lambda_0\lambda_1)^2[P_{k-7}(\bm e)]^3$;
	\item[-] on each face $\bm f$ of $\bm K$, set $f_{\bm f,i}^2(\bm v) = (\bm v\cdot \bm n, b_{n,\bm f}^{\ast, \div, i})_{\bm f}$, where $b_{n,\bm f}^{\ast, \div, i}$, form a basis of $(\lambda_0\lambda_1\lambda_2)^2P_{k-8}(\bm f)/\mathbb R$; set $f_{\bm f, i}^2(\bm v) = (E_{\bm f}\bm v, b_{n, \bm f}^{\div, i})_{\bm f}$, where $b_{n, \bm f}^{\div,i}$, form a basis of $(\lambda_0\lambda_1\lambda_2)^2P_{k-8}(\bm f)$;
	\item[-] inside $\bm K$, set $f_{\bm K, i}^2(\bm v) = \LL \bm v,  b_{n, \bm K}^{\div, i}\RR_{2,\bm K}$, where $b_{n, \bm K}^{\div, i}$, form a basis of $[(\lambda_0\lambda_1\lambda_2\lambda_3) P_{k-6}(\bm K)]^3.$
\end{itemize}

The $L^2$ finite element space $\mathbf{N}^0_{k-3}$ is defined as follows: for $p \in P_{k-3}(\bm K)$, define the following degrees of freedom, for any $\bm K$,
\begin{itemize}
	\item[-] the moment, $(p,1)_{\bm K}$;
	\item[-] at each vertex $\bm x$ of $\bm K$, set $f_{\bm x, i}^3(p)$ as the function value and first order derivative of $p$ at $\bm x$;
	\item[-] on each edge $\bm e$ of $\bm K$, set $f_{\bm e,i}^3(p) = (p, b_{n,\bm e}^{0,i})_{\bm e}$, where $b_{n, \bm e}^{0,i}$, form a basis of $(\lambda_0\lambda_1)^2P_{k-7}(\bm e)$;
	\item[-] inside the element $\bm K$, set $f_{\bm K, i}^3(p) = (p, b_{n, \bm K}^{0,i})_{\bm K}$, where $b_{n,\bm K}^{0,i}$, form a basis of $B_{n, \bm K}/\mathbb R$ with $$B_{n,\bm K} = \{ q \in P_{k-3}(\bm K): q|_{\bm e} = 0 \text{ on each } \bm e \text{ on } \bm K\}.$$
\end{itemize}

The corresponding edge bubble complex
is
$$(\lambda_0\lambda_1)^5 P_{k-10}(\bm e) \xrightarrow{\partial/\partial \bm t} (\lambda_0\lambda_1)^4P_{k-9}(\bm e) /\mathbb R,$$
the face bubble complex is
\begin{equation}
	\label{eq:neilan-facebubble}(\lambda_0\lambda_1\lambda_2)^3P_{k-9}(\bm f) \xrightarrow{\curl_{\bm f}} (\lambda_0\lambda_1\lambda_2)^2 B^{BDM}_{\bm f, k-7}  \xrightarrow{\div_{\bm f}} (\lambda_0\lambda_1\lambda_2)^2P_{k-8}(\bm f)/\mathbb R,
\end{equation}
and the element bubble complex is
\begin{equation}
	\label{eq:neilan-cellbubble}
	(\lambda_0 \lambda_1 \lambda_2 \lambda_3)^2P_{k-8}(\bm K)\xrightarrow{\grad} B_{n, \bm K}^{\curl} \xrightarrow{\curl} [(\lambda_0\lambda_1\lambda_2\lambda_3) P_{k-6}(\bm K)]^3 \xrightarrow{\div} B_{n,K}/\mathbb{R}.
\end{equation}

The cell bubble complex \eqref{eq:neilan-cellbubble} is much more complicated, whose exactness has been shown in \cite{2015Neilan}. Here only the exactness of the face bubble complex will be proved.

\begin{lemma}
	The polynomial sequence
	\eqref{eq:neilan-facebubble} is an exact complex.
\end{lemma}
\begin{proof}
	Let $\lambda_0,\lambda_1,\lambda_2$ be the three barycenter coordinates of $\bm f$, and set $b = \lambda_0\lambda_1\lambda_2$. Given $u = b^3u_1$ for some $u_1 \in P_{k-9}(\bm f)$, it follows that
	$$ \curl_{\bm f} u = b^2(3u_1 \curl_{\bm f} b + b \curl_{\bm f} u_1).$$
	Since $\curl_{\bm f} b \cdot \bm n |_{\bm e} = \frac{\partial b}{\partial \bm t} |_{\bm e} = 0$, on each edge $\bm e$ of $\bm f$ with the unit normal vector $\bm n$ and the corresponding tangential vector $\bm t$, it then follows that $\curl_{\bm f} u \in b^2 B^{BDM}_{\bm f, k-7}.$
	Similarly, given $\bm \xi = b^2 \bm \xi_1$ for some $\bm \xi_1 \in B^{BDM}_{\bm f, k-7}$, then it holds that 
	$$\div (b^2 \bm \xi_1) = 2b \grad b\cdot \bm \xi_1 + b^2 \div \bm \xi_1.$$
	It then suffices to check that $\grad b\cdot \bm \xi_1$ vanishes on all edges. Since
	$\grad b \cdot \bm \xi_1 = \frac{\partial b}{\partial \bm t} (\bm \xi_1 \cdot \bm t) + \frac{\partial b}{\partial \bm n}(\bm \xi_1 \cdot \bm n) = 0$ on edge $\bm e$ of $\bm f$, it holds that $\div \bm \xi \in b^2P_{k-8}(\bm f)$. Further, it follows from $\int_{\bm f}\div\bm \xi = \int_{\bm f}\bm \xi \cdot \bm n = 0$ that $\div \bm \xi \in b^2P_{k-8}(\bm f) /\mathbb R.$ This proves that the sequence in \eqref{eq:neilan-facebubble} is a complex.

Similarly, it can be shown that if $\bm \xi \in b^2 B^{BDM}_{\bm f, k-7}$ satisfies $\div \bm \xi = 0$, then there exists $\varphi \in  b^3 P_{k-9}(\bm f)$ such that $\curl \varphi = \bm \xi.$

To show the exactness, it then suffices to count the dimensions. It follows from \cite{2013BoffiBrezziFortin} that
	$$\dim B^{BDM}_{\bm f, k - 7} = (k-6)(k-8),$$
	and that 
	$$\dim P_{k-9} + \dim P_{k-8} - 1 = \frac{1}{2}(k-8)(k-7) + \frac{1}{2}(k-7)(k-6) - 1 = (k-6)(k-8) = \dim B^{BDM}_{\bm f, k - 7},$$
	which imply the exactness of \eqref{eq:neilan-facebubble}.

\end{proof}

To check Assumption (B2), it suffices to check those degrees of freedom which are {\bf not} defined by the harmonic inner product.

Similar to that in the previous subsection, this is reduced to the algebraic calculation. For $\nabla \varphi_{\bm x}$, it includes the following cases:

\begin{itemize}
\item[-] (1.b): $f_{\bm e, i}^1(\nabla \varphi_{\bm x}) =  \LL \frac{\partial}{\partial \bm n_j}  \varphi_{\bm x},  b_{n,\bm e}^{\ast, \curl, i}\RR_{0,\bm e} = 0,$ for $ b_{n,\bm e}^{\ast, \curl, i} \in (\lambda_0\lambda_1)^4 P_{k-9}(\bm e)$;
\item[-] (1.d): $f_{\bm e, i}^1(\nabla \varphi_{\bm x}) = (\frac{\partial^2}{\partial \bm n_{j} \partial \bm n_{j'}}\varphi_{\bm x}, b_{n, \bm e}^{\ast 2, \curl, i} )_{\bm e} = 0$, for $b_{n, \bm e}^{\ast 2, \curl, i} \in (\lambda_0\lambda_1)^3[P_{k-8}(\bm e)]^3$;
\item[-] (1.e) and (1.f): they vanish directly from $\curl \nabla \varphi_{\bm x} = 0$;  
\item[-] (2.a): $f_{\bm e, i}^1(\nabla \varphi_{\bm x}) = (\frac{\partial}{ \partial \bm n} \varphi_{\bm x}, b_{n, \bm f}^{\curl, i})_{\bm f} = 0$ for $b_{n, \bm f}^{\curl, i}\in (\lambda_0\lambda_1\lambda_2)^2P_{k-7}(\bm f)$;
\item[-] (2.c): they vanish directly from $\curl \nabla \varphi_{\bm x} = 0$. 

\end{itemize}

Note that there exist cases that cannot be reduced to algebraic calculations, which are listed below:

For (1.c), an integration by parts yields that 
$(\frac{\partial^2}{\partial \bm t \partial \bm n_j} \varphi_{\bm x}, 1)_{e} = 0$.
Consider $f_{\bm e, i}^1(\nabla \varphi_{\bm x}) = (\frac{\partial^2}{\partial \bm t \partial \bm n_j} \varphi_{\bm x}, \bm b_{n, \bm e}^{\ast 2, \curl, i} )_{\bm e}$, for $\bm b_{n, \bm e}^{\ast 2, \curl, i} \in (\lambda_0\lambda_1)^3[P_{k-8}(\bm e)] / \mathbb R$. Since $\frac{\partial}{\partial \bm t}: (\lambda_0\lambda_1)^4P_{k-9}(\bm e) \to (\lambda_0\lambda_1)^3[P_{k-8}(\bm e)] / \mathbb R$ is a surjection, there exists $b_1$ such that $\frac{\partial}{\partial \bm t}b_1 = \bm b_{n, \bm e}^{\ast 2, \curl, i}$. As a result,  
$ f_{\bm e, i}^1(\nabla \varphi_{\bm x})= \LL \frac{\partial}{\partial \bm n_{j}}\varphi_{x}, b_1\RR_{0,\bm e} = 0 .$ 

For $\curl \varphi_{\bm e}$, it includes the following cases:

\begin{itemize}
\item[-] For the first part of $f_{\bm e,i}^2(\cdot)$, note that 
$$\curl \bm \xi = (\frac{\partial}{\partial \bm n_2}(\bm \xi \cdot \bm n_1) - \frac{\partial}{\partial \bm n_1}(\bm \xi \cdot \bm n_2))\bm t + (\frac{\partial}{\partial \bm t}(\bm \xi \cdot \bm n_2) - \frac{\partial}{\partial \bm n_2}(\bm \xi \cdot \bm t))\bm n_1 + (\frac{\partial}{\partial \bm t}(\bm \xi \cdot \bm n_1) - \frac{\partial}{\partial \bm n_1}(\bm \xi \cdot \bm t))\bm n_2.
$$
It then follows from (1.b), (1.c) and (1.d) and 
$\frac{\partial}{\partial \bm t}(\bm \varphi_{\bm e} \cdot \bm n_j)|_{\bm e} = 0$ that $f_{\bm e,i}^2(\curl \varphi_{\bm e}) = 0$.
\item[-] For the second part of $f_{\bm e,i}^2(\cdot)$, note that
$$\bm n_2 \cdot \frac{\partial(\curl \bm \xi)}{\partial \bm n_2} =  - \bm t \cdot \frac{\partial(\curl \bm \xi)}{\partial \bm t}  - \bm n_1 \cdot \frac{\partial(\curl \bm \xi)}{\partial \bm n_1}.$$
It follows from the above equation, (1.e) and (1.f), that it remains to check 
$$(\frac{\partial}{\partial \bm t}(\curl\varphi_{\bm e} \cdot \bm t), b)_{\bm e} = 0,\quad\forall b \in (\lambda_0\lambda_1)^2P_{k-7}(\bm e).$$
The first part of $f_{\bm e,i}^2(\cdot)$ and the fact $\partial/\partial \bm t : (\lambda_0\lambda_1)^3P_{k-8}(\bm e) \to (\lambda_0\lambda_1)^2P_{k-7}(\bm e) / \mathbb R$ imply that it remains to check
$$(\frac{\partial}{\partial \bm t}(\curl\varphi_{\bm e} \cdot \bm t), 1)_{\bm e} = 0,$$ which comes from the definition of $\varphi_{\bm e}.$

\end{itemize}

Finally, it remains to check that the degrees of freedom $f_{\bm e,i}^3(\div\varphi_{\bm f}) = 0$. Note that this is implied by the second part of $f_{\bm e,i}^2(\varphi_{\bm f}) = 0$ and $\varphi_{\bm f}|_{\bm e} =0$.

\section{Concluding Remarks}
In this paper, we construct local bounded commuting interpolation operators from the continuous de Rham complex to the (nonstandard) finite element de Rham complex. We develop the theory for general meshes in two and three dimensions, and provide various examples. Similar frameworks can be established in higher dimensions,though rare finite element subcomplexes with additional regularity can be found in the literature. Moreover, the applications discussed in this paper can be extended to finite element subcomplexes on macro-elements.
\bibliographystyle{plain}
\bibliography{reference}
\newpage
\appendix
\renewcommand{\theequation}{{\Alph{section}.\arabic{equation}}}

The appendix contains two sections. \Cref{sec:proof-2d} proves \Cref{thm:main-2d}, and \Cref{sec:neilanh1curl} shows the technical details in \Cref{sec:neilan3d}.

\section{Proof of \Cref{thm:main-2d}}
\label{sec:proof-2d}

This section constructs $\pi^i, i = 0,1,2,$ in two dimensions, and the local estimates are similar to that in three dimensions, see \Cref{sec:proof}.
Given any patch $\omega$, the auxiliary operators can be summarized in the following diagram, which is a concrete example of \eqref{eq:cd:splitproj-long}:
\begin{equation*}
	\xymatrix{
	& H^1(\omega)\ar@<-1mm>@{.>}[dl]|-{\cR^0_{\omega}} \ar[r]^-{\curl} \ar@<-1mm>[d]^-{\mathcal Q^0_{\omega}} & H(\div, \omega; \mathbb R^2) \ar@<-1mm>@{.>}[dl]|-{\cR^1_{\omega}} \ar[r]^-{\div} \ar@<-1mm>[d]^-{\mathcal Q^1_{\omega}} &L^2(\omega) \ar@<-1mm>@{.>}[dl]|-{\mathcal R^2_{\omega}}  & \\
	\mathbb R \ar[r] & \Lambda_h^0(\omega) \ar[r]^-{\curl} & \Lambda_h^1(\omega) \ar[r]^-{\div}  &\Lambda_h^2(\omega) \ar[r] & 0.
	}
\end{equation*}

\subsection{Construction of $\pi^0$}

\begin{equation}
	\label{eq:pi0harmonic}
	\xymatrix{
	&  & H^1(\omega) \ar@<-1mm>@{.>}[dl]|-{\cR^0_{\omega}} \ar[r]^-{\curl} \ar@<-1mm>[d]^-{\mathcal Q^0_{\omega}} &H(\div,\omega;\mathbb R^2) \ar@<-1mm>@{.>}[dl]|-{\mathcal R^1_{\omega}}  & \\
	0 \ar[r] & \mathbb R \ar[r]^-{\subseteq} & \Lambda_h^0(\omega) \ar[r]^-{\curl}  &\Lambda_h^1(\omega) \ar[r] & 0.
	}
\end{equation}

To construct the projection operator $\pi^0$, for a scalar function $u \in H^1(\omega)$,
for any local patch $\omega$, consider the following local projection
\begin{equation}
	\label{eq:cQ0u-2d}
\cQ_{\omega}^0 u = \cM_{\omega} u + \cR_{\omega}^1 \curl u.
\end{equation}
Here $\cM_{\omega} u = \cR^0_{\omega} u : = \fint_{\omega} u \in \mathbb R$ is the (local) integral mean of $u$ over $\omega$, and $\cR_{\omega}^1 \xi \perp \mathbb R$ is defined for all $\bm\xi \in H(\div,\omega; \mathbb R^2)$ such that
\begin{equation}
	\label{eq:R1-2d}
	(\operatorname{curl}\cR_{\omega}^1 \bm \xi,\operatorname{curl}v)_{\omega} = (\bm \xi, \curl v),\quad \forall v\in \Lambda_{h}^0(\omega).
\end{equation}
A direct consequence is $\Hproj_{\omega}^0 u = u$ for $u \in \Lambda_h^0(\omega)$. 

Then, given $u \in H^1(\Omega)$, the projection $\pi^0 u \in \Lambda_h^0$ is defined as
\begin{equation}\label{eq:defn:pi0-2d} \pi^0 u = \sum_{\sigma} L_{\sigma}^0\cQ_{\sigma}^0u + \sum_{\bm x} (\cQ _{\bm x}^0 u)(\bm x) \varphi_{\bm x}.\end{equation}

For $u \in \Lambda_h^0$, it holds that $(\Hproj^0_{\sigma} u)|_{\omega_{\sigma}} = u|_{\omega_{\sigma}}$. Therefore,  by \eqref{eq:dofid-u}, it holds that
$$\pi^0 u = \sum_{\sigma}L^0_{\sigma} u + \sum_{x} u(\bm x)\varphi_{\bm x} = u.$$ This is to say $\pi^0$ is a projection operator.
It is worth mentioning that the value of $\pi^0 u$ on $\omega_{\sigma}$ is determined by the value of $u$ on $\omega^{[1]}_{\sigma}$.

\subsection{Construction of $\pi^1$}

By Assumption (A3) that $L_{\sigma}^0$ vanishes for any constant in $\mathbb R$, it now follows from \eqref{eq:cQ0u-2d} and the definition of $\mathcal M^0 u$ that $L_{\sigma}^0 \mathcal Q_{\sigma}^0 u = L_{\sigma}^0 \cR_{\sigma}^1 \curl u.$ Taking $\curl$ in \eqref{eq:defn:pi0-2d} yields that,
\begin{equation}
	\label{eq:curlpi0u-2d}
	\begin{split}
		\curl \pi^0 u = & \sum_{\sigma}\curl  L_{\sigma}^0{\cQ}_{\sigma}^0u + \sum_{\bm x} \cQ_{\bm x}^0 u (\bm x)\curl \varphi_{\bm x} \\
		= &  \sum_{\sigma}\curl  L_{\sigma}^0\cR_{\bm x}^1 \curl u + \sum_{\bm x}  (\cR_{\sigma}^1 \curl u)(\bm x)\curl \varphi_{\bm x} + \operatorname{curl}\avg^0 u,
	\end{split}
\end{equation}
where $\avg^0u  = \sum_{x} \avg_{\bm x} u \varphi(\bm x)$ is defined in \Cref{prop:dblcmplx}. It follows from \Cref{prop:dblcmplx} that
$\curl \avg^0 u= \avg^1 \curl u$ for all $u \in H^1(\Omega)$, where $\avg^1$ is also defined in \Cref{prop:dblcmplx}.

This motivates to define an operator $\widehat{\pi}^1: L^2(\Omega;\R^2)\rightarrow \Lambda_{h}^1$ by
\begin{equation}
	\label{eq:hatpi1-2d}
	\widehat{\pi}^1 \bm v = \sum_{\sigma}  \curl L_{\sigma}^0 \mathcal R_{\sigma}^1 \bm v + \sum_{\bm x} (\mathcal R_{\bm x}^1 \bm v)(\bm x)\curl \varphi_{\bm x} + \avg^1 \bm v.
\end{equation}
Note that the value of $\widehat{\pi}^1 u $ on $\omega_{\sigma}$ is determined by the value of $u$ on $\omega_{\sigma}^{[1]}.$

Then, it follows from \eqref{eq:curlpi0u-2d} and \eqref{eq:hatpi1-2d} that $ \widehat{\pi}^1 \curl u = \curl \pi^0 u$ for any $u\in H^1(\omega)$. In particular, for any $u\in H^1(\omega_{\sigma}^{[1]})$, it holds that $ \widehat{\pi}^1 \curl u|_{\omega_{\sigma}} = \curl \pi^0 u|_{\omega_{\sigma}}$.

To modify $\widehat{\pi}^1$ to construct $\pi^1$, consider the following diagram
\begin{equation*}
	\xymatrix{
	&  & H(\div, \omega; \mathbb R^2) \ar@<-1mm>@{.>}[dl]|-{\cR^1_{\omega}} \ar[r]^-{\div} \ar@<-1mm>[d]^-{\mathcal Q^1_{\omega}} &L^2(\omega) \ar@<-1mm>@{.>}[dl]^-{\mathcal R^2_{\omega}}  & \\
	\mathbb R \ar[r] & \Lambda_h^0(\omega) \ar[r]^-{\curl} & \Lambda_h^1(\omega) \ar[r]^-{\div}  &\Lambda_h^2(\omega) \ar[r] & 0.
	}
\end{equation*}
	Here $\Lambda_h^2(\omega)$ is the restriction of $\Lambda_h^2$ on the patch $\omega$. Given $\bm v \in H(\div,\Omega;\mathbb R^2)$,
define the interpolation $\cQ^1_{\omega} \bm v := \curl \bm \cR_{\omega}^1 \bm v+ \cR_{\omega}^2 \div \bm v$, where $\cR_{\omega}^2 : L^2(\Omega) \to \Lambda_h^1(\omega)$ is defined for $p \in L^2(\Omega)$ such that $\cR^2_{\omega} p \perp \curl \Lambda_h^0(\omega)$ and
\begin{equation}\label{eq:R2-2d}(\div \cR_{\omega}^2 p, \div \bm w)_{\omega} = (p, \div \bm w)_{\omega}.\end{equation}

Thus, define the following modified interpolation of $\hat{\pi}^1 \bm v$, for any $\bm v \in H(\div,\Omega;\mathbb R^2)$,
\begin{equation}
	\label{eq:pi1-2d}
	\pi^1 \bm v = \widehat{\pi}^1 \bm v + \sum_{\sigma} L _{\sigma}^1(I - \widehat{\pi}^1) \cQ_{\sigma,[1]}^1  \bm v + \sum_{\bm e} ( (I - \widehat{\pi}^1) \cQ_{\bm e,[1]}^1 \bm v,\bm n)_{\bm e}\varphi_{\bm e},
\end{equation}
where $\cQ^1_{\sigma, [\ell]}$ is the harmonic interpolation operator $\cQ^1$ with respect to the patch $\omega_{\sigma}^{[\ell]}$. Note that the value $\widehat{\pi}^1\bm v$ on $\omega_{\sigma}$ only depends on the value of $\bm v$ on $\omega_{\sigma}^{[1]}.$ Therefore, $\pi^1$ defined in \eqref{eq:pi1-2d} is well-defined. We first show that $\pi^1$ is a projection operator. For any $\bm v \in \Lambda_h^1$, $(\cQ_{\sigma,[1]} \bm v )|_{\omega_{\sigma}^{[1]}}= \bm v|_{\omega_{\sigma}^{[1]}}$ according to the definition of the harmonic projection, and hence
$$ \pi^1 \bm v = \widehat \pi^1 \bm v + \sum_{\sigma} L _{\sigma}^1(I - \widehat{\pi}^1) \bm v + \sum_{\bm e} ( (I - \widehat{\pi}^1)\bm v,\bm n)_{\bm e}\varphi_{\bm e} = \widehat \pi^1 \bm v+ (I - \widehat \pi^1) \bm v = \bm v.$$

Substituting the expression of $\cQ^1_{\omega} \bm v$ into \eqref{eq:pi1-2d} yields that
\begin{equation*}
	\begin{split}
		(I - \widehat \pi^1) \cQ_{\sigma, [1]} \bm v  = &  (I - \widehat \pi^1 ) \curl \cR^1_{\sigma, [1]} \bm v + (I - \widehat \pi^1) \cR_{\sigma, [1]}^2\div \bm v \\
		= & \curl (I - \pi^0)\cR^1_{\sigma, [1]} \bm v + (I - \widehat \pi^1) \cR_{\sigma, [1]}^2\div \bm v \\
		= & (I - \widehat \pi^1) \cR_{\sigma, [1]}^2\div \bm v.
	\end{split}
\end{equation*}
Therefore, the projection $\pi^1 \bm v$ can be simplified as
\begin{equation}
	\label{eq:pi1-2ds}
	\pi^1 \bm v = \widehat{\pi}^1 \bm v + \sum_{\sigma} L _{\sigma}^1(I - \widehat{\pi}^1) \cR_{\sigma,[1]}^2  \div \bm v + \sum_{\bm e} ( (I - \widehat{\pi}^1) \cR_{\bm e,[1]}^2 \div \bm v,\bm n)_{\bm e}\varphi_{\bm e}.
\end{equation}
This and the fact that $\hat\pi^1\curl u=\curl \pi^0 u $ lead to the crucial commuting property: $\pi^1 \curl u = \curl \pi^0 u, \forall u \in H^1(\Omega).$

\subsection{Construction of $\pi^2$}
From the definition of $\widehat{\pi}^1$ in \eqref{eq:hatpi1-2d}, it follows that $\div \widehat{\pi}^1 \bm v = \div \avg^1 \bm v = \avg^2\div \bm v$ for any $\bm v \in H(\div,\Omega)$.
Therefore,
$$\div \pi^1 \bm v = \div \avg^1 \bm v + \sum_{\sigma} \div L _{\sigma}^1(I - \widehat{\pi}^1) \cR_{\sigma,[1]}^2  \div \bm v + \sum_{\bm e} ( (I - \widehat{\pi}^1) \cR_{\bm e,[1]}^2 \div \bm v,\bm n)_{\bm e}\div \varphi_{\bm e}.$$

This motivates to define for any $p \in L^2(\Omega)$ the interpolation $\pi^2 p \in \Lambda_h^2$ as follows,
$$\pi^2 p  =  \avg^2 p + \sum_{\sigma} \div L _{\sigma}^1(I - \widehat{\pi}^1) \cR_{\sigma,[1]}^2  p + \sum_{\bm e} ( (I - \widehat{\pi}^1) \cR_{\bm e,[1]}^2 p,\bm n)_{\bm e}\div \varphi_{\bm e}.$$
Clearly, such a construction satisfies the commuting property, that is,
\begin{equation}
	\pi^2 \div \bm v = \div \pi^1 \bm v.
\end{equation}

Next, it will be shown that $\pi^2$ is a projection operator. Here we use the fact that for any $p \in \Lambda_h^2$, there exists some $v \in \Lambda_h^1$ such that $p = \div \bm v$. Hence,
$$\pi^2 p = \pi^2 \div \bm v = \div \pi^1 \bm v = \div \bm v = p.$$

\section{Technical Details for the Neilan Complex in 3D}
\label{sec:neilanh1curl}
The remaining technical details in \Cref{sec:neilan3d} are provided here.
First, we prove that the degrees of freedom are equivalent to those in \cite{2015Neilan}.
On each face $\bm f$, recall the following degrees of freedom from those in $\mathbf{N}^{\curl}_k$, in \Cref{sec:neilan3d}:

\begin{equation}
	\label{eq:neilan-dof1}
	(\bm \xi\cdot \bm n, b)_{\bm f }, \quad \forall b \in (\lambda_0\lambda_1\lambda_2)^2P_{k-7}(\bm f),
\end{equation}
\begin{equation}
	\label{eq:neilan-dof2}
	\LL E_{\bm f} \bm \xi, b\RR_{1,\bm f},\quad \forall b \in (\lambda_0\lambda_1\lambda_2)^2 [B^{BDM}_{\bm f, k-7}]^{\perp},
\end{equation}
\begin{equation}
	\label{eq:neialn-dof3}
	(E_{\bm f}\curl \bm \xi, b)_{\bm f}, \forall b \in (\lambda_0\lambda_1\lambda_2)^2P_{k-8}(\bm f),
\end{equation}

while those in \cite{2015Neilan} are 
$$
	(\bm \xi\cdot \bm n, b)_{\bm f }, \quad \forall b \in P_{k-7}(\bm f),
$$
$$(E_f \bm \xi, b)_{\bm f}, \quad\forall b \in RT_{k-9}(\bm f),$$
$$	(E_{\bm f}\curl \bm \xi, b)_{\bm f}, \forall b \in P_{k-8}(\bm f).$$

In what follows, we prove that the degrees of freedom are unisolvent with respect to the shape function space, and the proof of the unisolvency implies that the resulting finite element space is just $\mathbf{N}_{k-1}^{\curl}$, namely, for finite element functions $\bm \xi$ in this finite element space, it holds that $\bm \xi \in C^0(\Omega)$ is $C^3$ at each vertex of $\mathcal T$, $C^1$ on each edge of $\mathcal T$,  and $\curl \bm \xi \in C^0(\Omega)$ is $C^1$ on each edge of $\mathcal T$. The following proof is similar to that in \cite{2015Neilan}. Precisely speaking, \Cref{lem:edge} is parallel to \cite[Lemma 3.2]{2015Neilan}, \Cref{lem:face} is parallel to \cite[Lemma 3.3]{2015Neilan}, and \Cref{thm:unisolvent} is parallel to \cite[Theorem 3.5]{2015Neilan}.

\begin{lemma}
	\label{lem:edge}
	Suppose that $\bm \xi \in [P_{k-1}(\bm K)]^3$ vanishes for the degrees of freedom $(\bm \xi \cdot \bm t_e, 1)_{\bm e}$, $f_{\bm x,i}^1(\bm \xi)$, $f_{\bm e, i}^1(\bm \xi)$, defined in \Cref{sec:neilan3d}. Then $\bm \xi, \nabla \bm \xi$ and $\nabla (\curl \bm \xi)$ vanish on all the edges of $\bm K$. 
\end{lemma}
\begin{proof}
It is easy to see that $\bm \xi$ and $\nabla \bm \xi$ vanish on all the edges of $\bm K$. For the term $\nabla(\curl \bm \xi)$, it follows from the vanishing degrees of freedom that 
$\frac{\partial(\curl \bm \xi)}{\partial \bm n_1}$, $\bm t \cdot \frac{\partial(\curl \bm \xi)}{\partial \bm n_2}$, and $\bm n_1 \cdot \frac{\partial(\curl \bm \xi)}{\partial \bm n_2}$ vanish on all the edges. Furthermore, a direct calculation yields that 
$$\bm n_2 \cdot \frac{\partial(\curl \bm \xi)}{\partial \bm n_2} =  - \bm t \cdot \frac{\partial(\curl \bm \xi)}{\partial \bm t}  - \bm n_1 \cdot \frac{\partial(\curl \bm \xi)}{\partial \bm n_1} = 0.$$ Therefore, $\nabla(\curl \bm \xi)$ vanishes on all the edges of $\bm K$. See \cite[Theorem 3.5]{2015Neilan} for more details.
\end{proof}

\begin{lemma}
	\label{lem:face}
	Suppose that $\bm \xi \in [P_{k-1}(\bm K)]^3$ vanishes for the degrees of freedom $(\bm \xi \cdot \bm t, 1)_{\bm e}$, $f_{\bm x,i}^1(\bm \xi)$, $f_{\bm e, i}^1(\bm \xi)$, $f_{\bm f,i}^1(\bm \xi)$, defined in \Cref{sec:neilan3d}. Then $\bm \xi = \curl \bm \xi = 0$ on each face $\bm f$ of $\bm K$.
\end{lemma}
\begin{proof}
	By \Cref{lem:edge}, it holds that \begin{equation}\label{eq:neilan-vanishingedge}\bm \xi, \nabla \bm \xi, \nabla(\curl \bm \xi)\text{ all vanish on the boundary }\partial \bm f \text{ of face } \bm f .
	\end{equation} Then it follows that $\bm \xi \cdot \bm n_{\bm f} = (\lambda_0\lambda_1\lambda_2)^2 z_1$ for some $z_1 \in P_{k-7}(\bm f)$, with $\bm n_{\bm f}$ the unit normal vector of $\bm f$. It then follows from \eqref{eq:neilan-dof1} that $\bm \xi \cdot \bm n_{\bm f} = 0$ on $\bm f$. A similar argument shows that $E_{\bm f}\curl \bm \xi = 0.$

	Now we show that $E_{\bm f} \bm \xi = 0$. Denote by $\bm t_{\bm e}$ ($\bm n_{\bm e}$) the unit tangential (normal) vector of the edge $\bm e$ on face $\bm f$. From \eqref{eq:neilan-vanishingedge} it holds that $E_{\bm f} \bm \xi = (\lambda_0\lambda_1\lambda_2)^2\bm \xi_{\bm f}$ for some $\bm \xi_{\bm f} \in [P_{k-7}(\bm f)]^2$, it then suffices to show that $\bm \xi_f \cdot \bm t_{e}$ vanishes on all the edges $\bm e$ of $\partial \bm f$. Let $b_{\bm f} = \lambda_0\lambda_1\lambda_2$ be the bubble function on $\bm f$. On each $\bm e$ of $\bm f$, it holds that
	\begin{equation}
		\begin{split}
			0  &= \frac{\partial}{\partial \bm n_{\bm e}} (\curl \bm \xi \cdot \bm n_{\bm f}) = \frac{\partial}{\partial \bm n_{\bm e}} (\rot_{f} E_{\bm f} \bm \xi) \\
			& = \frac{\partial}{\partial \bm n_{\bm e}}  (2b_{\bm f} \curl_{\bm f} b_{\bm f}\cdot \bm \xi_{\bm f} + b^2 \rot_{\bm f} \bm \xi_{\bm f}).
		\end{split}
	\end{equation}
	Restricted on $\bm e$, the second term vanishes. Therefore,
	\begin{equation}
		\begin{split}
			0 & = \frac{\partial}{\partial \bm n_{\bm e}} (2b_{\bm f} \curl_{\bm f} b_{\bm f} \cdot \bm \xi_{\bm f}) \\
			&  = 2(\frac{\partial b_{\bm f}}{\partial \bm n_{\bm e}})^2 (\bm \xi_{\bm f}\cdot \bm t_{\bm e}) - 2\frac{\partial b_{\bm f}}{\partial \bm n_{\bm e}}\frac{\partial b_{\bm f}}{\partial \bm t_{\bm e}}(\bm \xi_{\bm f}\cdot \bm n_{\bm e}) + 2b_{\bm f} \frac{\partial}{\partial \bm n_{\bm e}} (\curl_{\bm f} b_{\bm f} \cdot \bm \xi_{\bm f}).
		\end{split}
	\end{equation}
	As a result, the second and thrid term vanish, then $\bm \xi_{\bm f} \cdot \bm t_{\bm e} = 0$, hence $\bm \xi_{\bm f} \in [B^{BDM}_{\bm f, k-7}]^{\perp}$, and it follows from \eqref{eq:neilan-dof2} that $\bm \xi = 0$ on $\bm f$.

	Finally, $\curl \bm \xi\cdot \bm n_{\bm f} = \rot_{\bm f} (E_f\bm \xi) = 0$, together with $E_{\bm f} \curl \bm \xi = 0$ it then concludes that $\curl \bm \xi = 0$.

\end{proof}

\begin{theorem}\label{thm:unisolvent} The degrees of freedom are unisolvent with respect to the shape function space, and the resulting finite element space is $\mathbf{N}^{\curl}_{k-1}$.
\end{theorem}
\begin{proof}
Note that the number of degrees of freedom is the same as that in \cite{2015Neilan}, since $\dim B_{\bm f, k - 7}^{BDM} = \dim RT_{k-9}$. It suffices to show that if $\bm \xi \in [P_{k-1}(\bm K)]^3$ vanishes for the above all the degrees of freedom, then $\bm \xi = 0$. It follows from \Cref{lem:face} that $\bm \xi = \curl \bm \xi = 0$ in all the faces of $\bm K$, and therefore it follows from the degrees of freedom $f_{\bm K, i}^1(\cdot)$ such that $\bm \xi = 0$. \Cref{lem:edge,lem:face} imply the additional regularity of the resulting finite element space, which completes the proof. 
\end{proof}

\end{document}

%% file: omega_x.tikz
\begin{tikzpicture}[scale = .6]
		\node (0) at (0, 0) {~~$x$};
		\node  (1) at (0.75, 1) {};
		\node  (2) at (-1, 1) {};
		\node (3) at (-1, -0.5) {};
		\node  (4) at (0.25, -1.5) {};
		\node  (5) at (1.5, -0.75) {};
		\node  (6) at (1.25, 2.25) {};
		\node  (7) at (0, 2.25) {};
		\node  (8) at (-1.5, 2) {};
		\node  (9) at (-2.75, 0.5) {};
		\node  (10) at (-2, -1) {};
		\node  (11) at (-1.25, -2.25) {};
		\node  (12) at (0.25, -3) {};
		\node  (13) at (2, -2.75) {};
		\node  (14) at (3, -2) {};
		\node  (15) at (3.25, -1) {};
		\node  (16) at (2.5, 1.25) {};
		\node  (17) at (2.75, 5.25) {};
		\fill[green]  (1.center) to (2.center) to (3.center) to (4.center) to (5.center) to  (1.center);
		\draw (0.center) to (2.center);
		\draw (2.center) to (3.center);
		\draw (3.center) to (0.center);
		\draw (0.center) to (4.center);
		\draw (3.center) to (4.center);
		\draw (4.center) to (5.center);
		\draw (5.center) to (0.center);
		\draw (0.center) to (1.center);
		\draw (1.center) to (5.center);
		\draw (2.center) to (1.center);
		\draw (1.center) to (6.center);
		\draw (7.center) to (6.center);
		\draw (7.center) to (1.center);
		\draw (2.center) to (7.center);
		\draw (2.center) to (8.center);
		\draw (8.center) to (7.center);
		\draw (6.center) to (16.center);
		\draw (1.center) to (16.center);
		\draw (16.center) to (5.center);
		\draw (5.center) to (15.center);
		\draw (15.center) to (16.center);
		\draw (15.center) to (13.center);
		\draw (13.center) to (5.center);
		\draw (12.center) to (13.center);
		\draw (12.center) to (4.center);
		\draw (5.center) to (12.center);
		\draw (4.center) to (11.center);
		\draw (11.center) to (12.center);
		\draw (10.center) to (11.center);
		\draw (10.center) to (3.center);
		\draw (3.center) to (11.center);
		\draw (3.center) to (9.center);
		\draw (9.center) to (10.center);
		\draw (9.center) to (2.center);
		\draw (8.center) to (9.center);
\end{tikzpicture}

%% file: omega_e.tikz
\begin{tikzpicture}[scale = .6]
		\node  (0) at (0, 1) {};
		\node  (1) at (0, -1.25) {};
		\node  (2) at (-1, 0) {};
		\node  (3) at (1.5, 0) {};
		\node  (4) at (1.5, 2) {};
		\node  (5) at (3, 0.5) {};
		\node  (6) at (2.75, -1.5) {};
		\node  (7) at (1.25, -2.5) {};
		\node  (8) at (-1.25, -2.5) {};
		\node  (9) at (-2.5, -0.75) {};
		\node  (10) at (-2, 1.75) {};
		\node  (11) at (-0.75, 3) {};
		\fill[green]  (0.center) to (2.center) to (1.center) to (3.center) to (0.center);
		\draw (0.center) to (1.center);
		\draw (1.center) to (2.center);
		\draw (2.center) to (0.center);
		\draw (0.center) to (3.center);
		\draw (3.center) to (1.center);
		\draw (11.center) to (0.center);
		\draw (0.center) to (4.center);
		\draw (4.center) to (3.center);
		\draw (3.center) to (5.center);
		\draw (4.center) to (5.center);
		\draw (5.center) to (6.center);
		\draw (6.center) to (7.center);
		\draw (3.center) to (6.center);
		\draw (1.center) to (7.center);
		\draw (1.center) to (6.center);
		\draw (1.center) to (8.center);
		\draw (8.center) to (7.center);
		\draw (2.center) to (8.center);
		\draw (9.center) to (8.center);
		\draw (9.center) to (2.center);
		\draw (11.center) to (2.center);
		\draw (11.center) to (10.center);
		\draw (10.center) to (2.center);
		\draw (11.center) to (4.center);
		\draw (10.center) to (9.center);
\end{tikzpicture}

%% file: omega_f.tikz
\begin{tikzpicture}[scale=.4]
		\node  (0) at (0, 2) {};
		\node  (1) at (-1.5, -1.25) {};
		\node  (2) at (1.75, -1.25) {};
		\node  (3) at (2.25, 2.75) {};
		\node  (4) at (3.75, 0.75) {};
		\node  (5) at (3.75, -2.5) {};
		\node  (6) at (1.5, -3.75) {};
		\node  (7) at (-1.25, -3.5) {};
		\node  (8) at (-3.75, -0.75) {};
		\node  (9) at (-3, 1.5) {};
		\node  (10) at (-1, 3.5) {};
		\fill[green] (0.center) to (1.center) to (2.center) to (0.center);
		\draw (0.center) to (1.center);
		\draw (1.center) to (2.center);
		\draw (2.center) to (0.center);
		\draw (0.center) to (3.center);
		\draw (3.center) to (4.center);
		\draw (0.center) to (4.center);
		\draw (4.center) to (2.center);
		\draw (2.center) to (5.center);
		\draw (4.center) to (5.center);
		\draw (5.center) to (6.center);
		\draw (2.center) to (6.center);
		\draw (6.center) to (7.center);
		\draw (7.center) to (2.center);
		\draw (1.center) to (7.center);
		\draw (7.center) to (8.center);
		\draw (8.center) to (1.center);
		\draw (1.center) to (9.center);
		\draw (9.center) to (8.center);
		\draw (9.center) to (0.center);
		\draw (10.center) to (0.center);
		\draw (3.center) to (10.center);
		\draw (10.center) to (9.center);
\end{tikzpicture}

%% file: interp_dR.bbl
\begin{thebibliography}{10}

\bibitem{2021ArnoldGuzman}
Douglas Arnold and Johnny Guzm\'{a}n.
\newblock Local {$L^2$}-bounded commuting projections in {FEEC}.
\newblock {\em ESAIM Math. Model. Numer. Anal.}, 55(5):2169--2184, 2021.

\bibitem{2018Arnold}
Douglas~N. Arnold.
\newblock {\em Finite element exterior calculus}, volume~93 of {\em CBMS-NSF
  Regional Conference Series in Applied Mathematics}.
\newblock Society for Industrial and Applied Mathematics (SIAM), Philadelphia,
  PA, 2018.

\bibitem{2006ArnoldFalkWinther}
Douglas~N. Arnold, Richard~S. Falk, and Ragnar Winther.
\newblock Finite element exterior calculus, homological techniques, and
  applications.
\newblock {\em Acta Numer.}, 15:1--155, 2006.

\bibitem{2010ArnoldFalkWinther}
Douglas~N. Arnold, Richard~S. Falk, and Ragnar Winther.
\newblock Finite element exterior calculus: from {H}odge theory to numerical
  stability.
\newblock {\em Bull. Amer. Math. Soc. (N.S.)}, 47(2):281--354, 2010.

\bibitem{2013BoffiBrezziFortin}
Daniele Boffi, Franco Brezzi, Michel Fortin, et~al.
\newblock {\em Mixed finite element methods and applications}, volume~44.
\newblock Springer, 2013.

\bibitem{1974Brezzi}
F.~Brezzi.
\newblock On the existence, uniqueness and approximation of saddle-point
  problems arising from {L}agrangian multipliers.
\newblock {\em Rev. Fran\c{c}aise Automat. Informat. Recherche
  Op\'{e}rationnelle S\'{e}r. Rouge}, 8({\rm R}-2):129--151, 1974.

\bibitem{1985BrezziDouglasMarini}
Franco Brezzi, Jim Douglas, Jr., and L.~D. Marini.
\newblock Two families of mixed finite elements for second order elliptic
  problems.
\newblock {\em Numer. Math.}, 47(2):217--235, 1985.

\bibitem{1991BrezziFortin}
Franco Brezzi and Michel Fortin.
\newblock {\em Mixed and hybrid finite element methods}, volume~15 of {\em
  Springer Series in Computational Mathematics}.
\newblock Springer-Verlag, New York, 1991.

\bibitem{2018ChristiansenHuHu}
Snorre~H. Christiansen, Jun Hu, and Kaibo Hu.
\newblock Nodal finite element de {R}ham complexes.
\newblock {\em Numer. Math.}, 139(2):411--446, 2018.

\bibitem{2008ChristiansenWinther}
Snorre~H. Christiansen and Ragnar Winther.
\newblock Smoothed projections in finite element exterior calculus.
\newblock {\em Math. Comp.}, 77(262):813--829, 2008.

\bibitem{2013FalkNeilan}
Richard~S Falk and Michael Neilan.
\newblock Stokes complexes and the construction of stable finite elements with
  pointwise mass conservation.
\newblock {\em SIAM Journal on Numerical Analysis}, 51(2):1308--1326, 2013.

\bibitem{2014FakWinther}
Richard~S. Falk and Ragnar Winther.
\newblock Local bounded cochain projections.
\newblock {\em Math. Comp.}, 83(290):2631--2656, 2014.

\bibitem{2015FalkWinther}
Richard~S. Falk and Ragnar Winther.
\newblock Double complexes and local cochain projections.
\newblock {\em Numer. Methods Partial Differential Equations}, 31(2):541--551,
  2015.

\bibitem{1980Nedelec}
J.-C. N\'{e}d\'{e}lec.
\newblock Mixed finite elements in {${\bf R}^{3}$}.
\newblock {\em Numer. Math.}, 35(3):315--341, 1980.

\bibitem{1986Nedelec}
J.-C. N\'{e}d\'{e}lec.
\newblock A new family of mixed finite elements in {${\bf R}^3$}.
\newblock {\em Numer. Math.}, 50(1):57--81, 1986.

\bibitem{2015Neilan}
Michael Neilan.
\newblock Discrete and conforming smooth de {R}ham complexes in three
  dimensions.
\newblock {\em Math. Comp.}, 84(295):2059--2081, 2015.

\bibitem{1977RaviartThomas}
P.-A. Raviart and J.~M. Thomas.
\newblock A mixed finite element method for 2nd order elliptic problems.
\newblock In {\em Mathematical aspects of finite element methods ({P}roc.
  {C}onf., {C}onsiglio {N}az. delle {R}icerche ({C}.{N}.{R}.), {R}ome, 1975)},
  Lecture Notes in Math., Vol. 606, pages 292--315. Springer, Berlin, 1977.

\bibitem{2005Schoberl}
Joachim Sch{\"o}berl.
\newblock A multilevel decomposition result in {H} (curl).
\newblock {\em Multigrid, Multilevel and Multiscale Methods, EMG}, 2005:243,
  2005.

\bibitem{2010Stenberg}
Rolf Stenberg.
\newblock A nonstandard mixed finite element family.
\newblock {\em Numer. Math.}, 115(1):131--139, 2010.

\end{thebibliography}
